\documentclass[reqno]{amsart}
\usepackage{mathtools}
\usepackage{amsrefs}
\usepackage[normalem]{ulem}
\usepackage{hyperref} 
\usepackage{color}
\usepackage{latexsym}
\usepackage{graphicx}
\usepackage{mathrsfs}
\usepackage{amssymb}
\usepackage{amsfonts}
\usepackage{amssymb,amsmath,amsthm}

\newcommand{\Sch}{Schr\"odinger }

\DeclareMathOperator{\Ai}{Ai}

\newcommand{\set}[1]{\{#1\}}
\newcommand{\ihbar}{{\frac{i}{h}}}

\newcommand{\ZN}{|Z_{\Phi_{\hbar, E}}|}
\newcommand{\inprod}[2]{\ensuremath{\left\langle#1,#2\right\rangle}}
\newcommand{\twiddle}[1]{\ensuremath{\widetilde{#1}}}
\newcommand{\W}{\ensuremath{\Omega}}
\newcommand{\gives}{\ensuremath{\rightarrow}}
\newcommand{\x}{\ensuremath{\times}}
\newcommand{\abs}[1]{\left\lvert #1 \right\rvert}
\newcommand{\norm}[1]{\left\lVert#1\right\rVert}
\newcommand{\lr}[1]{\ensuremath{\left(#1\right)}}
\newcommand{\dell}{\ensuremath{\partial}}

\renewcommand{\Re}{{\text Re}}

\newcommand{\lan}{\langle}
\newcommand{\ran}{\rangle}

\hypersetup{colorlinks = true, urlcolor = black}
\newcommand{\red}{\color{red}}
\newcommand{\cs}{$\clubsuit$}

\newcommand{\edit}[1]{ {\red \cs #1 \cs}}

\newcommand{\kk}{\left(\frac{k}{ 2\pi}\right)}
\newcommand{\Hb}{{\mathbb H}}
\newcommand{\la}{\langle}
\newcommand{\ra}{\rangle}
\newcommand{\bpp}{\begin{prop}}
\newcommand{\epp}{\end{prop}}
\renewcommand{\b}{\bar}

\newcommand{\z}{\text}
\renewcommand{\ss}{\subsection}

\DeclareMathOperator{\Vol}{Vol}

\DeclareMathOperator{\Erf}{Erf}

\renewcommand{\Re}{\text{Re}}
\renewcommand{\Im}{\text{Im}}

\newcommand{\Szego}{Szeg\"o }
\newcommand{\szego}{Szeg\"o }

\newcommand{\Sj}{Sj\"ostrand }
\newcommand{\kahler}{K\"ahler }
\newcommand{\Kahler}{K\"ahler }

\newcommand{\C}{\mathbb{C}}
\newcommand{\E}{\mathbb{E}}
\renewcommand{\H}{\mathbb{H}}
\newcommand{\CP}{\mathbb{CP}}

\newcommand{\R}{\mathbb{R}}

\newcommand{\Z}{\mathbb{Z}}

\renewcommand{\P}{\mathbb{P}}

\newcommand{\acal}{\mathcal{A}}
\newcommand{\bcal}{\mathcal{B}}
\newcommand{\ccal}{\mathcal{C}}
\newcommand{\dcal}{\mathcal{D}}

\newcommand{\fcal}{\mathcal{F}}

\newcommand{\hcal}{\mathcal{H}}

\newcommand{\jcal}{\mathcal{J}}
\newcommand{\kcal}{\mathcal{K}}
\newcommand{\lcal}{\mathcal{L}}

\newcommand{\ncal}{\mathcal{N}}
\newcommand{\ocal}{\mathcal{O}}

\newcommand{\scal}{\mathcal{S}}

\newcommand{\ucal}{\mathcal{U}}

\newcommand{\wcal}{\mathcal{W}}

\newcommand{\wb}{\overline}

\newcommand{\hPi}{\h \Pi}

\newcommand{\h}{\hat}

\newcommand{\bma}{\begin{pmatrix}}
\newcommand{\ema}{\end{pmatrix}}
\newcommand{\baa}{\begin{align*}}
\newcommand{\eaa}{\end{align*}}
\newcommand{\bea}{\begin{eqnarray*} }
\newcommand{\eea}{\end{eqnarray*} }
\newcommand{\bee}{\begin{eqnarray} }
\newcommand{\eee}{\end{eqnarray} }
\newcommand{\be}{\begin{equation} }
\newcommand{\ee}{\end{equation} }
\newcommand{\bp}{\begin{prop}}
\newcommand{\ep}{\end{prop}}
\newcommand{\bt}{\begin{theo}}
\newcommand{\et}{\end{theo}}
\newcommand{\bpf}{\begin{proof}}
\newcommand{\epf}{\end{proof}}
\newcommand{\bl}{\begin{lem}}
\newcommand{\el}{\end{lem}}
\newcommand{\bc}{\begin{cor}}
\newcommand{\ec}{\end{cor}}
\newcommand{\bd}{\begin{defn}}
\newcommand{\ed}{\end{defn}}
\newcommand{\bcs}{\begin{cases}}
\newcommand{\ecs}{\end{cases}}
\newcommand{\bex}{\begin{example}}
\newcommand{\eex}{\end{example}}
\newcommand{\brem}{\begin{rem}}
\newcommand{\erem}{\end{rem}}
\newcommand{\bnum}{\begin{enumerate}}
\newcommand{\enum}{\end{enumerate}}

\newcommand{\pa}{\partial}
\newcommand{\ot}{\otimes}
\newcommand{\half}{\frac{1}{2}}

\newcommand{\ddbar}{\partial\dbar}
\newcommand{\RM}{\backslash}

\def\Xint#1{\mathchoice
{\XXint\displaystyle\textstyle{#1}}%
{\XXint\textstyle\scriptstyle{#1}}%
{\XXint\scriptstyle\scriptscriptstyle{#1}}%
{\XXint\scriptscriptstyle\scriptscriptstyle{#1}}%
\!\int}
\def\XXint#1#2#3{{\setbox0=\hbox{$#1{#2#3}{\int}$ }
\vcenter{\hbox{$#2#3$ }}\kern-.6\wd0}}

\def\dashint{\Xint-}
\newtheorem{remark}{{\sc Remark}}

\newtheorem*{maintheo}{{\sc {\bf Main Theorem}}}
\newtheorem{theointro}{{\sc {\bf Theorem}}}

\newtheorem{theo}{{\sc Theorem}}[section]
\newtheorem{cor}[theo]{{\sc Corollary}}
\newtheorem{defin}[theo]{{\sc Definition}}

\newtheorem{lem}[theo]{{\sc Lemma}}
\newtheorem{prop}[theo]{{\sc Proposition}}
\newtheorem{defn}[theo]{{\sc Definition}}
\newtheorem{rem}[theo]{{\sc rem}}
\newtheorem{example}[theo]{{\sc Example}}

\newcommand{\Hess}{{\operatorname {Hess}}}

\title{Interfaces in spectral asymptotics and nodal sets}

\author{Steve Zelditch}
\address{Department of Mathematics, Northwestern  University, Evanston, IL 60208, USA}
\email{zelditch@math.northwestern.edu, \,pzhou.math@gmail.com}
\thanks{Research partially supported by NSF grant  DMS-1810747}

\date{\today}

\begin{document}

\begin{abstract} This is a survey of results obtained jointly with Boris Hanin and Peng Zhou on interfaces in spectral asymptotics, both for \Sch operators on $L^2(\R^d)$
and for Toeplitz Hamiltonians acting on holomorphic sections of ample line bundles $L \to M$ over \kahler manifolds $(M, \omega)$.
By an interface is meant a hypersurface, either in physical space $\R^d$ or in phase space, separating an allowed region where
spectral asymptotics are standard and a forbidden region where they are non-standard. The main question is to give the detailed
transition between the two types of asymptotics across the hypersurface (i.e. interface). In the real \Sch setting, the asymptotics
are of Airy type; in the \kahler setting they are of Erf (Gaussian error function) type. 

A principal purpose of this survey is to compare the results in the two settings. Each is apparently universal in its setting. This
is now established for Toeplitz operators,  but in the \Sch setting it is only established for the simplest model operator, the isotropic
harmonic oscillator. It is explained that the latter result is most comparable to the behavior of the canonical degree operator on 
the Bargmann-Fock space of a line bundle, a new construction introduced in these notes.

\end{abstract}

\maketitle


\section{Introduction}

This is a mainly expository article on interfaces in spectral asymptotics.
Interfaces  are studied in many fields of mathematics and physics but seem to be  a  novel area of spectral
asymptotics. Spectral asymptotics refers to the behavior of spectral projections and nodal sets for a  quantum Hamiltonian $\hat{H}_{\hbar}$, which might be a \Sch operator on $L^2(\R^d)$ or on a Riemannian manifold
$(M, g)$, with or without boundary, or a Toeplitz Hamiltonian acting on holomorphic sections $H^0(M, L^k)$ of line bundles over a \kahler manifold.  Interface asymptotics refers
to the change in behavior of the spectral projections or nodal sets as a hypersurface is crossed, either in physical space (configuration space) or in phase space. Interfaces  exist in diverse settings and indeed the purpose of this article is to compare interface behavior in different settings and to consider possible future settings that have yet to be explored.

What is meant by an `interface' in the sense of this article? The general idea
is that there is a hypersurface in the  phase
space separating two regions in which  the asymptotic behavior of a spectral projections kernel has different types of behavior: In the first,  that we will term the  `allowed' region, the asymptotics are constant and, after normalization, equal $1$, so that one has a  plateau over the region; in the second `forbidden' region the asymptotics are rapidly decaying, so that one
has a rather flat $0$ region. The interface is the shape of the graph of the spectral kernel connecting $1$ and $0$  in a thin region separating the allowed and forbidden region.  One expects that when scaled properly, the limit shape is universal. More precisely, universality
holds in each type of model (e.g. \Sch or \kahler) but is model-dependent: one expects `Airy interfaces' in the \Sch setting and
Erf interfaces in the \kahler setting. The separation into  different regions for the spectral projections kernel often coincides with the separation of other spectral behavior, such as nodal sets of the eigenfunctions.

The terminology (classically) `allowed' and (classically) `forbidden' is standard in quantum mechanics 
for regions inside, resp. outside, of an energy surface in phase space, or more commonly, the projection of these regions to configuration space. This will indeed be the meaning of `interface' for most of this article. We will describe results of B. Hanin, P. Zhou and the author \cite{HZZ15,HZZ16} on the different behavior of nodal
sets of \Sch eigenfunctions in allowed resp. forbidden regions for the simplest \Sch Hamiltonian $\hat{H}_{\hbar}$, namely the isotropic Harmonic oscillator on $\R^d$. We then consider phase space interfaces of Wigner distributions for the same model, following \cite{HZ19,HZ19b}. We then turn to phase space interfaces in the \kahler (complex holomorphic) setting, and discuss results of   Pokorny-Singer \cite{PS},
Ross-Singer \cite{RS}, P.  Zhou and the author \cite{ZZ16,ZZ17}
on interfaces for {\it partial Bergman kernel} asymptotics. In Section \ref{BFLB} we explain that the exact analogue of the results
on Wigner distributions for the isotropic harmonic oscillator in the complex setting is a series of results on interfaces for disc bundles
in the Bargmann-Fock space of a line bundle. This Bargmann-Fock space and the interface results constitute the new results of 
the article.

Roughly speaking, interfaces in spectral asymptotics involve two types of localization: (i) spectral, i.e. quantum, 
localization where the eigenvalues are constrained to lie in an interval $I$, (ii) classical, i.e. phase space, localization where
a phase space point is constrained to lie in an open set $U $ of phase space. It has long been understood that spectral localization
$E_j(\hbar) \in I$ implies phase space localization in the sense that quantum objects decay in the complement of the allowed region $H^{-1}(I)$. But  the study of interfaces is devoted to the precise behavior of quantum objects as one crosses the interface between allowed and forbidden regions, and more generally, considers all possible combinations of spectral localization $E_j(\hbar) \in I$ and phase space localization
$\zeta \in  U$, where $U$ may have any position relative to $H^{-1}(I)$.

Often, the interface corresponds to a sharp cutoff in a spectral parameter and signals something discontinuous. In fact, the earliest studies
of interface asymptotics are classical analysis studies of Bernstein polynomials of discontinuous functions with jump discontinuities \cite{Ch,L,Lev,Mir,O}. These studies were intended to be analogues of Gibbs phenomena for Fourier series of discontinuous functions, which have been
generalized to wave equations on Riemannian manifolds in \cite{PT97}.

In this article we review the following results on interface asymptotics:

\begin{itemize}
\item Interface behavior for spectral projections and for  nodal sets of random  eigenfunctions of  energy
$E_N(\hbar) = \hbar (N + \frac{d}{2}) = E$ of the
isotropic harmonic oscillator on $\R^d$ across the {\it caustic set} in physical space, where
the potential $V(x) = |x|^2/2 = E$. \bigskip

\item Interface behavior for Wigner distributions of the same eigenspace
projections, and more generally for various types of Wigner-Weyl sums  
across an energy surface in phase space; \bigskip

\item Interface behavior for the holomorphic analogues of such Wigner distributions, namely for {\it partial Bergman kernels} for general Berezin-Toeplitz Hamiltonians on general \kahler phase spaces.
\bigskip

\item Interface results for partial Bergman kernels corresponding to the canonical $S^1$ action on the total space $L^*$ of the
dual line bundle of an ample line bundle $L \to M$ over a \kahler manifold.

\end{itemize}\bigskip

In the case of \Sch operators, the results are only proved in the special
case of the isotropic harmonic oscillator. It is plausible that some of the
results should be universal among \Sch operators, but at the present time the generalizations have not been formulated or proved. See Section \ref{FURTHER} for further problems.  Among other gaps in the theory, Wigner distributions per se are only defined when the Riemannian manifold is $\R^d$ and are closely connected to the representation theory 
of the Heisenberg and metaplectic groups. Wigner distributions of eigenfunctions  are special types of ``microlocal lifts'' of eigenfunctions;  there is no generally accepted canonical microlocal lift on a general Riemannian manifold. Despite the restrictive setting, Wigner distributions  are important in 
mathematical physics, in particular in quantum optics. The results in the complex holomorphic (\kahler) setting are much more complete, due to the
fact that the theory of Bergman kernels is technically simpler and more complete than the corresponding theory of Wigner distributions for \Sch operators. The results are proved
for any Toeplitz Hamiltonian on any projective \kahler manifold. In fact, the exact analogue of the Wigner result is proved in Section
\ref{BFLB}, where a new construction is introduced in this article: the Bargmann-Fock space of a holomorphic line bundle. It is a Gaussian
space of holomorphic functions on the total space $L^*$ of the dual of a holomorphic Hermitian line bundle $L \to M$ over a \kahler manifold. 
This total space carries a natural $S^1$ action \footnote{$S^1$ always denotes the unit circle} and this $S^1$ action plays the role
of the propagator of the isotropic Harmonic Oscillator. Thus, the interfaces are the boundaries of the co-disc bundles $D^*_E \subset L^*$
of different energy levels (i.e. radii).  The interface results in Section \ref{BFLB} are a `new result' of this article, but the proofs are similar to,
and 
simpler than,  those in \cite{ZZ17,ZZ18} .

This survey is organized as follows: 

\begin{enumerate}
\item In Section \ref{SETUP},  we review the basic linear models: the Harmonic oscillator in the \Sch representation 
on  $L^2(\R^d)$ and in the Bargmann-Fock (holomorphic) representation on entire holomorphic functions on $\C^d$. We also present a list of analogies between the real \Sch setting and the complex holomorphic quantization.  Section \ref{BF} is devoted to the Bergman kernel on  Bargmann-Fock  space,
and the Bargmann-Fock representations of the Heisenberg and Symplectic groups on Bargmann-Fock space. \bigskip

\item In Section \ref{SCH}, we review the  interface results in physical space for spectral projections for the isotropic Harmonic Oscillator. 
These imply interface results for nodal sets of random eigenfunctions in a fixed eigenspace. \bigskip

\item In Section \ref{WIGNER}, we  change the setting to phase space $T^*\R^d$ and review the  interface results in physical space for Wigner distributions of spectral projections for the isotropic Harmonic Oscillator. 
 \bigskip

\item In Section \ref{pBK}, we switch to the complex holomorphic setting and review interface results for partial Bergman kernels
on general compact \kahler manifolds. \bigskip

\item In Section \ref{BFINTER} we specialize to the isotropic harmonic oscillator on the  standard Bargmann-Fock space and describe
its interfaces;  \bigskip

\item In Section \ref{BFLB} we  introduce a new model: the  Bargmann-Fock space of a holomorphic line bundle. We then
consider interfaces with respect to a natural $S^1$ action on this space,  generalalizing the previous result on the Bargmann-Fock isotropic Harmonic oscillator.
 \bigskip
 
 \item In Section \ref{FURTHER} we list some further problems on interfaces. \bigskip
 
 \item In Section \ref{BACKGROUND} we give some background to the holomorphic setting. 

\end{enumerate}

\subsection{Results surveyed in this article}

The articles surveyed in this article are the following:

\section{\label{SETUP} The basic linear models} 

As mentioned above, our aim in this survey is not only to describe interface results in various settings but to compare the
results in the real \Sch setting and the complex holomorphic Bargmann-Fock or Berezin-Toeplitz setting. The real setting
is self-explanatory to mathematical physicists but the complex holomorphic setting is probably less familiar. In this section,
we give some background on the basic linear models (isotropic Harmonic Oscillator in both settings) to make the relations
between the real and complex settings more familiar. We then give a list of analogies between the two settings.  In addition,
we present a list of open problems on interfaces to amplify the scope of spectral interface problems. It would be laborious
to present all of the background for the geometric setting before getting to the main results and phenomena, so we have put
that background into an Appendix Section \ref{BACKGROUND}.

A preliminary remark: Since the early days of quantum mechanics, it was understood that there are many equivalent representations
(or `pictures') of quantum mechanics. In the case of $\R^d$ they correspond to different but unitarily equivalent representations
of the Heisenberg and metaplectic groups (see \cite{F} for background). The most common are the \Sch representation on $L^2(\R^d)$ and the Bargmann-Fock  representation on $H^2(\C^d, e^{-|Z|^2} dL(Z))$, the Bargmann-Fock space of entire holomorphic functions on $\C^d$
which are in $L^2$ with respect to Gaussian measure; here $dL$ is Lebesgue measure.  One refers to $\R^d$ as `configuration space' or `physical space' and to $T^* \R^d$ as phase space. Of course, $T^* \R^d \simeq \C^d$, so that Bargmann-Fock space employs  a complex structure
on phase space. A natural unitary intertwining operator
is the Bargmann transform (see \eqref{BT} below). We refer to \cite{F} and to  \cite{HSj16} for background on Bargmann-Fock space and metaplectic operators.

The first item is to give background on the isotropic Harmonic oscillator in both the \Sch representation and the Bargmann-Fock
representation.

\subsection{\Sch representation of the isotropic Harmonic oscillator}

The \Sch representation of quantum mechanics is too familiar to need a detailed review here. 
The  isotropic Harmonic Oscillator on $L^2(\R^d,dx)$. is the operator,
\begin{equation} \label{Hh}
\widehat{H}_{\hbar} =  \sum_{j = 1}^d \left(- \frac{\hbar^2}{2}   \frac{\partial^2 }{\partial
x_j^2} + \frac{x_j^2}{2} \right).
\end{equation}
It has a discerete  spectrum of eigenvalues 
\begin{equation} \label{ENh} E_N(\hbar)=\hbar\lr{N+d/2},\qquad  \;\;  (N = 0, 1, 2, \dots)\end{equation} 
with multiplicities given by the composition function $p(N,d)$ of $N$ and $d$ (i.e. the number of ways to write $N$ as an ordered sum of $d$ non-negative integers). That is,
the eigenspaces 
\begin{equation} \label{VhE} V_{\hbar, E_N(\hbar)}: = \{\psi \in L^2(\R^d): \widehat{H}_{\hbar} \psi = 
E_N(\hbar) \psi \}, \end{equation}
have dimensions given by \begin{equation} \label{dimV} \dim V_{\hbar_N,E}=  p(N,d)  =  \frac{1}{(d-1)!} N^{d-1}(1 + O(N^{-1})).\end{equation}
When $E_N(\hbar) = E$ we also write
\begin{equation} \label{hNEDEF}\hbar = \hbar_N(E):=\frac{E}{N+\frac{d}{2}},\end{equation}

 An orthonormal basis of its eigenfunctions is given by the product Hermite functions,
\begin{equation}
\phi_{\alpha,h}(x)=h^{-d/4}p_{\alpha}\lr{x\cdot
h^{-1/2}}e^{-x^2/2h},\label{E:Scaling Relation}
\end{equation}
where $\alpha=\lr{\alpha_1,\ldots, \alpha_d}\geq (0,\ldots,0)$ is a
$d-$dimensional multi-index and
$p_{\alpha}(x)$ is the product $\prod_{j=1}^d p_{\alpha_j}(x_j)$ of the
hermite polynomials $p_k$ (of degree $k$)
in one variable. 

The eigenspace projections are
the orthogonal projections
\begin{equation} \label{PiDEF} 
\Pi_{\hbar, E_N(\hbar)}: L^2(\R^d) \to V_{\hbar,E_N(\hbar)}.
\end{equation}
When $E_N(\hbar) = E$ \eqref{hNEDEF}, their Schwartz kernels are given in terms of an orthonormal basis by,
\begin{equation} \label{COV}
\Pi_{h_N, E} (x, y) =
\sum_{\abs{\alpha}=N} \phi_{\alpha,h_N}(x)
\phi_{\alpha,h_N}(y).   \end{equation}

The high multiplicities are due to the $U(d)$-invariance of the isotropic Harmonic
Oscillator.   
Due to extreme degeneracy of the spectrum of \eqref{Hh} when $d\geq 2$, the eigenspace projections 
have very special  semi-classical asymptotic properties, reflecting the periodicity of the classical Hamiltonian flow and of the Schr\"odinger propagator $\exp[- \frac{it}{\hbar} \widehat{H}_{\hbar}]$. 
 In particular, the eigenspace projections \eqref{PiDEF} are semi-classical
Fourier integral operators (see e.g. \cite{GU12,GUW,HZ19}. We exploit this very rare property to obtain scaling asymptotics across the caustic. This explains why 
the results to date are only available for isotropic oscillators.  For general Harmonic Oscillators with incommensurate frequencies the eigenvalues have multiplicity one and the eigenspace projections are of a very different type.  For  general \Sch operator,
one would need to take appropriate combinations of eigenspace projections with eigenvalues in an interval.

As with any 1-parameter  metaplectic unitary group \cite{F,HSj16}, one has an explicit
 Mehler formula for the Schwartz kernel $U_h(t,
x,y)$  of the propagator,
 $e^{-\ihbar t H_h}.$ The Mehler formula \cite{F} reads
\begin{equation}
 U_h(t, x,y) =e^{-\ihbar t H_h}(x,y)= \frac{1}{(2\pi i h \sin t)^{d/2}}
 \exp\left( \frac{i}{h}\left(
 \frac{\abs{x}^2 + \abs{y}^2}{2} \frac{\cos t}{\sin t} - \frac{x\cdot
 y}{\sin t} \right) \right),
 \label{E:Mehler}
\end{equation}
where $t \in \R$ and $x,y \in \R^d$. The right hand side is singular at
$t=0.$ It is well-defined as a
distribution, however, with $t$ understood as $t-i0$. Indeed, since $H_h$
has a positive spectrum the propagator
$U_h$ is holomorphic in
the lower half-plane and $U_h(t, x, y)$ is the boundary value of a
holomorphic function in $\{\Im t < 0\}$.

One may express the $N$th spectral projection as a Fourier coefficient of the propagator. It is somewhat simpler
to work with the number  operator $\ncal$, i.e. the \Sch operator 
with the same eigenfunctions as $H_h$ and eigenvalues $h |\alpha|$.  If we replace $U_h(t)$ by $e^{- \frac{i t}{h} \ncal}$ then the
spectral projections $\Pi_{h, E}$ are
simply the Fourier coefficients of $e^{- \frac{i t}{h} \ncal}$. In \cite{HZZ15, HZZ16} it is shown that \begin{align}
\label{E:Projector Integral Forbidden}
\Pi_{h_N, E}(x,y)&=\int_{-\pi}^{\pi} U_h(t-i\epsilon,x,y) e^{\ihbar
(t-i\epsilon) E} \frac{dt}{2\pi}.
\end{align}
 The integral is independent of $\epsilon$. Combining \eqref{E:Projector Integral Forbidden} with the  Mehler formula
 \eqref{E:Mehler}, one has an explicit
 integral representation of \eqref{COV}.

 \subsubsection{Wigner distributions}

For any Schwartz kernel $K_{\hbar} \in L^2(\R^d \times \R^d)$ one may define the Wigner distribution of $K_{\hbar}$ by
 \begin{equation}
 W_{K, \hbar}(x, \xi): =  \int_{\R^d} K_{\hbar} \left( x+\frac{v}{2}, x-\frac{v}{2} \right) e^{-\frac{i}{\hbar} v \xi} \frac{dv}{(2\pi h)^d}, \label{E:WignerK}
\end{equation}  The  map from $K_\hbar \to W_{K, \hbar}$ defines the unitary `Wigner transform',  $$\wcal_{\hbar}: L^2(\R^d \times \R^d) \to L^2(T^*\R^d).$$  
The inverse Wigner transform is given by  (see page 79 of \cite{F})
\begin{equation} \label{INVERSEWIG}
f \otimes g^* (x, y)  = \int W_{f, g}(\frac{x + y}{2}, \xi) e^{i \langle x - y, \xi \rangle} d\xi. \end{equation}
Here, $W_{f, g} := W_{f \otimes g^*}$ is the Wigner transform of the rank one operator $f \otimes g^*$.

The unitary group $U(d)$ acts on $L^2(\R^d \times \R^d)$ by conjugation,$U(g) \cdot K = g K g^*$.
where we identify $K(x,y) \in L^2(\R^d \times \R^d)$ with the associated Hilbert-Schmidt operator.  Metaplectic covariance implies that,
$$\wcal_{\hbar} U(g) = T_g \wcal_{\hbar}. $$

\begin{defn} \label{WIGNERPROJDEF} The Wigner distributions $W_{\hbar, E_N(\hbar)}(x, p) \in L^2(T^*\R^d)$ of
the  eigenspace  projections $\Pi_{\hbar, E_N(\hbar)}$ are defined by,
\begin{equation}
 W_{\hbar,E_N(\hbar)}(x, \xi) = \int_{\R^d} \Pi_{\hbar, E_N(\hbar)} \left( x+\frac{v}{2}, x-\frac{v}{2} \right) e^{-\frac{i}{\hbar} v \cdot \xi} \frac{dv}{(2\pi h)^d} \label{E:Wignera}.
\end{equation}  
\end{defn}

\noindent When $E_N(\hbar)  = E$, 
 the Wigner distribution $W_{\hbar, E_N(\hbar)}$ of a single eigenspace projection \eqref{E:Wignera} is the `quantization' of the energy surface of energy
$E$ and should therefore be localized at the classical energy level $H(x, \xi)  = E$, where  $H(x, \xi) = \half \sum_{j=1}^d (\xi_j^2 +  x_j^2) $. We denote the (energy) level sets by, \begin{equation} \label{SIGMAEDEF}
 \Sigma_E  =\{(x, \xi) \in T^*\R^d: H(x, \xi): = \half(||x||^2 + ||\xi||^2) = E\}. \end{equation}
The Hamiltonian flow of $H$ is $2 \pi $ periodic, and its orbits form the
complex projective space $\CP^{d-1} \simeq \Sigma_E /\sim$ where $\sim$ is the equivalence relation of belonging to the same Hamilton orbit. Due to this periodicity, the projections \eqref{PiDEF} are semi-classical Fourier integral operators (see \cite{GU12, GUW, HZZ15}). This is also true for the Wigner distributions \eqref{E:Wignera}.  Their properties are basically unique to the isotropic oscillator \eqref{Hh}. These properties are visible in Figure \ref{fig-Wigner-eigenspace-2} depicting the graph of $W_{\hbar, 1/2}$.

\subsubsection{Weyl pseudo-differential operators, metaplectic covariance}

A semi-classical Weyl pseudo-differential operator is defined by the formula, 
$$Op_h^w(a) u(x) = \int_{\R^d} \int_{\R^d} a_{\hbar}(\half(x + y), \xi) e^{\frac{i}{\hbar} \langle x - y, \xi \rangle } u(y) dy d \xi. $$ 
See \cite{F, Zw} for background.
By using the identity
$$\langle Op^w(a) f, f \rangle = \int_{T^*\R^d} a(x, \xi) W_{f, f}(x, \xi) dx d\xi, $$ of \cite[Proposition 2.5]{F} for  orthonormal basis elements $f = \phi_{\alpha, \hbar_N}$ of $V_{\hbar, E_N(\hbar)}$ and summing over $\alpha$, one obtains the (well-known) identity, 
\begin{equation} \label{TRACEP} \mathrm{Tr} \; Op_h^w(a) \Pi_{\hbar, E_N(\hbar)} = \int_{T^* \R^d} a(x, \xi) W_{\hbar, E_N(\hbar)}(x, \xi) dx d\xi.  \end{equation}
This formula is one of the key properties of Wigner distributions and Weyl quantization.

The  Wigner transform \eqref{WIGNERDEF1} taking kernels to Wigner functions  is therefore an isometry from Hilbert-Schmidt kernels $K(x,y)$ on $\R^d \times \R^d$ to their Wigner distributions
on $T^*\R^d$ \cite{F}.  From \eqref{TRACEP} and this isometry, it is straightforward to check that,
\begin{equation} \label{INTEGRALS}\left\{  \begin{array}{ll}(i) &  \int_{T^* \R^d} W_{\hbar, E_N(\hbar)}(x, \xi) dx d \xi = \mathrm{Tr} \Pi_{\hbar, E_N(\hbar)} = \dim V_{\hbar, E_N(\hbar)} = \binom{N+d-1}{d-1} \\ &  \\(ii) & 
\int_{T^* \R^d}\left| W_{\hbar, E_N(\hbar)}(x, \xi) \right|^2 dx d \xi = \mathrm{Tr} \Pi^2_{\hbar, E_N(\hbar)} = \dim V_{\hbar, E_N(\hbar)}=\binom{N+d-1}{d-1}\\ & \\ (iii)& \int_{T^* \R^d} W_{\hbar, E_N(\hbar)}(x, \xi) \overline{W_{\hbar, E_M(\hbar)}(x, \xi)}dx d \xi = \mathrm{Tr} \Pi_{\hbar, E_N(\hbar)}  \Pi_{\hbar, E_M(\hbar)}= 0, \; \mathrm{for}\; M \not=N. \end{array} \right.,\end{equation} 
In these equations, $N=\frac{E}{\hbar}-\frac{d}{2},$ and $\binom{N+d-1}{d-1}$ is the composition function of $(N,d)$ (i.e. the number of ways to write $N$ as an ordered us of $d$ non-negative integers). Thus, the sequence,
$$\{ \frac{1}{\sqrt{\dim V_{\hbar, E_N(\hbar)}}} W_{\hbar, E_N(\hbar)}\}_{N=1}^{\infty}\subset L^2(\R^{2n}) $$
is orthonormal.

In comparing \eqref{TRACEP}, \eqref{INTEGRALS}(i)-(ii) one should keep
in mind that $W_{\hbar, E_N(\hbar)}$ is rapidly oscillating in $\{H \leq E\}$
with slowly decaying tails in the interior of $\{H \leq E\}$, with a large `bump' near $\Sigma_E$  and with maximum
given by Proposition \ref{WIGBOUND}. Integrals (e.g. of $a \equiv 1$)
against $W_{\hbar, E_N(\hbar)}$ involve a lot of cancellation due to the
oscillations. The square integrals in (ii) enhance the `bump' and decrease the tails and of course are positive.

Another key property of Weyl quantization is its metaplectic covariance (see Section \ref{METASECT} for background).
Let $Sp(2d, \R) =  Sp(T^* \R^d, \sigma)$ denote the symplectic group and let $\mu(g)$ denote
the metaplectic representation of its double cover. Then,
$\mu(g) Op_h^w(a) \mu(g)  = Op_h^w(a \circ T_g), $ where $T_g: T^*\R^d \to T^*\R^d$ denotes translation by $g$.
See \cite{F} and  Section \ref{METASECT} for background.
In particular, $U \in U(d)$ acts on $L^2(T^* \R^d)$ by translation $T_U$ of functions, using the identification $T^*\R^d \simeq \C^d$ defined by the standard complex structure $J$. 
 $U(d) \subset Sp(2d, \R) $ is a subgroup of the symplectic group and the complete symbol $H(x, \xi)$ of \eqref{Hh} is $U(d)$ invariant, so by metaplectic covariance, 
 $\hat{H}_{\hbar}$ commutes with the metaplectic represenation of $U(d).$

\section{\label{BF} Bargmann-Fock space and the Toeplitz representation of the isotropic oscillator}

 Bargmann-Fock space of degree $k$ on $\C^{m+1}$ is defined by
\[ \hcal_k = \{ f(z) \text{ holomorphic function on $\C^{m+1}$}, \quad  \int_{\C^{m+1}} |f|^2 e^{-k|z|^2} dVol_{\C^{m+1}}  < \infty \}. \]
The volume form on $\C^{m+1}$ is $d \Vol_{\C^{m+1}} = \omega^{m+1}/(m+1)!$, and  $dL(z)$ denotes Lebesgue measure. We note
that $$  \int_{\C^{m+1}}  e^{-k|z|^2} dL(z) = \omega_{m+1} \int_0^{\infty} e^{- k \rho^2} \rho^{2m +1} d\rho =
\omega_{m+1} \int_0^{\infty} e^{- k x} x^{m} dx $$ and that $$ \int_0^{\infty} e^{- k x} x^{m} dx = k^{-(m+1)} \Gamma(m + 1) = m! k^{-(m+1)},  $$
where we use polar coordinates $(\theta, \rho)$ on $\C^{m+1}$ and where $\omega_{m+1} = |S^{2m +1}|$ is the surface
measure of the unit sphere in $\C^{m+1}$. We normalize the Gaussian measure to have mass $1$ and denote it by,
\begin{equation} \label{BFG} d \Gamma_{m+1, k} : = \frac{ k^{(m+1)} }{m! \omega_{m+1}}  e^{-k|z|^2} dL(z) . \end{equation}

Let us fix $k=1$.   An orthonormal basis is
given by the holomorphic monomials, $$\{\frac{z^{\alpha}}{\sqrt{\alpha!}}\} |_{\alpha \in {\mathbb N}^{m+1}},$$ where $\alpha = (\alpha_1, \dots, \alpha_{m+1}) $ is a lattice point
in the orthant $\alpha_j \in {\mathbb N}$ and $z^{\alpha} = \prod_{j=1}^{m+1} z_j^{\alpha_j}$, $\alpha! : = \prod_{j=1}^{m+1} \alpha_j !$. If we fix the
degree $|\alpha| = \sum_{j=1}^{m+1} \alpha_j$ we get the subspaces
$$\hcal_N = \rm{Span}\; \{z^{\alpha}: |\alpha| = N\}, $$  and one has the orthogonal decompositon, $$L^2_{\rm{hol}}(\C^{m+1}, d \Gamma_{m+1, k})  = \bigoplus_{N=0}^{\infty} \hcal_N. $$ Further, there is a canonical
isomorphism
$$\hcal_N \simeq H^0(\CP^m, \ocal(N))$$
between $\hcal_N$ and the space of holomorphic sections of the $N$th power of the standard line bundle $\ocal(1) \to \CP^m$ over projective space. 
The isomorphism is essentially by the lift
$$\hat{s}(z, \lambda)  = \lambda^{\otimes N}(s(z)) $$
 of a section $s \in H^0(M, \ocal(N))$  to the total  space
$\ocal(-1) \to \CP^m$ of the line bundle dual to $\ocal(1)$, as an equivariant
holomorphic function $\hat{s}$ of degree $N$. The lifted function vanishes at the zero section. If one blows down the zero section to a point, then $\ocal(-1) \simeq \C^{m+1}$ and the lifted sections are, again, homogeneous holomorphic polynomials of degree $N$. This implies that
Bargmann-Fock space is, as a vector space, isomorphic to  
$\bigoplus_{N=0}^{\infty}  H^0(\CP^m, \ocal(N)).$ The direct sum is endowed with 
 the Bargmann-Fock Hilbert space inner product and, up to a scalar,  this inner product on $\hcal_N$ is the same as the Fubini-Study inner product on 
 $ H^0(M, \ocal(N))$.

The degree $k$ {\it Bargmann-Fock Bergman kernel} is the orthogonal
projection from $L^2(\C^{m+1}, d \Gamma_{m+1, k})  \to \hcal_k$.  Its Schwartz kernel relative to Gaussian measure $ d \Gamma_{m+1, k}$ is given by
\[ \Pi_k(z,w) = \kk^{m +1} e^{k z \bar w}, \] i.e.
for  any function $f \in L^2(\C^{m+1},  d \Gamma_{m+1, k})$, its orthogonal projection to Bargmann-Fock space is given by 
\[ (\Pi_k f)(z) = \int_{\C^m} \Pi_k(z,w) f(w)  d \Gamma_{m+1, k}(dw)). \]

More generally, fix $(V, \omega)$ be a real $2m$ dimensional symplectic vector space. Let $J: V \to V$ be a $\omega$ compatible linear complex structure, that is $g(v,w): = \omega(v,Jw)$ is a positive-definite bilinear form and $\omega(v,w) = \omega(Jv, Jw)$. 
There exists a canonical identification of $V \cong \C^m$ up to $U(m)$ action, identifying $\omega$ and $J$. We denote the BF space for $(V, \omega, J)$ by $\hcal_{k,J}$.

To put Bargmann-Fock space into the general framework of holomorphic line bundles over \kahler manifolds, we let $M=\C^m$ with coordinate $z_i=x_i + \sqrt{-1} y_i$, $L \to M$ be the trivial line bundle, let  $L \cong \C^m \times \C$, and
let  $\omega = i \sum_i dz_i \wedge d\bar z_i$ be the \kahler form, whose   potential is $\varphi(z)=|z|^2: = \sum_i |z_i|^2$.
\subsection{Lifting to the Heisenberg group}

It is useful to lift holomorphic sections of line bundles to equivariant functions on the dual $L^*$ of the total space of the line bundle. Since
they are equivariant with respect to the natural $S^1$ action, one often restricts them 
to the unit circle bundle $X = X_h$ defined by a Hermitian metric $h$ on $L^*$.

 In  the case of Bargmann-Fock space, $X$ is the Heisenberg group $\Hb^m_{red} = \C^m \times S^1$, with group multiplication 
\[ (z, \theta) \circ (z', \theta') = (z+z', \theta+\theta' + \Im( z \bar z')). \] 
  The circle bundle $\pi: X \to M$ can be trivialized as $X \cong \C^m \times S^1$. The contact form on $X$ is
\[ \alpha = d\theta + (i/2) \sum_j(z_j d\bar z_j - \bar z_j dz_j). \]
The contact form $\alpha = d\theta + \frac{i}{2} \sum_j (z_j d\bar z_j - \bar z_j dz_j)$ on $\H^m_{red}$ is invariant under the left multiplication
\[ L_{(z_0, \theta_0)}: (z, \theta) \mapsto (z_0, \theta_0)  \circ (z, \theta) = (z + z_0, \theta + \theta_0 + \frac{ z_0 \bar z-  \bar z_0 z }{2i} ). \]
The volume form on $X=\C^m \times S^1$ is $d \Vol_{X} =(d\theta/2\pi) \wedge \omega^m/m!$. 

The action of the Heisenberg group is by {\it Heisenberg translations} on phase space. As seen in the next Lemma, Heisenberg
translations are Euclidean translations in the $\C^m$ component but also have a non-trivial change in the  angular component. 
The infinitesimal Heisenberg group action on $X$ can be identified with the contact vector field generated by a linear Hamiltonian function $H: \C^m \to \R$. 
\bl \cite[Section 3.2]{ZZ17}\label{flow-lin}
For any $\beta \in \C^m$, we define a linear Hamiltonian function on $\C^m$ by
\[ H(z) = z \bar \beta + \beta \bar z. \]
The Hamiltonian vector field on $\C^m$ is  
\[ \xi_H = - i \beta \pa_z + i \bar \beta \pa_{\bar z}, \]
and its contact lift is 
\[ \h \xi_H = - i \beta \pa_z + i \bar \beta \pa_{\bar z} -\half( z \bar \beta + \beta \bar z) \pa_\theta. \]
The time $t$ flow $\h g^t$ on $X$ is given by left multiplication 
\[ \h g^t(z,\theta) = (-i \beta t, 0) \circ (z,\theta) = (z-i\beta t, \theta - t\Re(\beta \bar z)). \]
\el

The lift of a holomorphic section of $L^k \to \C^m$ is the CR-holomorphic function defined by, 
\[ \h s(z, \theta) = e^{k(i \theta - \half |z|^2)} s(z). \]
Indeed, the horizontal lift of $\pa_{\bar z_j}$ is $ \pa_{\bar z_j}^h =\pa_{\bar z_j} -  \frac{i}{2} z_j \pa_\theta, $
and $\pa_{\bar z_j}^h \h s(z, \theta) = 0$.

 The corresponding lift of the  degree $k$ Bergman (or, \Szego) kernel $\h \Pi_k(\h z, \h w)$ to $X=\C^m \times S^1$ is given by 
\begin{equation} \label{BFSZEGODEF}   \h \Pi_k(\h z, \h w) =   \kk^m e^{k \h \psi(\h z, \h w)}, \end{equation} where $\h z =(z, \theta_z), \; \h w = (w, \theta_w)$ and the phase function is 
\be \label{BF-phase}  \psi(\h z, \h w)  = i   (\theta_z - \theta_w) + z \bar w - \half |z|^2 - \half |w|^2. \ee

\subsection{\label{METASECT} Metapletic Representation} The Harmonic oscillator is a quadratic operator. Such operators
form the symplectic Lie algebra. Their representations on Bargmann-Fock space is a unitary representation of the Lie algebra.
The integration this representation gives the metaplectic representation. There exist exact formulae for the Schwartz kernels
of metaplectic propagators, generalizing the Mehler formula. We need these formulae later on.  A thorough treatment can be found in \cite{F, HSj16}.

Let $\R^{2m}, \omega = 2 \sum_{j=1}^m dx_j \wedge d y_j$ be a sympletic vector space. The space $Sp(m, \R)$ consists of linear transformation $S: \R^{2m} \to \R^{2m}$, such that $S^*\omega = \omega$. In coordinates, we write 
\[ \bma x' \\ y' \ema = S \bma x \\ y \ema = \bma A & B \\ C & D \ema \bma x \\ y \ema. \]
 The semi-direct product of the symplectic group and Heisenberg group (sometimes called
the Jacobi group) thus consists of linear transformations fixing $0$ together with Heisenberg translations moving $0$ to any point.

In complex coordinates $z_i = x_i + i y_i$, we have then 
\[ \bma z' \\ \bar z' \ema = \bma P & Q \\ \bar Q & \bar P \ema \bma z \\ \bar z \ema =: \acal \bma z \\ \bar z \ema, \]
where 
\begin{equation} \label{PQDEF}  \bma P & Q \\ \bar Q & \bar P \ema = \wcal^{-1} \bma A & B \\ C & D \ema \wcal, \quad \wcal = \frac{1}{\sqrt 2} \bma I & I \\ -i I & iI \ema. \end{equation}
The choice of normalization of $\wcal$ is such that $W^{-1} = W^*$.Thus, 
\[ P = \half(A+D + i (C-B)). \]
 We say such $ \acal \in Sp_c(m, \R) \subset M(2n,\C)$. The following identities are often useful.
\bpp [ \cite{F} Prop 4.17]
Let $ \acal= \bma P & Q \\ \bar Q & \bar P \ema \in Sp_c$, then \\
(1) $ \bma P & Q \\ \bar Q & \bar P \ema^{-1} =\bma P^* & -Q^t \\ -Q^* & P^t \ema = K  \acal^* K$, where $K =  \bma I & 0 \\ 0 & -I \ema.$ \\
(2) $ PP^* - QQ^* = I$ and $P Q^t = Q P^t$. \\
(3) $P^*P - Q^t \bar Q = I$ and $P^t \bar Q = Q^* P$. 
\epp

The (double cover) of $Sp(m,\R)$ acts on the Bargmann-Fock  space $\hcal_k$ of $\C^{m}$ as an integral operator
with the following kernel: given $M= \bma P & Q \\ \bar Q & \bar P \ema \in Sp_c$, we define
\[  \kcal_{k,M}(z,  w) = \kk^{m} (\det P)^{-1/2} \exp \left\{k \half \left( z \bar{Q} P^{-1} z + 2 \bar{w} {P}^{-1} z
- \bar{w} P^{-1} Q \bar w \right)  \right\} \]
where the ambiguity of the sign the square root $(\det P)^{-1/2}$ is determined by the lift to the double cover. When $ \acal=Id$, then $\kcal_{k, \acal}(z, \bar w) = \Pi_k(z, \bar w)$.  The lifted  kernel upstairs on the reduced Heisenberg group $X$ is given by,
\be \h \kcal_{k, \acal}(\h z, \h w) = \kcal_{k,M}(z, \bar w) e^{k(i\theta_z -|z|^2/2) + k(-i\theta_w - |w|^2/2)}. \label{hatK}\ee

\subsection{\label{TOEPMETA}Toeplitz construction of the metaplectic representation}

The analogue of Weyl pseudo-differential operators on $L^2(\R^m)$ is (Berezin-)Toeplitz operators on Bargmann-Fock space. 
Given the semi-classical parameter $k$, the Berezin-Toeplitz quantization of a multiplication operator by a semi-classical symbol
$\sigma_k(Z, \bar{Z})$ on $\C^m$ is defined by 
\begin{equation} \label{TOEP3} \Pi_k \sigma_k(Z, \bar{Z}) \Pi_k. \end{equation}
It operators on Bargmann-Fock space by multiplying a holomorphic function by $\sigma_k$ and then projecting back onto
Bargmann-Fock space. More generally, one could let $\sigma_k$ be a semi-classical pseudo-differential operator. 

The isotropic Harmonic oscillator is on represented  on $\hcal_k(\C^d)$ as
$$\hat{H}_k = \Pi_k |Z|^2 \Pi_k. $$
It is equally well representated by $\sum_{j=1}^m a_j^* a_j + \frac{d}{2} = \sum_{j=1}^m z_j \frac{\partial}{\partial z_j} + \frac{d}{2}$, where $a_j = \frac{\partial}{\partial z_j}$ and
$a_j^* = z_j$ are the annihilation/creation operators. The operator $\sum_{j=1}^m a_j^* a_j $ is called the degree or number operator
since its action on a holomorphic polynomial is to give its degree. In a similar way, the infinitesimal metaplectic representation of quadratic polynomials $Q = Q(z, \bar{z})$  is by Toeplitz operators $\Pi_k Q \Pi_k$.   

 The Toeplitz construction of the metaplectic representation is due to Daubechies \cite{D80}.  The integrated  metaplectic representation  $W_J(S)$ 
of  $S \in Mp(n,\R)$ on $\hcal_J$  is defined as follows:    Let $S \in Sp(n, \R)$ and let 
$U_S$ be the unitary translation operator  on $L^2(\R^{2n}, d L)$ defined  by $U_S F(x, \xi): = F(S^{-1}(x, \xi))$. The metaplectic representation of $S$ on $\hcal_J$
is given by (\cite{D80},(5.5) and (6.3 b)) 
\begin{equation} \label{eta}W_J(S) = \eta_{J,S} \Pi_J U_S \Pi_J, \\
\end{equation} where  (see 
  \cite{D80} (6.1) and (6.3a)), 
\begin{equation} \label{ETAJS} \begin{array}{lll}
\eta_{J,S} & = & 2^{-n}  \det (I - i J) + S (I + i J)^{\half}

 \end{array} \end{equation} and $\Pi_J$ is the Bargmann-Fock
Szeg\"o projector.  

In the notation of the previous section, 
a quadratic Hamiltonian function $H: \C^m \to \R$  generates a one-parameter family of symplectic linear transformations $ \acal_t = g^t: \C^m \to \C^m$, which in general is only $\R$-linear and not $\C$-linear, i.e. $M_t$ does not preserve the complex structure of $\C^m$. Hence, one need to orthogonal project back to holomorphic sections. To compensate for the loss of norm due to the projection, one need to multiply a factor $\eta_{ \acal_t}$. 

\bpp \label{toep-met}
Let $ \acal: \C^m \to \C^m$ be a linear symplectic map, $\acal =  \bma P & Q \\ \bar Q & \bar P \ema$, and let $\h  \acal: X \to X$ be the contact lift that fixes the fiber over $0$, then 
\[  \h \kcal_{k, \acal}(\h z, \h w) = (\det P^*)^{1/2} \int_X \h \Pi_k(\h z, \h  \acal \h u) \h \Pi_k(\h u, \h w) d \Vol_X(\h u) \]
\epp
\bpf
The contact lift $\h \acal: \C^m \times S^1 \to \C^m \times S^1$ is given by $\acal$ acting on the first factor:
\[ \h \acal: (z, \theta) \mapsto (P z + Q \bar z, \theta), \]
one can check that $\h \acal^* \alpha = \alpha$. The integral over $X$ is a standard complex Gaussian integral, analogous to \cite[Prop 4.31]{F}, and with determinant Hessian $1/|\det P|$, hence we have $(\det P^*)^{1/2}/|\det P| = (\det P)^{-1/2}$. 
\epf

\subsection{Toeplitz Quantization of Hamiltonian flows} \label{TQD}

The Toeplitz construction of the metaplectic representation generalizes to the construction 
of a Toeplitz quantization of any symplectic map on any \kahler manifold as a Toeplitz operator on the quantizing line bundles \cite{Z97}.
In this section we briefly review the construction of a Toeplitz parametrix
for the propagtor $U_k(t)$ of the quantum Hamiltonian  \eqref{TOEP}. We refer to Section \ref{BACKGROUND} and to
 \cite{ZZ17,ZZ18} for the details.

Let $(M, \omega, L, h)$ be a polarized \Kahler manifold, and $\pi: X \to M$ the unit circle bundle in the dual bundle $(L^*, h^*)$.  $X$ is a contact manifold,  equipped with the Chern connection contact one-form $\alpha$, whose associated Reeb flow $R$ is the rotation $\pa_\theta$ in the fiber direction of $X$. Any Hamiltonian vector field $\xi_H$ on $M$ generated by a a smooth function $H: M \to R$ can be lifted to a contact Hamiltonian vector field $\h \xi_H$ on $X$, which generates a contact flow $\hat{g}^t$. The following Proposition from \cite{Z97} expresses the lift of \eqref{Ukt} to  $\hcal(X) = \bigoplus_{k\geq 0} \hcal_k(X)$. 

\begin{prop} \label{SC}
There exists a semi-classical symbol  $\sigma_{k}(t)$ so that the unitary group \eqref{Ukt}  has the form
\be  \label{TREP}  \hat{U}_k(t)   = \hat{\Pi}_{k}  (\hat{g}^{-t})^* \sigma_{k}(t) \hat{\Pi}_{k}  \ee
modulo smooth kernels of order $k^{-\infty}$.
\end{prop}

\subsection{Bargmann intertwining operator between \Sch and Bargmann-Fock}

 The standard unitary intertwining operator between the Schrodinger representation and the Bargmann-Fock representation is the (Segal-)Bargmann transform,
\begin{equation} \label{BT} B f(Z) = \int_{\R^n} \exp \left(- ( Z \cdot Z - 2 \sqrt{2} Z \cdot X + X \cdot X)/2\right) f(X) dX. \end{equation} Its inverse is its adjoint,
$$B^* F(x) = \int_{\C^n} \exp \left(- ( \bar{Z} \cdot \bar{Z} - 2 \sqrt{2} \bar{Z} \cdot X + X \cdot X)/2\right) F(Z) e^{- |Z|^2} L(dZ). $$
Another inversion formula is
$$f(x) = \pi^{-n/4} (2 \pi)^{-n/2} e^{- |x|^2} \int_{\R^n} (B f)(x + i y) e^{- |y|^2/2} dy. $$

The Bargmann transform is obtained from the Euclidean heat kernel by analytic continuation in the first variable. It might
be surprising that this transform is useful in studying the Harmonic oscillator. One could just as well analytically continue
the propagator \eqref{E:Mehler}, which also defines a unitary intertwining operator. However, that operator would simply analytically
continue Hermite functions, which does not simply the analysis. The Bargmann transform maps Hermite functions to holomorphic
polynomials, and the Hermite operator to the degree operator (up to a constant) and this is a significant simplification.

One may also use the Bargmann transform to convert Wigner distributions associated to spectral projections of the
Harmonic oscillator to the much simpler orthogonal projections onto spaces of holomorphic polynomials of fixed degree. 
The density of states (diagonal of a Bergman kernel) is known as a Husimi distribution in physics. An interesting historical
fact is that Cahill-Glauber studied the relation between Wigner distributions $W_{\Pi_{\hbar, E_M}}(x, \xi) $ and the Bargmann-conjugate Bergman Husimi
distributions $$ B \Pi_{\hbar, E_N} B^* (Z, \bar{Z}) $$
 in \cite{CG69I, CG69II}.
The Bargmann transform is the same as the spectral projections
of the Bargmann-Fock quantization $\Pi_{BF, k} |Z|^2 \Pi_{BF, k}$ of
$|Z|^2$. They showed that
$$\begin{array}{l} B_x \otimes B_y \int W_{\Pi_{\hbar, E_N}}(\frac{x + y}{2}, \xi) e^{i \langle x - y, \xi \rangle} d\xi =
\int_{\R^n} \int_{\R^n} \int_{\R^n} B(x, Z) B(y, Z)  W_{\Pi_{\hbar, E_N}}(\frac{x + y}{2}, \xi) e^{i \langle x - y, \xi \rangle} d\xi dx dy \end{array} $$
is  convolution of $W_{\Pi_{\hbar, E_M}}(x, \xi) $ with a complex Gaussian.

\subsection{\label{ANALOGS} Analogies and correspondences between the real and complex settings}
We now list some  important analogies to help navigate the results of this article, and to compare the results
in the real and complex settings. The undefined notation and terminology will be provided in the relevant section of this article. 
The reader is encouraged to consult this list as the article proceeds; it is probably not possible to understand much of it from the start.

Microlocal analysis provides a generalization of this equivalence to general manifolds. The generalization of the Bargmann
transform   (see Section \ref{BT})  is called an FBI transform. It is well-recognized that the setting of holomorphic sections of high powers $L^k \to M$
of ample line bundles over \kahler manifolds is quite analogous to  the setting of \Sch operators on Riemannian manifolds, to the
extent that one may expect parallel results in both domains. The role of the Planck constant $\hbar$ in semi-classical analysis 
is analogous to $k^{-1}$ in the line bundle setting. In fact, the relation between Wigner distributions and ``Husimi distributions'' (or partial
Bergman density of states) was first given by Cahill-Glauber in 1969 \cite{CG69I, CG69II} for applications in quantum
optics.  We refer to \cite{R87, Zw} for background in semi-classical analysis
and to  \cite{BG81} for background on Toeplitz operators. 

Here is a list of analogies which are relevant to the present survey.

\bigskip

\begin{itemize}

\item The cotangent bundle $(T^* \R^d, \sigma)$ equipped with its canonical symplectic structure 
is analogous to a \kahler manifold $(M, \omega)$. One may equip $T^*\R^d$ with a complex structure $J$ so that it becomes
the \kahler manifold $\C^d$. \bigskip

\item The total space of the  dual line bundle  $L^*$ of a holomorphic line bundle $L \to M$ is analogous to $\C^d$. Indeed, if $M = \CP^{d-1}$ (complex projective space),
then $\C^d = L^*$ where $L^*= \ocal(-1)$ is the tautological line bundle over $\CP^{d-1}$. (More precisely, $\C^d = \ocal(-1)$ with the zero
section `blown down'.) \bigskip

\item When $L$ is an `ample' line bundle, sections  $s_k \in H^0(M, L^k)$ in the space of holomorphic sections of the $k$th power of $L$ lift in a canonical
way to equivariant holomorphic functions $\hat{s}_k$  on $L^*$. In the case $(M, L) = (\CP^{d-1}, \ocal(-1))$, lifts of sections of $L^k$
are the holomorphic homogeneous polynomials on $\C^d$ of degree $k$. \bigskip

\item The total space $L$ carries an $S^1$ (circle) action, namely rotation in the fibers $L_z$ of $\pi: L \to M$. The generator
$D_{\theta}$ of this circle action  is analogous to the isotropic harmonic oscillator and to the degree operator. Namely if
$D_{\theta} \hat{s}_k = k \hat{s}_k$. The isotropic harmonic oscillator $\hat{H}_{\hbar}$ on $L^2(\R^d)$ is unitarily equivalent to the degree operator
on $\C^d$ under the Bargmann transform. \bigskip

 \item In the case $(M, L) = (\CP^{d-1}, \ocal(-1))$, $H^0(\CP^{d-1}, \ocal(k))$ is canonically isomorphic to the eigenspace 
of eigenvalue $k + \frac{d}{2}$ of the isotropic harmonic oscillator. \bigskip

\item Eigenspace spectral projection kernels $\Pi_{\hbar, E_N(\hbar)}(x,y)$ for eigenspaces $V_N$  of isotropic harmonic oscillators are analogous to Bergman kernels
$\Pi_{h^k}(z,w)$ for spaces $H^0(M, L^k)$ of holomorphic sections of powers of a positive Hermitian line bundle $(L, h)$ over
a \kahler manifold $(M, \omega)$. \bigskip

\item The Wigner distribution $W_{\hbar, E_N(\hbar)}(x, \xi)$ of an eigenspace projection is  analogous to the density of 
states $\Pi_{h^k}(z, z)$ where $\Pi_{h^k}$ is the  Bergman kernel for $H^0(M, L^k)$. The density of states is the contraction
of the diagonal of the Bergman kernel. \bigskip

\item Airy scaling asymptotics of  scaled Wigner distributions of eigenspace projections of the isotropic harmonic oscillator around an energy surface
$\Sigma_E \subset T^*\R^d$ are analogous to Gaussian error
function asymptotics of scaled Bergman kernels around an energy surface. Both live on `phase space'. The eigenspace projections of
the oscillator live on configuration (or, physical) space and have no simple analogue in the \kahler setting. \bigskip

\item The unitary Bargman transform $\bcal: L^2(\R^d) \to H^2(\C^d, e^{- |Z|^2} d L(Z))$ intertwines the real \Sch and holomorphic
Bargmann-Fock representations of quantum mechanics on $\R^d$. There is no simple analogue for general \kahler manifolds. It would be a 
unitary intertwining operator between the  Bargmann-Fock spaces of  $L^*$ and $L^2(N)$ where $N \subset M$ would be a totally
real Lagrangian submanifold.  See Section \ref{BT} for background.

\end{itemize}\bigskip

There is an important difference between the results on Wigner distributions and the results on partial Bergman kernels, which indicates
that there is much more to be done on interfaces in spectral asymptotics. Namely, in the \kahler setting we have two Hamiltonians: (i) A Toeplitz Hamiltonian
$\hat{H}_k: = \Pi_{h^k} H \Pi_{h^k}$ (where $H: M \to \R$ is a smooth function), and (ii) the operator $D_{\theta}$ on $L^*$ defining the degree $k$ of a lifted section. The latter is analogous to the isotropic oscillator. The interfaces for $D_{\theta}$ are interfaces across `disc bundles'
$D^*_R \subset L^*$ defined by a Hermitian metric $h$ on $L$. The analogue of Airy scaling asymptotics of Wigner distributions
is Gaussian error function asymptotics for lifts of Bergman kernels to $L^*$.  A Toeplitz Hamiltonian $\hat{H}_k$ lifts to a Hamiltonian
on $L^*$ which commutes with $D_{\theta}$, and our results on partial Bergman kernels pertain to the pair. So far, we have not considered
the analogous problem on $L^2(\R^d)$ defined by a second \Sch operator which commutes with the isotropic harmonic oscillator. As this
brief discussion indicates, there are many types of interface phenomena that remain to be explored.

\section{\label{SCH} Interface problems for \Sch equations}

In this section we consider the simplest \Sch operator, namely the isotropic Harmonic Oscillator on $\R^d$. We review three types 
of interface scaling results: 

\begin{itemize}

\item Scaling of the spectral projections kernel for a single eigenspace around the caustic. At the same time, we consider
scaling of nodal sets of random eigenfunctions around the caustic. \bigskip

\item Scaling asymptotics  of the Wigner distributions of the spectral projections kernel around an energy level in phase space. \bigskip

\item Scaling asymptotics of the Wigner distributions of Weyl sums of spectral projections kernels over an interval of energies 
at the boundary of the interval.

\end{itemize}

\subsection{Allowed and forbidden regions and the caustic}

Consider a general \Sch operator $\hat{H}_{\hbar} : = -\hbar^2 \Delta + V$ on $L^2(\R^d)$ with $V(x) \to \infty$ as $|x| \to \infty$. Then 
$\hat{H}_{\hbar}$ has a discrete spectrum of eigenfunctions $E_j(\hbar)$, 
\begin{equation} \label{SchEieg} \hat{H}_{\hbar} \psi_{\hbar, j} = E_j(\hbar) \psi_{\hbar, j}. \end{equation}
In the semi-classical limit 
\begin{equation} \label{SCLIMIT} \hbar \to 0, j \to \infty, E_j(\hbar) = E, \end{equation} the eigenfunctions of $\hat{H}_{\hbar}$ 
 are  rapidly oscillating in the classically allowed region 
\[\acal_E:=\set{V(x) \leq E},\]
and exponentially decaying in the classically forbidden region 
\[\fcal_E:=\acal_E^c=\set{V(x)>E}.\]
This reflects the fact that a classical particle of energy $E$  is  confined to $\acal_E=\set{V(x)\leq E}.$  We define the {\it caustic} to be
\begin{equation} \label{ccal} \ccal_E:=\dell \acal_E=\set{V(x)=E}.\end{equation} The exponential decay rate of eigenfunctions in the forbidden region as $\hbar \to 0$ is  measured by the Agmon distance to the caustic. We refer to \cite{Ag,HS} for background. 

In the first series of results we are interested in the transition between the oscillatory and exponential decay behavior of eigenfunctions
in a zone around the caustic \eqref{ccal}.  We review two types of results: (i) Airy scaling asymptotics of spectral projections kernels,
and (ii) interface asymptotics of nodal (i.e. zero) sets of `random eigenfunctions' in a spectral eigenspace. At this time, results
are only proved in the special case of the isotropic harmonic oscillator, but one may expect that suitably generalized results hold
rather universally.  

In the case of the isotropic Harmonic Oscillator, the  allowed region $\acal_E$, resp. the  forbidden region $\fcal_E$ are
given  respectively by, \begin{equation}
\label{AF2}
\acal_E = \{x:\abs{x}^2<2E\}, \quad \fcal_E = \{x:\abs{x}^2>2E\}.
\end{equation}
Thus, $\acal_E$ is the projection to $\R^d$ of the energy surface $\{H = E\}
\subset T^* \R^d$,  $\fcal_E$ is
its complement, and the  caustic set is given by,$$\ccal_E = \{|x| = 2 E\}. $$
 
The semi-classical limit at the energy level $E>0$ is the limit
as $\hbar \to 0, N \to \infty$ with fixed $E$, so that $\hbar$ only takes the values \eqref{hNEDEF}.

\subsection{Scaling asymptotics around the caustic in physical space}
Due to the homogeneity of the isotropic oscillator, it suffices to consider one value of $E$.  We fix $E = \half$ and consider $E_N(\hbar) =  \half$.  For this choice 
of $E$,  \eqref{PiDEF} is $\Pi_{\hbar, \half}. $

When $d=1,$ the eigenspaces $V_{\hbar_N, E}$ have dimension $1$ and it is a classical fact (based on WKB or ODE techniques) that Hermite functions and more general \Sch eigenfunctions exhibit Airy asympotics at the caustic (turning points).  See for instance \cite{O,T,FW}.  
It is not true for $d > 1$ that individual eigenfunctions exhibit analogous Airy scaling asymptotics around the caustic. Indeed, due to 
the high multiplicity of eigenvalues, there is a good theory of Gaussian random eigenfunctions of the isotropic oscillator, and random eigenfunctions do not exhibit Airy scaling asymptotics. The proper generalization of the $d =1$ result is to consider the scaling asymtptoics of the eigenspace projection
kernels  \eqref{PiDEF} with $x, y$ in an $\hbar^{2/3}$-tube around 
$\ccal_E$.

The first result states that  {\it individual}  eigenspace projection
kernels \eqref{PiDEF} exhibit Airy scaling asymtotics around a point $x_0 \in \ccal_E$ of the caustic.
Let $x_0$ be a point on the caustic $|x_0|^2=1$ for $E=1/2$. Points in an $\hbar^{2/3}$ neighborhood
of $x_0$ may be expressed as $x_0 + \hbar^{2/3} u$ with
$u \in \R^d$.  The caustic is a $(d-1)$-sphere whose
normal direction at $x_0$  is $x_0$, so the normal component of $u$ is
 $u_1 x_0$ when $|x_0| = 1$, where $u_1:=\inprod{x_0}{u}$. We also put  $u':=u-u_1x_0$ for the tangential component, and identify $T_{x_0} \ccal_E \cong T^*_{x_0} \ccal_E \cong \R^{d-1}$.   By rotational symmetry, we 
 may assume  $x_0 = (1, 0, \cdots, 0)$, so that $u=(u_1, u_2, \cdots, u_d) =: (u_1; u')$.

\begin{theo}\label{SCLintro}
Let $x_0$ be a point on the caustic $|x_0|^2=1$ for $E=1/2$. Then for  $u,v \in \R^d$,
\begin{equation}
\Pi_{\hbar,1/2} (x_0 + \hbar^{2/3} u, x_0 + \hbar^{2/3} v) = \hbar^{-2d/3+1/3} \Pi_0(u, v) (1 + O(\hbar^{1/3})), \label{E:CausticScaling}
\end{equation}
where  
\be 
\label{Pi0}
\Pi_0(u_1, u'; v_1, v') := 2^{2/3} (2 \pi)^{-d+1} \int_{\R^{d-1}} e^{i \lan u'-v', p\ran } \Ai(2^{1/3}(u_1 + p^2/2))\Ai(2^{1/3}(v_1 + p^2/2)) dp,
\ee
and  $u_1:=\inprod{x_0}{u}$, $u':=u-u_1x_0$ (similarly for $v_1.$)
On the diagonal, let $\abs{x}^2 = \abs{x_0+\hbar^{2/3} u}^2 = 1 + \hbar^{2/3} s+O(\hbar^{4/3})$ with $s = 2 \lan x_0, u\ran \in \R$. Then,
\be\label{pi-tube-1} \Pi_{\hbar}(x,x) = 2^{-d+1}\pi^{-d/2} \hbar^{(1-2d)/3} \Ai_{-d/2}(s)(1+O(\hbar^{1/3})). \ee
The error terms in \eqref{E:CausticScaling} and \eqref{pi-tube-1} are uniform when $u,v,s$ vary over a compact set. 
\end{theo}

\noindent 
Above, $\Ai$ is the Airy function, and $\Ai_{-d/2}$ is a {\it weighted Airy function},  defined  for $k \in \R$ by
\begin{equation}
 \label{eq:Ai_k} \Ai_{k}(s) := \int_\ccal T^{k} \exp \left(  \frac{T^3}{3} - T s\right) \frac{dT}{2\pi i},\qquad u\in \R
\end{equation}
where $\ccal$ is the usual contour for Airy function, running from $e^{-i \pi/3} \infty$ to $e^{i \pi/3} \infty$ on the right half of the complex plane (see Section \ref{AIRYAPP}  for a brief review of the Airy function).  

\begin{remark}\label{R:TW} When $d = 3$, the kernel \eqref{Pi0} with $u' = v'$, i.e.  $\Pi_0(u_1, u'; v_1, u')$, coincides modulo the factor of $\sqrt{\lambda}$ with  the Airy kernel $K(x,y)$ of the Tracy-Widom distribution.  The ``allowed region'' of this article is analogous to the `bulk' in random matrix theory, and the ``caustic'' of this article is
analogous to the ``edge of the spectrum''.  \end{remark}

\subsection{Nodal sets of random Hermite eigenfunctions}
 Theorem \ref{SCLintro} can be used to determine the interface behavior of nodal (zero) sets of random eigenfunctions of the isotropic
 oscillator of a fixed eigenvalue. In many ways, the isotropic oscillator is the analogue among \Sch operators on $L^2(\R^d)$ of
 the Laplacian on a standard sphere ${\mathbb S}^d$, and the  study of random Hermite eigenfunctions is somewhat analogous
 to the study of random spherical harmonics. However, there are no forbidden regions in the case of ${\mathbb S}^d$, and the
 interface behavior of random Hermite eigenfunctions has no parallel for random spherical harmonics.

\begin{defn}
  A Gaussian random eigenfunction for $H_h$ with eigenvalue $E$ is the
  random series
$$ \Phi_N(x):=\sum_{\abs{\alpha}=N}  a_{\alpha}\phi_{\alpha,h_N}(x),  $$
for $a_{\alpha}\sim N(0,1)_{\R}$ i.i.d. Equivalently, it is the Gaussian
measure $\gamma_N$ on $V_N$
which is given by $e^{- \sum_{\alpha} |a_{\alpha}|^2/2} \prod d
a_{\alpha}$.
\end{defn}
\noindent We denote by $$Z_{\Phi_N} = \{x: \Phi_N(x) = 0\} $$
 the nodal set of $\Phi_N$ and by $\ZN$ the random measure of integration
 over
$Z_{\Phi_N}$ with respect to
 the Euclidean surface measure (the Hausdorff measure) of the nodal set.
 Thus for any ball $B \subset \R^d$,
$$\ZN (B) = \hcal^{d-1} (B \cap Z_{\Phi_N}).$$
Thus $\E \ZN$ is a measure on $\R^n$ given by
$$\E \ZN (B) = \int_{V_N}  \hcal^{d-1} (B \cap Z_{\Phi_N}) d\gamma_N. $$

The first result gives semi-classical asymptotics of the hypersurface volumes of the nodal sets of random Hermite eigenfunctions
of fixed eigenvalue in the allowed, resp. forbidden region. 

\begin{theo}\label{T:Main}
Let $x\in \R^d$ such that $0<\abs{x}\neq \sqrt{2E}.$ Then the measure $\E
\ZN$ has a density $F_N(x)$ with
respect to Lebesgue measure given by
$$\left\{\begin{array}{ll}
\mbox{If}~x\in \acal_E\backslash \set{0},  &
F_N(x) \simeq h^{-1}\cdot c_d \sqrt{2E-\abs{x}^2}\lr{1+O(h)}\label{E:Allowed
Density} \\ & \\
\mbox{If}~ x\in \fcal_E, &
F_N(x) \simeq h^{-1/2}\cdot C_d
\frac{E^{1/2}}{\abs{x}^{1/2}\lr{\abs{x}^2-2E}^{1/4}}
\lr{1+O(h)}\label{E:Forbidden Density}
 \end{array}, \right.$$
where the implied constants in the `$O$' symbols are uniform on compact
subsets of the interiors of
$\acal_E\backslash\set{0}$ and $\fcal_E$, and where
\[c_d=
\frac{\Gamma\lr{\frac{d+1}{2}}}{\sqrt{d\pi}\Gamma\lr{\frac{d}{2}}}\qquad
\text{and}\qquad C_d =
\frac{\Gamma\lr{\frac{d+1}{2}}}{\sqrt{\pi}\Gamma\lr{\frac{d}{2}}}.\]
\end{theo}
The key point is  the different growth rates in
$h$ for the density of zeros in the
allowed and forbidden region. 
In dimension one, eigenfunctions have no zeros in the forbidden region, but in dimensions $d \geq 2$ they do. In the allowed region, nodal sets of eigenfunctions behave in a similar way to nodal sets on Riemannian manifolds \cite{Jin},  but in the forbidden region they are sparser. 

The next  result on nodal sets (Theorem \ref{CAUSTIC}) gives scaling asymptotics for the average nodal density that `interpolate' between \eqref{E:Allowed Density} and \eqref{E:Forbidden Density}. Fix $x \in \ccal_E$, where $E=1/2$, and study the  rescaled ensemble
\[\Phi_{\hbar, E}^{x,\alpha}(u):=\Phi_{\hbar,E}(x+\hbar^\alpha u)\]
and the associated hypersurface measure 
\[\abs{Z_{\hbar, E}^{x,\alpha}}(B)=\hcal^{d-1}\lr{\set{\Phi_{\hbar, E}^{x,\alpha}(v)=0}\cap B},\qquad B\subset \R^d.\]
The next result gives the asymptotics of $\E \abs{Z_{\hbar, E}^{x,\alpha}}$ when $\alpha = 2/3$ is in terms of the weighted Airy functions $\Ai_k$ (see \eqref{eq:Ai_k}).

\begin{theo}[Nodal set in a shrinking ball around a caustic point]\label{CAUSTIC} Fix $E=1/2$ and $x\in \mathcal C_E$, i.e. $|x|=1$. For any bounded measurable $B\subseteq \R^d,$ 
\[\E \abs{Z_{\hbar, E}^{x,2/3}}(B)=\int_B \fcal(u)du,\]
where 
\be \fcal(u)= \lr{2\pi}^{-\frac{d+1}{2}}\int_{\R^d}|\Omega(u)^{1/2}\xi|e^{-\abs{\xi}^2/2}d\xi ~(1 + O(\hbar^{1/3}))\label{E:CausticDensity}\ee
and $\Omega=\lr{\Omega_{ij}}_{1\leq i,j\leq n}$ is the symmetric matrix
\begin{equation}
\Omega_{ij}(u) = x_i x_j \left( \frac{\Ai_{2-d/2}(s)}{\Ai_{-d/2}(s)} - \frac{\Ai^2_{1-d/2}(s)}{\Ai^2_{-d/2}(s)} \right) + \frac{\delta_{ij}}{2} \frac{\Ai_{-1-d/2}(s)}{\Ai_{-d/2}(s)}.\label{OMEGAintro}
\end{equation}
where $s = 2\lan u, x \ran$. 
The implied constant in the error estimate from \eqref{E:CausticDensity} is uniform when $u$ varies in compact subsets of $\R^d$. 
\end{theo}
\begin{remark}
The leading term in $\fcal$ is $\hbar$-independent and positive everywhere since the matrix $\Omega_{ij}(u)$ as a linear operator has nontrivial range.  The matrix $\lr{x_i x_j}_{i,j}$ in \eqref{OMEGAintro} is a rank $1$ projection onto the $x-$direction; since the dimension $d \geq 2$, it cannot cancel out the second term. We refer to \cite{HZZ15,HZZ16} for details.
\end{remark}
\begin{remark}\label{R:Unscaled}
Theorem \ref{CAUSTIC} says that if $x\in \ccal_E$ and $\twiddle{B}_\hbar=x+\hbar^{2/3}B$ for some bounded measurable $B,$ then 
\[\E{\abs{Z_{\Phi_{\hbar,E}}}}(\twiddle{B}_\hbar)=\hbar^{2/3\lr{d-1}}\E{\abs{Z_{\hbar, E}^{x,\alpha}}}(B)=\hbar^{-2/3}\int_{\twiddle{B}_\hbar}\fcal(\hbar^{-2/3}\lr{y-x})dy,\]
which shows that the average (unscaled) density of zeros in a $\hbar^{2/3}-$tube around $\ccal_E$ grows like $\hbar^{-2/3}$ as $\hbar\gives 0.$
\end{remark}

\begin{remark} 
The  scaling asymptotics of zeros around the caustic,
especially in the radial (normal) direction,   is analogous to the scaling asyptotics of  eigenvalues of
random Hermitian matrices around the edge of the spectrum. \end{remark}







\begin{figure}
\vspace{-11pt}
\begin{center}\label{F:Heller}
  \includegraphics[width=.8 \textwidth]{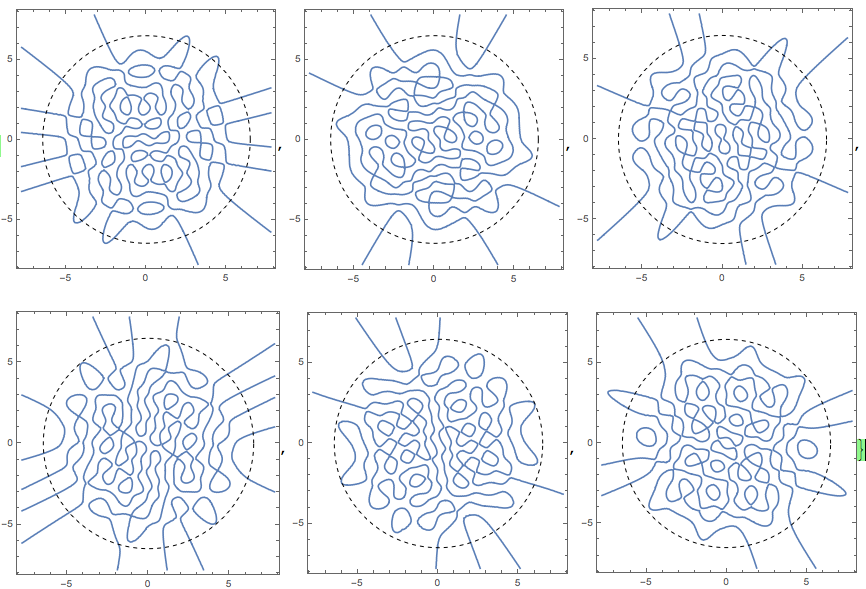}
\end{center}

The nodal set is very dense and busy in $\acal_E$ and rather sparse and
`non-oscillating' in $\fcal_E$.

\end{figure}

\subsection{Discussion of the nodal results}

Computer graphics of Bies-Heller \cite{BH} (reprinted as Figure \ref{F:Heller} in \cite{HZZ15}) and the displayed graphics of Peng Zhou  show that the nodal set in $\acal_E$ near the caustic $\partial \acal_E$ consists of a large number of highly curved nodal components apparently touching the caustic while the nodal set in $\fcal_E$ near $\partial \acal_E$ consists of fewer and less curved nodal components all of which touch the caustic. This is because, if  $\psi\in V_{\hbar, E}$ is non-zero,  $\Delta \psi = (V - E) \psi$ forces $\psi$ and $\Delta \psi$ to have the same sign in $\fcal_E$. In a nodal domain $\dcal$ we may assume $\psi > 0$, but then $\psi$ is a positive subharmonic function in $\dcal$ and cannot be zero on $\partial \dcal$ without vanishing identically. Hence, every nodal component which intersects $\fcal_E$ must also intersect $\acal_E$ and therefore $\ccal_E$. 

The scaling limit of the density of zeros in a shrinking neighborhood of the caustic, or in annular subdomains of $\acal_E$ and $\fcal_E$ at shrinking distances from the caustic is given in  Theorem \ref{CAUSTIC}.

\subsection{\label{KRSECT} The Kac-Rice Formula} The proof of Theorem \ref{CAUSTIC}  is based on the Kac-Rice formula for the average density of zeros.

\begin{lem}[Kac-Rice for Gaussian Fields]\label{L:Gaussian KR}
Let $\Phi_{\hbar, E}$ be the random Hermite eigenfunction of $\widehat{H}_\hbar$ with eigenvalue $E$. Then the density of zeros of $\Phi_{\hbar, E}$ is given by
\begin{equation}
F_{\hbar, E}(x)= \lr{2\pi}^{-\frac{d+1}{2}}\int_{\R^d}|\W^{1/2}(x)\xi| \;\; e^{-\abs{\xi}^2/2} \;\; d\xi,\label{E:Gaussian KR}
\end{equation}
where $\W(x)$ is the $d\x d$ matrix
\begin{align}
\notag \W_{ij}(x) &= (\dell_{x_i}\dell_{y_j} \log \Pi_{\hbar, E})(x,x)\\
\label{E:Gaussian Cov Mat} &=
\frac{(\Pi_{\hbar, E}\cdot \dell_{x_i}\dell_{y_j}\Pi_{\hbar, E})(x,x)-(\dell_{x_i}\Pi_{\hbar, E} \cdot \dell_{y_j}\Pi_{\hbar, E})(x,x)}{\Pi_{\hbar, E} (x,x)^2}
\end{align}
and $\Pi_{\hbar, E}(x,y)$ is the kernel of eigenspace projection \eqref{COV}.
\end{lem}

We refer to \cite{HZZ15, HZZ16} for background. The main task in proving results on zeros near the caustic is therefore to work out the asymptotics of $\Pi_{\hbar, E}(x,x)$ and its derivatives there.

\section{\label{WIGNER} Interfaces in phase space for \Sch operators: Wigner distributions}

We now turn to phase space interfaces. Instead of studying the scaling asymptotics of the spectral projections \eqref{PiDEF} \begin{equation} \label{PiDEF2} \Pi_{\hbar, E_N(\hbar)}: L^2(\R^d) \to V_{\hbar,E_N(\hbar)} \end{equation}
we study the scaling asymptotics of their semi-classical Wigner distributions \begin{equation}\label{WIGNERDEF1}
 W_{\hbar, E_N(\hbar)}(x, \xi) := \int_{\R^d} \Pi_{\hbar, E_N(\hbar)} \left( x+\frac{v}{2}, x-\frac{v}{2} \right) e^{-\frac{i}{\hbar} v \cdot \xi} \frac{dv}{(2\pi h)^d} 
\end{equation} 
across the phase space energy surface \eqref{SIGMAEDEF}.

When $E_N(\hbar) = E + o(1)$ as $\hbar \to 0$, $W_{\hbar, E_N(\hbar)}$ is thought of as the `quantization' of the energy surface,
and \eqref{WIGNERDEF1} is thought of as an approximate $\delta$-function on \eqref{SIGMAEDEF}. This is true in the weak* sense, but the pointwise behavior is quite a bit more complicated and is studied in \cite{HZ19}.

Wigner distributions were introduced in \cite{W32} as phase space densities. Heuristically, the Wigner distribution \eqref{PiDEF}  is a kind of probability density in phase space of finding a particle of energy $E_N(\hbar)$ at the
point $(x, \xi) \in T^* \R^d$. This is not literally true, since $W_{\hbar, E_N(\hbar)}(x, \xi)$ is not positive: it oscillates with  heavy tails inside the energy surface \eqref{SIGMAEDEF}, 
has a kind of transition across $\Sigma_E$ and then
decays  rapidly outside the energy surface.  The purpose of this paper is to give detailed results on  the concentration and oscillation properties of these Wigner distributions in three phase space regimes, depending on
the position of $(x, \xi)$  with respect to $\Sigma_E$.

 There is an  exact formula for the Wigner distributions \eqref{E:Wignera} of the eigenspace projections for the isotropic Harmonic oscillator in terms of Laguerre functions (see Appendix \ref{S:Laguerre}
 and \cite{T} for background on Laguerre functions).

\begin{prop} \label{WIGNERLAGUERRE} The Wigner distribution of 
Definition \ref{WIGNERPROJDEF} 
 is given by,
\begin{equation}    \label{E:Wigner-sp}
    W_{\hbar, E_N(\hbar)}(x, \xi) =  \frac{(-1)^N}{(\pi \hbar)^d}
    e^{-  2H/\hbar}  L^{(d-1)}_N(4H/\hbar),\qquad H=H(x,\xi)=\frac{\abs{x}^2+\abs{\xi}^2}{2},
\end{equation}
where $L_N^{(d-1)}$ is the associated Laguerre polynomial of degree $N$ and type $d-1$.
\end{prop}
See  \cite{O,JZ} for $d=1$ and  \cite[Theorem 1.3.5]{T} and \cite{HZ19} for general dimensions.
The second result is a weak* limit result for normalized Wigner distributions.

\begin{prop}  \label{OPWa} Let  $a_0$ be a semi-classical symbol of order zero and let $Op_h^w(a)$ be its Weyl quantization.  Then, as $\hbar \to 0$, with
$E_N(\hbar) \to E$,
$$\frac{1}{\dim V_{\hbar, E_N(\hbar)}} \int_{T^* \R^d} a_0(x,\xi)  W_{\hbar, E_N(\hbar)}(x,\xi) dx d \xi \to \dashint_{\Sigma_E} a_0 d \mu_E,$$
where $d\mu_E$ is Liouville measure on $\Sigma_E$ and $ \dashint_{\Sigma_E} a_0 d \mu_E = \frac{1}{\mu_E(\Sigma_E)}  \int_{\Sigma_E} a_0 d \mu_E$.
\end{prop}
 Thus,  $W_{\hbar, E_N(\hbar)}(x,\xi) \to \delta_{\Sigma_E}$ in the sense of weak* convergence. But this limit is due to the oscillations inside the energy ball; the  pointwise asymptotics are far more complicated.



\begin{figure}
\begin{center}
  \includegraphics[width=.6 \textwidth]{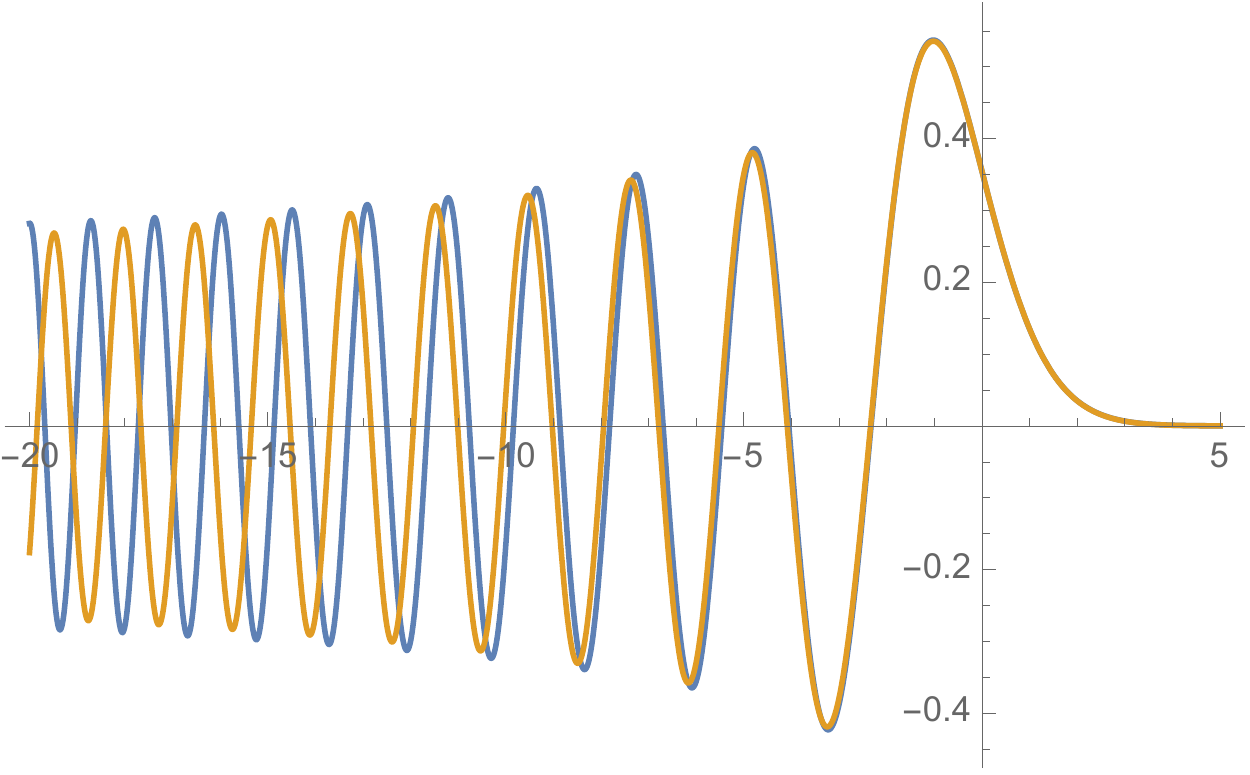} 
\end{center}
\caption{\label{fig-Wigner-eigenspace-2}
 The Wigner function $W_{\hbar, E_N(\hbar)}$ of the eigenspace projection $\Pi_{\hbar, E_N(\hbar)}$ is always radial (see Proposition \ref{WIGNERLAGUERRE}). Displayed above is the graph of the Airy function (orange) and of $W_{\hbar, E_N(\hbar)}$ with $N=500$ (blue) as a function of the rescaled radial variable $\rho$ in a $\hbar^{2/3}$ tube around the energy surface $H(x,\xi)=E_N(\hbar)=1/2.$ Theorem \ref{SCALINGCOR-old} predicts that, when properly scaled, $W_{\hbar, E_N(\hbar)}$ should converge to the Airy function (with the rate of convergence being slower farther from the energy surface, which is defined here by $\rho=0$).}
\vspace{-10pt}
\end{figure}

%


 \subsection{Interface asymptotics for Wigner distributions of  individual eigenspace projections}

Our first main result gives the scaling asymptotics for the Wigner function $W_{\hbar, E_N(\hbar)}(x,\xi)$ of the projection onto the $E$-eigenspace of $\widehat{H}_{\hbar}$ when $(x,\xi)$ lies in an $\hbar^{2/3}$ neighborhood of the energy surface $\Sigma_E.$
\begin{theo}\label{SCALINGCOR-old} 
Fix $E>0,d\geq 1$. Assume $E_N(\hbar)  = E$ and let $\hbar = \hbar_N(E)$ \eqref{hNEDEF}. 
Suppose $(x,\xi)\in T^*\R^d$ satisfies
\begin{equation} \label{udef}
H(x,\xi)=E + u \lr{\frac{\hbar}{2E}}^{2/3},\qquad u \in \R,\, H(x,\xi)=\frac{\norm{x}^2+\norm{\xi}^2}{2}\end{equation}
with $\abs{u}<\hbar^{-1/3}$. \footnote {The errors blow up when $u=\hbar^{-1/3}$.} Then,
\begin{equation}
  \label{E:W-scaling}
 W_{\hbar, E_N(\hbar)}(x, \xi)=
  \begin{cases}
 \frac{2}{(2\pi \hbar)^d} \lr{\frac{\hbar}{2E}}^{1/3} \lr{\Ai(u/E) + O\lr{(1+\abs{u})^{1/4}u^2\hbar^{2/3}}},&\qquad u<0\\
 \frac{2}{(2\pi \hbar)^d} \lr{\frac{\hbar}{2E}}^{1/3} \Ai(u/E)\lr{1 + O\lr{(1+\abs{u})^{3/2}u\hbar^{2/3}}},&\qquad u>0
  \end{cases}
\end{equation}
\end{theo}

\noindent Here, $\Ai(x)$ is the Airy function. The Airy scaling of $W_{\hbar, E_N(\hbar)}$ is illustrated in Figure 2. 
The assumption \eqref{udef} may be stated more invariantly that $(x, \xi)$ lies in the tube of radius $O(\hbar^{2/3})$ around $\Sigma_E$ defined by the gradient flow of $H$ with respect to the Euclidean metric on $T^*\R^d$. The asymptotics are illustrated in figure \ref{fig-Wigner-eigenspace-2}. Due to the behavior of the Airy function $\Ai(s)$, these formulae show that in the semi-classical limit $\hbar \to 0$, $E_N(\hbar) \to E$, $W_{\hbar, E_N(\hbar)}(x, \xi)$  concentrates on the energy surface surface $\Sigma_E$, is oscillatory inside the energy ball $\{H \leq E\}$ and is exponentially decaying outside the ball. 

\subsection{\label{BESSELSECTINTRO} Interior Bessel asymptotics}

In addition to the Airy asymptotics in an $\hbar^{2/3}$-tube around
$\Sigma_E$, $W_{\hbar, E_N(\hbar)}$ exhibits Bessel asymptotics
in the interior of $\Sigma_E$.
There are two (or three, depending on taste)  uniform asymptotic regimes for the Laguerre polynomial $L_n^{(\alpha)}(x)$: Bessel, Trigonometric, Airy. 

 For  $t \in [0, 1)$, define
$$A(t) = \half [\sqrt{t -t^2} + \sin^{-1} \sqrt{t}], \; t \in [0,1].$$
For $t <0$ the $\sin^{-1}$ is replaced by $\sinh^{-1}$ and the $\half$ by
$i/2$ (see \cite[(2.7)]{FW}). Also, let $J_{d-1}$ be the Bessel function (of the first kind) of index $d-1$.

 \begin{theo} \label{BESSELPROP} Fix $E>0$ and suppose $E_N(\hbar)=E.$ For each $(x,\xi)\in T^*\R^d$ write
\[H_E := \frac{H(x,\xi)}{E}=\frac{\norm{x}^2+\norm{\xi}^2}{2E},\qquad \nu_E:=\frac{4E}{\hbar}.\]
Fix $0<a < 1/2.$ Uniformly over $a\leq H_E \leq 1-a$, there is an asymptotic expansion,
\[W_{\hbar, E_N(\hbar)}(x, \xi)=\frac{2}{(2\pi\hbar)^d}\left[\frac{J_{d-1}(\nu_E A(H_E))}{A(H_E)^{d-1}}\alpha_0(H_E)+O\lr{ \nu_E^{-1}\abs{\frac{J_{d}(\nu_E A(H_E))}{A(H_E)^{d}}}}\right].\]
In particular, uniformly over $H_E$ in a compact subset of $(0,1),$ we find
\begin{equation}
W_{\hbar, E_N(\hbar)}(x, \xi) = (2\pi\hbar)^{-d+1/2} P_{H,E}\cos\lr{\xi_{\hbar, E,H}}+O\lr{\hbar^{-d+3/2}},\label{E:interior-cosine}
\end{equation}
where we've set
\[ \xi_{\hbar, E,H} =-\frac{\pi}{4} -\frac{2H}{\hbar}\lr{H_E^{-1}-1}^{1/2}+\frac{2E}{\hbar}\cos^{-1}\lr{H_E^{1/2}}\]
and 
\[P_{E,H}:=\lr{\pi E^{1/2}\lr{H_E^{-1}-1}^{1/4}\lr{H_E}^{d/2}}^{-1}.\]
\end{theo}

\subsection{Small ball integrals }

The interior Bessel asymptotics do not encompass the behavior of $W_{\hbar, E_N(\hbar)}$ in shrinking balls around $\rho = 0$. In that
case, we have,
\begin{prop} \label{SMBALLS} For $\epsilon>0$ sufficiently small and for
any $a(x, \xi) \in C_b(T^*\R^d)$,
   \begin{equation}
\int_{T^*\R^d} a(x, \xi) 
W_{\hbar, E_N(\hbar)}(x,\xi) \psi_{\epsilon,\hbar}(x,\xi)d x d \xi= O(\hbar^{\frac{1-d}{2}-2d\epsilon}\norm{a}_{L^\infty(B_0(\hbar^{1/2-\epsilon}))}).
\label{E:localized}
\end{equation}
where $\psi_{\epsilon,\hbar}$ is a smooth radial cut-off that is identically $1$ on the ball of radius $\hbar^{1/2-\epsilon}$ and is identically $0$ outside the ball of radius $2\hbar^{1/2-\epsilon}.$ 
\end{prop}

\subsection{Exterior asymptotics} If $E_N(\hbar) \to E$, then $W_{\hbar, E_N(\hbar)}(x, \xi)$ concentrates on $\Sigma_E$ and is exponentially decaying in the complement $H=H(x,\xi)> E$. The precise statement is,
\begin{prop}\label{EXTDECAY} Suppose that $H_E=H(x, \xi)/E>1$ and let $E_N(\hbar)  =E$. Then, there exists $C_1>0$ so that
\[|W_{\hbar, E_N(\hbar)}(x, \xi)| \leq C_1 \hbar^{-d + \frac{1}{2}} 
e^{- \frac{2 E}{\hbar}[\sqrt{H_E^2 -H_E} - \cosh^{-1}\sqrt{H_E} ]}. \]
Moreover, as $H(x,\xi) \to \infty$, there exists $C_2>0$ so that
$$|W_{\hbar, E_N(\hbar)}(x, \xi)| \leq C_2 \hbar^{-d + \frac{1}{2}} e^{- \frac{2H(x,\xi)}{\hbar}}. $$
\end{prop}

 \subsection{\label{SUPSECT} Supremum at $\rho =0$} 
 
 The  reader may notice the `spike' at the origin $\rho = 0$; it is  the point at which $W_{\hbar, E_N(\hbar)}$ has its global maximum (see Figure \ref{fig-Wigner-0}).  
 The height is given by    
\begin{equation} \label{BUP} \begin{array}{lll} W_{\hbar, E_N(\hbar)}(0,0) & = &
 \frac{(-1)^N}{(\pi \hbar)^d}
      L^{d-1}_N(0) =  \frac{(-1)^N}{(\pi \hbar)^d}  \frac{\Gamma(N + d)}{\Gamma(N +1) \Gamma(d)}  \simeq \frac{(-1)^N}{\pi ^d}  C_d \hbar^{-d} N^{d-1}.\end{array} \end{equation}
The last statement follows from the explicit formula
$   L^{(d-1)}_N(0) =  \frac{\Gamma(N + d)}{\Gamma(N +1) \Gamma(d)} = \frac{(N + d-1)!}{N! (d-1)!}$ (see e.g. \cite[(1.1.39)]{T}).

\begin{figure}
\begin{center}
  \includegraphics[width=.6 \textwidth]{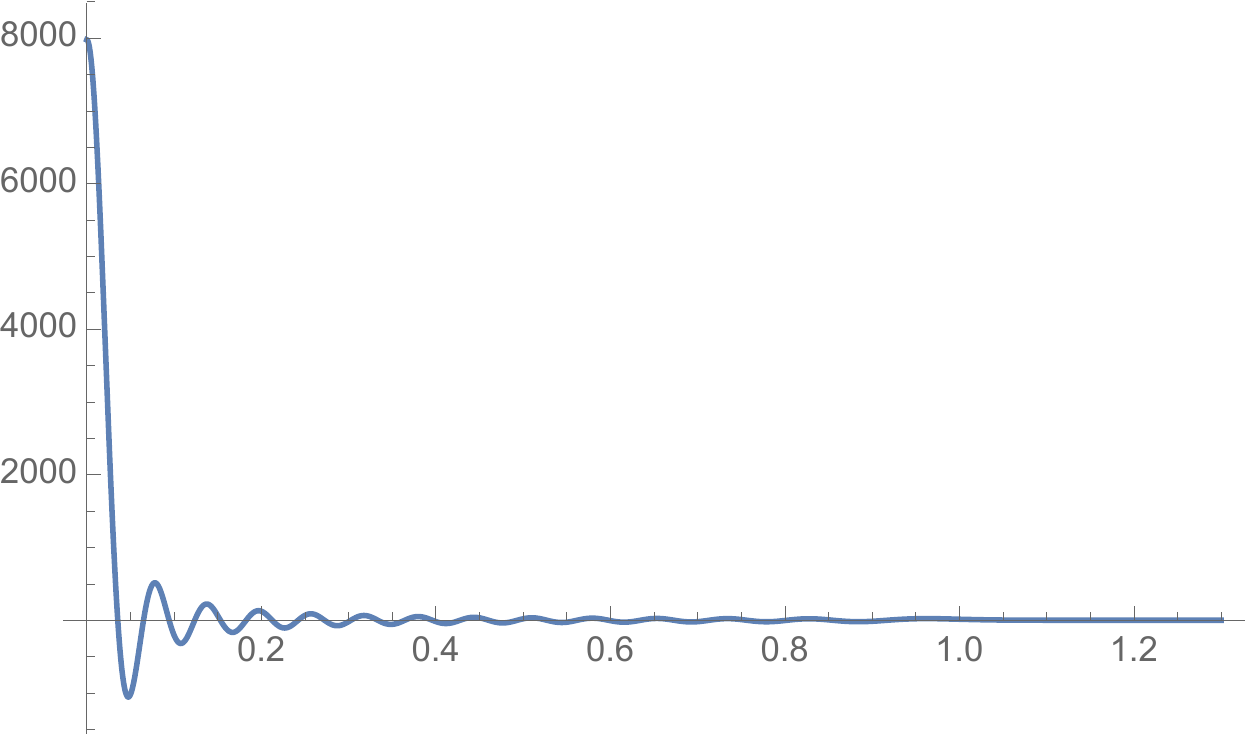} 
\end{center}
\caption{\label{fig-Wigner-0}
 The Wigner function $W_{\hbar, E_N(\hbar)}$ of the eigenspace projection $\Pi_{\hbar, E_N(\hbar)}$ is always radial (see Proposition \ref{WIGNERLAGUERRE}). Displayed above is the blow-up of the Wigner function at $(0,0)$.}
\vspace{-10pt}
\end{figure}

On the complement of the ball  $B(0, \hbar^{\half - \epsilon})$, the Wigner
distribution is much smaller than at its maximum. The following is proved
by combining the estimates of Theorem \ref{SCALINGCOR-old} , Theorem \ref{BESSELPROP} and Proposition
\ref{EXTDECAY}.

\begin{prop} \label{WIGBOUND} For any $\epsilon > 0$,
$$\sup_{(x, \xi): H(x, \xi) \geq \epsilon} |W_{\hbar, E_N(\hbar)}(x, \xi)| \leq C \hbar^{-d + \frac{1}{3}}. $$
The supremum in this region  is achieved in $\{H \leq E\}$ at $(x,\xi)$
satisfying \eqref{udef} where $u$ is the global maximum of $\rm{Ai}(x)$.
\end{prop}

Why the spike at $\rho= 0$? It is observed in \cite{HZ19} that 
$W_{\hbar, E_N(\hbar)}$ is an eigenfunction of the (essentially isotropic) \Sch operator \begin{equation} \label{WIGNEREIG3} \left(- \frac{\hbar^2}{8}   (\Delta_{\xi} + \Delta_x) + H(x,\xi)\right) W_{\hbar,E_N(\hbar)}
= E_N(\hbar) W_{\hbar, E_N(\hbar)}.  \end{equation} on $T^*\R^d$. 
 By
\cite[Lemma 10]{HZZ15},  the eigenspace spectral projections for the isotropic harmonic oscillator in dimension $d$ satisfies,
$$\Pi_{h,E}(x,x)=\lr{2\pi
  h}^{-(d-1)}\lr{2E-\abs{x}^2}^{\frac{d}{2}-1}\omega_{d-1}\lr{1+O(h)}, $$
  for a dimensional constant $\omega_d$. We apply this result to the eigenspace projections for \eqref{WIGNEREIG3} in dimension $2d$ and find that at the point $(0,0)$ its diagonal value is of order $\hbar^{-2d + 1}$.  We then express this eigenspace projection in terms of an orthonormal
  basis for the eigenspace.  From the inner product formulae \eqref{INTEGRALS}, it is seen that one of the orthonormal basis elements is  $\frac{1}{\sqrt{\dim V_{\hbar, E_N(\hbar)}}} W_{\hbar, E_N(\hbar)}$. Note that  ${\dim V_{\hbar, E_N(\hbar)}} \simeq \hbar^{-2d + 1}$ in dimension $2d$. Due to the normalization and \eqref{BUP}, 
  $$\frac{1}{\sqrt{\dim V_{\hbar, E_N(\hbar)}}} W_{\hbar, E_N(\hbar)}(0,0) \simeq \hbar^{-2 d + 1 + d -\half} = \hbar^{-d + \half}. $$
    
There exists a simple spectral geometric explanation for the order of magnitude at the origin: All eigenfunctions of \eqref{WIGNEREIG3} with the exception of the radial eigenfunction  $W_{\hbar, E_N(\hbar)}(0,0)$ vanish at the origin $(0,0)$ since they transform by non-trivial characters
  of $U(d)$ and $(0,0)$ is a fixed point of the action. Consequently, the value of the eigenspace projection on the diagonal at $(0,0)$ is the square of
   $W_{\hbar, E_N(\hbar)}(0,0)$ and that accounts precisely for the order of growth.
  

\subsection{Sums of eigenspace projections}

Let us begin by introducing the three types of spectral localization we are studying and the interfaces in each type. \bigskip
 
   \begin{itemize}
   
   \item (i)  $\hbar$-localized Weyl sums over eigenvalues in an $\hbar$-window  $E_N(\hbar) \in  [E - a \hbar, E + b \hbar] $ of width $O(\hbar)$. More generally we consider smoothed Weyl sums $W_{\hbar, E, f} $  with weights $f(\hbar^{-1}(E_N(\hbar) - E))$; see \eqref{UhfDEF} for such $\hbar$-energy localization. This is the scale of individual spectral projections but is substantially more general than the results of \cite{HZ19}. The scaling and asymptotics are in Theorem \ref{ELEVELLOC}. For general \Sch operators, $\hbar$- localization around a single energy level leads to expansions in terms of periodic orbits. Since all orbits of the classical isotropic oscillator are periodic, the asymptotics may be stated without reference to them. The generalization to all \Sch operators  will be studied in a future article. \bigskip


   \item (ii)  Airy-type $\hbar^{2/3}$-spectrally localized Weyl sums $W_{\hbar, f, 2/3}(x, \xi)$  
 over eigenvalues in a window $[E - a \hbar^{2/3}, E + a \hbar^{2/3}]$ of width $O(\hbar^{2/3})$. See Definition \ref{INTERDEF} for the precise definition. The levelset $\Sigma_E$ is viewed as the interface. The scaling asympotics of its Wigner distribution across the interface are given in Theorems \ref{RSCOR} and \ref{SHARPh23INTER}. To our knowledge, this scaling has not previously been considered in spectral asymptotics.
\bigskip

   \item (iii) Bulk Weyl sums  $\sum_{N: \hbar(N +\frac{d}{2}) \in [E_1,E_2]} W_{\hbar, E_N(\hbar)}(x, \xi)$ over energies in an $\hbar$-independent  `window'  $[E_1, E_2]$ of eigenvalues;   this `bulk' Weyl sum runs over $\simeq \hbar^{-1}$ distinct eigenvalues; See Definition \ref{BULKDEF}. We are mainly interested in its scaling asymptotics around the interface $\Sigma_{E_2}$ (see Theorem \ref{BULKSCALINGCOR}). However, we also prove that the Wigner distribution  approximates the indicator function of the shell $\{E_1 \leq H \leq E_2\} \subset T^* \R^d$ (see Proposition \ref{pp:PBK-leading}). As far as we know, this is also a new result and many details are rather subtle because of oscillations inside the energy shell. Indeed, the results of \cite{HZ19} show that the indvidual terms in the sum grow like $W_{\hbar, E_N(\hbar)}(x,\xi)\simeq \hbar^{-d+1/2}$ when $H(x,\xi)\in (E_1,E_2).$ Proposition \ref{pp:PBK-leading}, in constrast, shows although the bulk Weyl sums have $\simeq \hbar^{-1}$ such terms, their sum has size $\hbar^{-d}$, implying significant cancellation. \bigskip

   \end{itemize}

  We are particularly interested in `interface asymptotics' of the Bulk Wigner-Weyl distributions $W_{\hbar, f, \delta(\hbar)}$ around the edge (i.e. boundary) of the spectral interval when $(x, \xi)$ is near  the corresponding classical energy surface $\Sigma_E$.  Such edges occur when $f$ is discontinuous, e.g. the indicator function of an interval. In other words, we integrate the empirical measures \eqref{EMPDEF} below over an interval rather than against a Schwartz test function.  At the interface, there is an abrupt change in the asymptotics with a conjecturally universal shape. Theorem \ref{ELEVELLOC} gives the  shape of the interface for $\hbar$-localized sums, Theorem \ref{RSCOR} gives the shape for $\hbar^{2/3}$ localized sums and Theorem \ref{BULKSCALINGCOR} gives results on the bulk sums.   

Our results concern asymptotics of integrals of various types of test functions against  the weighted empirical measures,
\begin{equation} \label{EMPDEF} d\mu_{\hbar}^{(x, \xi)}(\tau): =
\sum_{N =0}^{\infty} W_{\hbar, E_N(\hbar)}(x, \xi) \delta_{E_N(\hbar)}(\tau),
\end{equation} and of  recentered and rescaled versions of these measures
 (see \eqref{mu2/3}  below). A  key property of  Wigner distributions of eigenspace projections \eqref{WIGNERDEF1} is  that the measures \eqref{EMPDEF} are  signed, reflecting the fact that Wigner distributions take both positive and negative values, and are  of infinite mass:

\begin{prop}\label{INFINITEPROP} The signed measures  \eqref{EMPDEF} 
are of infinite mass (total variation norm). On the other hand,  the mass of \eqref{EMPDEF} is finite on any one-sided interval of the form, $[-\infty, \tau]$.
Also,
$\int_{\R} d\mu_{\hbar}^{(x, \xi)}  = 1$ for all $(x, \xi). $\end{prop} 
 Moreover,  the $L^2$ norms of the  terms $W_{\hbar, E_N(\hbar)}$ grows in $N$ like $N^{\frac{d-1}{2}}$. Hence, the measures \eqref{EMPDEF} are highly oscillatory and the summands can be very large. 
   
\begin{figure}
\begin{center}
  \includegraphics[width=.5 \textwidth]{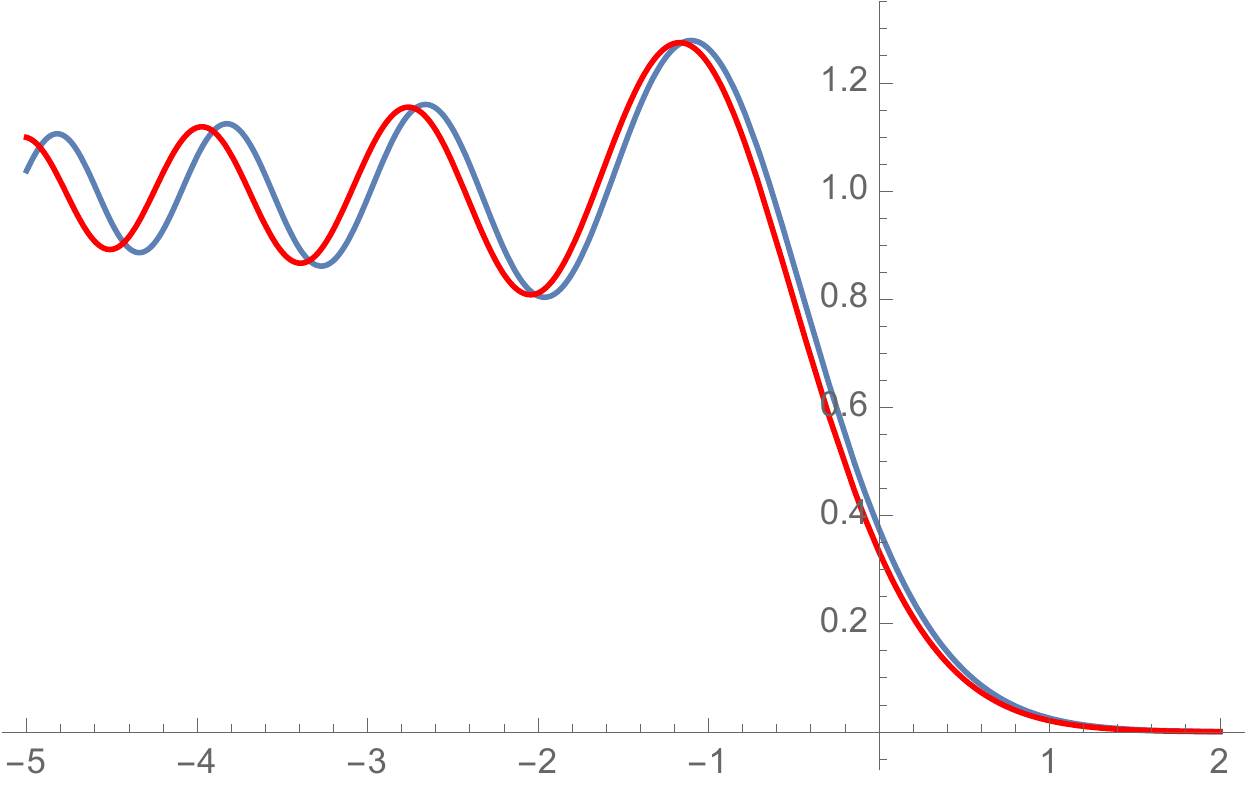} 
\end{center}
\caption{\label{fig:bulk-Wigner-Airy} Plot with $\hbar\approx 0.02,\,E=1/2$ of scaled bulk Wigner-Weyl sum $(2\pi\hbar)^d W_{\hbar, [0,E]}(x,\xi)$ when $H(x,\xi)=E+u(\hbar/2E)^{2/3}$ as a function of $u$ (blue) against its integrated Airy limit $\int_0^\infty \mathrm {Ai}(\lambda + u/E)d\lambda$ (red) from Theorem \ref{BULKSCALINGCOR}. }

\end{figure}




%

\subsection{\label{hSCSECT} Interior asymptotics for  $\hbar$-localized Weyl sums}
The first result we present pertains to the $\hbar$-spectrally localized Weyl sums of type (i), defined by  \begin{equation}
W_{\hbar, E, f}(x, \xi) :=\sum_N f(\hbar^{-1}(E-E_N(\hbar))) W_{\hbar, E_N(\hbar)}(x,\xi),\qquad f \in \scal(\R).\label{UhfDEF}
\end{equation}

\begin{theo} \label{ELEVELLOC} Fix $E>0$, and let $W_{\hbar, E, f} $ be the Wigner distribution as in \eqref{UhfDEF} with $f$ an even Schwartz function. If $H(x,\xi)>E,$ then $W_{\hbar, E,f}(x,\xi)=O(\hbar^\infty).$ In contrast, when $0<H(x, \xi) < E$, set $H_E:=H(x,\xi)/E$ and define
\begin{align*}
t_{+,\pm,k}&:= 4\pi k \pm 2\cos^{-1}\lr{H_E^{1/2}},\quad t_{-,\pm,k}:= 4\pi \lr{k+\frac{1}{2}} \pm 2\cos^{-1}\lr{H_E^{1/2}},\qquad k \in \Z.
\end{align*}
Fix any $\delta>0.$ Then
\[
W_{\hbar, f, E}(x,\xi)=\frac{\hbar^{-d+1}\lr{1+O_\delta(\hbar^{1-\delta})}}{(2E)^{1/2}(2\pi)^dH_E^{d/2}(H_E^{-1}-1)^{1/4}}~\sum_{\pm_1,\pm_2\in \set{+,-}} \frac{e^{\pm_2 i\lr{\frac{\pi}{4}-\frac{4E}{\hbar}}}}{\lr{\pm_1}^d}\sum_{k\in \Z} \widehat{f}(t_{\pm_1,\pm_2,k})e^{\frac{iE}{\hbar}t_{\pm_1,\pm_2,k}},
\]
where the notation $O_\delta$ means the implicit constant depends on $\delta.$
\end{theo}
\noindent Note that there are potentially an infinite number of `critical points' in the support of $\hat{f}$. 

\subsection{\label{SECT23} Interface asymptotics for smooth $\hbar^{2/3}$-localized Weyl sums}
We now consider  spectrally localized Wigner distributions  that are both spectrally localized and phase-space localized on the scale $\delta(\hbar)= \hbar^{2/3}$. They are mainly relevant when we study interface behavior around $\Sigma_E$ of Weyl sums. 

\begin{defn}\label{INTERDEF} Let $H(x, \xi) = (\norm{x}^2 + \norm{\xi}^2)/2$, and assume that $(x, \xi)$ satisfiy
  \begin{equation}
H(x,\xi) =  E + u\lr{\hbar/2E}^{2/3}.\label{E:near-shell}
\end{equation}
Let  $\delta(h) = \hbar^{2/3}$ and define the interface-localized Wigner distributions by
$$\begin{array}{lll} W_{\hbar, f, 2/3}(x, \xi): & = &   \sum_{N} f(h^{-2/3} (E-E_N(\hbar)))W_{ \hbar, E_N}(x, \xi) \;\; \end{array} $$

\end{defn}
\begin{theo} \label{RSCOR}
Assume that $(x, \xi)$ satisfies \eqref{E:near-shell} with $\abs{u}<\hbar^{-2/3}.$ Fix a Schwartz function $f \in \scal(\R)$ with compactly supported Fourier transform. Then
\[W_{\hbar, f, 2/3}(x, \xi)  = (2\pi \hbar)^{-d} I_0(u;f,E)~+~O((1+\abs{u})\hbar^{-d+2/3}),\]
where
\[I_0(u; f, E) = \int_{\R} f(-\lambda/C_E) \mathrm{Ai}\lr{\lambda + \frac{u}{E}} d\lambda,\qquad C_E=(E/4)^{1/3}.\]
More generally, there is an asymptotic expansion
\[W_{\hbar, f, 2/3}(x, \xi) ~\simeq~(2\pi\hbar)^d\sum_{m\geq 0}\hbar^{2m/3}I_m\lr{u;f,E}\]
in ascending powers of $\hbar^{2/3}$ where $I_m(u;f,E)$ are uniformly bounded when $u$ stays in a compact subset of $\R.$
\end{theo}

The calculations show that the results are valid with far less stringent conditions on $f$ than $f \in \scal(\R)$ and $\widehat{f}\in C_0^\infty$. To obtain a finite expansion and remainder it is sufficient that $\int_{\R} |\widehat{f}(t)| |t|^k dt  < \infty$ for all $k.$ It is not necessary that $\hat{f} \in C^k$ for any $k >0$. 

\subsection{Sharp $\hbar^{2/3}$-localized Weyl sums}
Next we consider the sums of Definition \ref{INTERDEF} when $f$ 
is the indicator function of a spectral interval, 
$$f = {\bf 1}_{[ \lambda_- ,\lambda_+]}.$$
Equivalently, we fix integers $0<n_{_\pm}$ such that
\[\lambda_{\pm}=\hbar^{1/3}n_{\pm}~~\text{are bounded},\]  
and consider the corresponding Wigner-Weyl sums $ W_{\hbar, f, 2/3}(x, \xi)$ of Definition \ref{INTERDEF}:
\begin{equation} \label{W23EDEF} W_{2/3,E,\lambda_{\pm}}(x,\xi):=\sum_{N:\, \lambda_- \hbar^{2/3}\leq E_N(\hbar)-E< \lambda_+\hbar^{2/3}}W_{\hbar, E_N(\hbar)}(x,\xi)=\sum_{N=N(E,\hbar)+n_-}^{N(E,\hbar)+n_+-1}W_{\hbar, E_N(\hbar)}(x,\xi)
,\end{equation}
where $N(E,\hbar)=E/\hbar - d/2$. Thus, the sums run over spectral intervals of size $\simeq \hbar^{2/3}$ centered at a fix $E>0$ and consist of sum of $\simeq \hbar^{-1/3}$ Wigner functions for spectral projections of individual eigenspaces.  The following extends Theorem \ref{RSCOR} to sharp Weyl sums at the cost of only giving a 1-term expansion plus remainder.

\begin{theo} \label{SHARPh23INTER} Assume that $(x, \xi)$ satisfies $\lr{\norm{x}^2+\norm{\xi}^2}/2 =E+ u\lr{\frac{\hbar}{2E}}^{2/3}$ with $\abs{u}<\hbar^{-2/3}.$ Then, \[W_{2/3,E,\lambda_{\pm}}(x,\xi)=(2\pi\hbar)^{-d}C_E\int_{-\lambda_+}^{-\lambda_-}\Ai\lr{\frac{u}{E}+\lambda C_E} d\lambda + O\lr{\hbar^{-d+1/3-\delta}+(1+\abs{u})\hbar^{-d+2/3-\delta}},\]
where $ C_E = (E/4)^{1/3}.$
\end{theo}
\noindent Theorem \ref{SHARPh23INTER}  can be rephrased in terms of  weighted empirical measures

\begin{equation} \label{mu2/3} d \mu^{u, E, \frac{2}{3}}_{\hbar} := \hbar^d\; \sum_{N} W_{ \hbar, E_N(\hbar)}\lr{ E + u\lr{\hbar/2E}^{2/3}} \delta_{[ \hbar^{-2/3} (E- E_N(\hbar) )]}.\end{equation} 
 obtained by centering and scaling the family \eqref{EMPDEF}. Thus, for  $(x, \xi)$ satisfying $\lr{\norm{x}^2+\norm{\xi}^2}/2 = E + u\lr{\frac{\hbar}{2E}}^{2/3},$ and for $f \in \scal(\R)$,
$$W_{\hbar, f, 2/3}(x, \xi): = \hbar^{-d} \int_{\R} f(\tau) d \mu^{u, E, \frac{2}{3}}_{\hbar}(\tau), \;\;\; W_{2/3,E,\lambda_{\pm}}(x,\xi) =   \hbar^{-d} \int_{\lambda_-}^{\lambda^+}  d \mu^{u, E, \frac{2}{3}}_{\hbar}(\tau).$$

\subsection{\label{BULKSECT} Bulk sums}
   
We next consider Weyl sums of eigenspace projections corresponding to an energy shell (or window) $[E_1, E_2]$. We consider both sharp and smoothed sums.

\begin{defn} \label{BULKDEF} Define the `bulk' Wigner distributions for an $\hbar$-independent energy window $[E_1, E_2]$ by
\begin{equation} \label{WE} W_{\hbar,[E_1, E_2]}(x, \xi): \sum_{N: E_N(\hbar) \in [E_1,E_2]} W_{\hbar, E_N(\hbar)}(x, \xi). \end{equation} 
More generally for $f \in C_b(\R)$ define
\begin{equation} \label{WEf} W_{\hbar, f}(x, \xi):= \sum_{N=1}^{\infty}
f(\hbar(N +d/2)) \; W_{\hbar, E_N(\hbar)}(x, \xi). \end{equation} 
\end{defn}
\noindent Our first result about the bulk Weyl sums concerns the smoothed Weyl sums $W_{\hbar, f}.$ 

\begin{prop} \label{smoothedbulk}For $f \in \scal(\R)$ with $\hat{f} \in C_0^{\infty}$,  $W_{\hbar, f}(x, \xi)$ admits a complete asymptotic expansion as $\hbar \to 0$ 
of the form,
$$\left\{ \begin{array}{lll} W_{\hbar, f}(x,\xi) & \simeq &  (\pi \hbar)^{-d} \sum_{j = 0}^{\infty} c_{j, f, H}(x, \xi) \hbar^j,\; \rm{with} \\&&\\c_{0, f, H} (x, \xi) 
& = & f(H(x, \xi)) = \int_{\R} \hat{f}(t) e^{i t H(x,\xi)} dt.

\end{array}\right. $$
In general $c_{k, f, H}(x, \xi)$ is a distribution of finite order on $f$ supported
at the point $(x, \xi)$.

\end{prop}
\noindent The proof merely involves
Taylor expansion of the phase. 

\subsection{Interior/exterior asymptotics for bulk Weyl sums of Definition \ref{BULKDEF}}
From Proposition \ref{smoothedbulk}, it is evident that the behavior of $W_{\hbar, [E_1,E_2]}(x,\xi)$ depends on whether $H(x, \xi) \in (E_1, E_2)$ or $H(x, \xi) \notin [E_1, E_2]$. Some of this dependence is captured in the following result. 
\begin{prop}
\label{pp:PBK-leading}  We have,
\[ 
W_{\hbar, [E_1, E_2]}(x, \xi) = \left\{
\begin{array}{ll} (i)\; 
(2\pi\hbar)^{-d} (1+O(\hbar^{1/2})), & 
 H(x, \xi) \in (E_1, E_2),\\ &\\ 
(ii) \;  O(\hbar^{-d +1/2}), & H(x, \xi) < E_1, \\&\\ 
 (iii) \;  O(\hbar^{\infty}), & H(x, \xi) > E_2
\end{array} \right.
\]
\end{prop}

The two `sides' $0 < H(x, \xi) < E_1$ and $H(x, \xi) > E_2$ also behave differently because the Wigner distributions have slowly decaying tails inside an energy ball but are exponentially decaying outside of it. If we write $W_{\hbar, [E_1, E_2]}(x, \xi)  = W_{\hbar, [0, E_2]}(x, \xi) -W_{\hbar, [0, E_1]}(x, \xi) $, we see that the two  cases with $H(x, \xi) > E_1$ are covered by results for $W_{\hbar, [0, E]}$ with $E = E_1 $ or $E = E_2$. When $H(x, \xi) < E_1$, then both terms of $W_{\hbar, [0, E_2]}(x, \xi) -W_{\hbar, [0, E_1]}(x, \xi) $ have the order of magnitude $\hbar^{-d}$ and the asymptotics reflect the cancellation between the terms. The boundary case where $H(x, \xi) = E_1,$ or $H(x, \xi)= E_2$ is special and is given in Theorem \ref{RSCOR}.

\subsection{Interface asymptotics for bulk  Weyl sums of Definition 
\ref{BULKDEF}}

Our final result concerns the asymptotics of $W_{\hbar, [E_1, E_2]}(x, \xi)$ in $\hbar^{\frac{2}{3}}$-tubes  around
the `interface' $H(x, \xi) = E_2$.  Again, it is suffcient to consider intervals $[0, E]$. It is at least intuitively clear that the interface asymptotics will depend only on the individual eigenspace projections with eigenvalues in an $\hbar^{2/3}$-interval around the energy level $E$, and since they add to $1$ away from the boundary point, one may expect the  asymptotics to be similar to the interface asymptotics  for individual eigenspace projections in \cite{HZ19}. 

\begin{theo}\label{BULKSCALINGCOR} 
Assume that $(x, \xi)$ satisfies
$\frac{\abs{x}^2+\abs{\xi}^2}{2} - E = u\lr{\frac{\hbar}{2E}}^{2/3}$ with $\abs{u}<\hbar^{-2/3}.$ Then, for any $\epsilon>0$ 
\[
W_{\hbar, [0, E]}(x, \xi)  =   \lr{2\pi \hbar}^{-d}  \left[\int_0^{\infty} \Ai\lr{\frac{u}{E}+\tau} d\tau +O(\hbar^{1/3-\epsilon}\abs{u}^{1/2})+ O(\abs{u}^{5/2}\hbar^{2/3-\epsilon})\right],
\]
where the implicit constant depends only on $d,\epsilon.$
\end{theo}
The Airy scaling the Wigner function is illustrated in Figure \ref{fig-Wigner-interface}.

\begin{figure}
\begin{center}
  \includegraphics[width=.6 \textwidth]{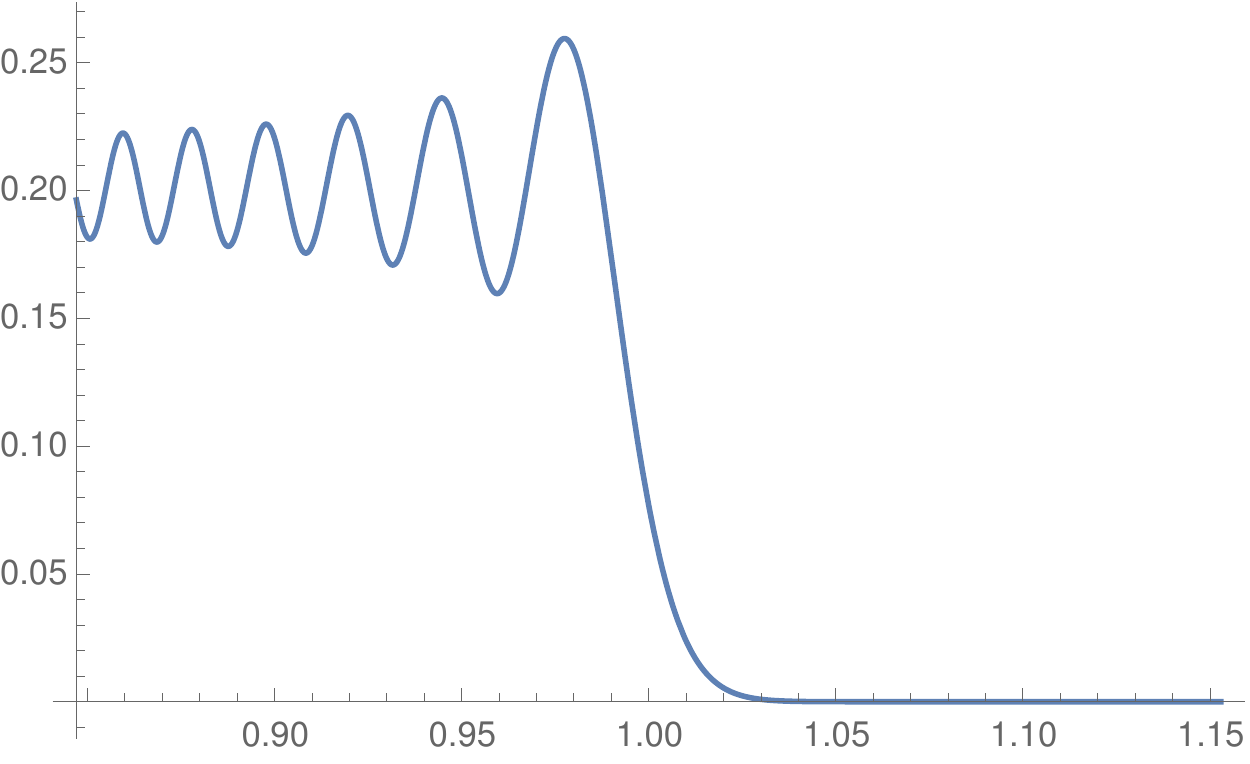} 
\end{center}
\caption{\label{fig-Wigner-interface}
 Scaling at energy surface of Wigner function of projection
  onto energy interval $[0,1/2]$.}
\end{figure}

\subsection{\label{HEURISTICSECT} Heuristics}  Wigner distributions are normalized so that the Wigner distribution of an $L^2$ normalized eigenfunction has $L^2$ norm $1$ in $T^*\R^d$. Due to the  multiplicity $N^{d-1}$ of eigenspaces \eqref{VhE}, the $L^2$ norm of $W_{\hbar, E_N(\hbar)}$ is  of order $N^{\frac{d-1}{2}}$.

  In the main results, we sum over windows of eigenvalues, e.g. $\lambda_- \hbar^{2/3}\leq E-E_N(\hbar)< \lambda_+ \hbar^{2/3}$ \eqref{W23EDEF}, resp. $E_N(\hbar) \in [0,E]$ in \eqref{BULKDEF}. Inevitably, the asymptotics are joint in $(\hbar, N)$. As $\hbar \downarrow 0$, the number of $N$ contributing to the sum  grows at the rate $\hbar^{-\frac{1}{3}}$, resp. $\hbar^{-1}$. Due to the $N$-dependence of the $L^2$ norm, terms with higher $N$ have norms of higher weight in $N$ than those of small $N$ but the precise size of the contribution depends on the position of $(x, \xi)$ relative to the interface $\{H = E\}$ and of course the relation \eqref{ENh}.
  
  $W_{\hbar, E_N(\hbar)}(x, \xi)$ peaks when $H(x, \xi) = E_N(\hbar)$, exponentially decays in $\hbar$ when $H(x, \xi) > E_N(\hbar)$ and has slowly decaying tails inside the energy ball $\{H < E_N(\hbar)\}$, which fall into three regimes: (i) Bessel near $0$, (ii) oscillatory or trigonometric in the bulk, and (iii) Airy near $\{H = E\}.$ In terms of $N$, when \eqref{ENh} holds, and $H(x, \xi)  < E_N(\hbar)$, then $W_{\hbar, E_N(\hbar)}(x, \xi) \simeq \hbar^{-d + 1/2} \simeq N^{d -1/2}$. Near the peak point, when $H(x,\xi)-E_N(\hbar)\approx \hbar^{2/3}$, we have in contrast $W_{\hbar, E_N(\hbar)}(x, \xi) \simeq \hbar^{-d + 1/3} \simeq_E N^{d - 1/3}.$
 
 It follows that the terms with a high value of $N$ and with $ E_N(\hbar) \geq H(x, \xi)$  in  \eqref{EMPDEF}  contribute high weights. There are an infinite number of such terms, and so \eqref{EMPDEF} is a signed measure of infinite mass (as stated in Proposition \ref{INFINITEPROP}.) This is why we mainly consider the restriction of the measures \eqref{EMPDEF} to compact intervals. 

\subsection{Remark on nodal sets in phase space}

In Section \ref{SCH} we discussed nodal sets of random eigenfunctions of the isotropic Harmonic oscillator.  It would also make sense
to consider nodal sets in phase space $T^*\R^d$ for Wigner distributions $W_{\Phi_{\hbar, E}}$ of random eigenfunctions
of the isotropic Harmonic oscillator. This is of interest because Wigner distributions are signed, i.e. not positive, and their
nodal sets and domains signal the extent of this `defect' in their interpretation as phase space densities.  But so far, this
has not been done. However, the covariance function is simply the Wigner distribution of the spectral projection kernels,
so the analysis of Wigner distributions and of their interfaces across energy surfaces provides the necessary techniques. 

In the next section we consider interfaces for partial Bergman kernels. The analogue in the complex domain of random
nodal sets of isotropic oscillator eigenfunctions is zero sets of random homogeneous holomorphic polynomials of fixed
degree in $\C^d$. This is essentially the same as studying such zero sets on complex projective space $\CP^{d-1}$, and to
that extent the theory has already been developed. But interface phenomena for complex zero sets has not so far been
studied.

\section{\label{pBK} Interfaces in phase space: partial Bergman kernels}

In this section, we continue to study phase space distributions of orthogonal projections, but change from the \Sch quantization
to the holomorphic quantization. The holomorphic setting consists of Berezin-Toeplitz operators acting on holomorphic sections
of line bundles over \kahler manifolds, and  is analytically simpler than the real \Sch setting. Hence we are able to present much
more general results. Instead of fixing a model \Sch operator like the isotropic Harmonic Oscillator,  we consider all possible Toeplitz
Hamiltonians on holomorphic sections of Hermitian line bundles $(L, h) \to (M, \omega)$ over all possible projective \kahler manifolds.
Here it is assumed that $i \partial \overline{\partial}  \log h = \omega$, i.e. that $(L, h)$ is a positive, ample line bundle.  For background
on Bargmann-Fock space, and on line bundles over general \kahler manifolds, we refer to Section \ref{BACKGROUND}.

Motivation to study partial Bergman kernels comes from two sources. On the one hand, they arise in many problems of complex
geometry (see \cite{Ber1, HW17,HW18,PS,RS} besides the articles surveyed here). On the other hand, they arise in the IQHE (integer
quantum Hall effect). The author's interest was stimulated by conversations with A. Abanov, S. Klevtsov and P. Wiegmann during a Simons' Center
program on complex geometry and the IQHE. We refer to \cite{W,Wieg,CFTW} for some physics articles where interfaces in the density
of states of  the IQHE are studied. It should be emphasized that there are many types of partial Bergman kernels, and the ones
most interesting in physics are still out of reach of the rigorous techniques described here. What we study here are {\it spectral
partial Bergman kernels}, i.e. orthogonal projection kernels onto spectral subspaces for Toeplitz Hamiltonians. By no means
do all  pBK's (partial Bergman
kernels) arise from spectral problems, but the spectral pBK's are the only types for which there exist general results (or almost
any results) and sometimes the pBK's of interest in the IQHE are spectral pBK's.

We do not review the basic definitions here (see Section \ref{BACKGROUND}) but head straight for the interface results. In place
of the spectral projections of the previous sections, we consider {\it partial Bergman kernels} 
on ``polarized'' \kahler manifolds $(L, h) \to (M^m, \omega, J)$, i.e. \kahler manifolds of (complex) dimension $m$
equipped with a Hermitian holomorphic line bundle  whose curvature form $F_\nabla$ for the Chern connection $\nabla$ satisfies $\omega = i F_\nabla$.
 Partial Bergman
kernels \begin{equation} \label{PIS} \Pi_{k, \scal_k} : L^2(M, L^k) \to \scal_k \subset H^0(M, L^k) \end{equation}
are Schwarz kernels for orthogonal projections onto proper subspaces $\scal_k$ of the holomorphic sections of $L^k$. 

For general subspaces, there is little one can say about the asymptotics of the partial density of states $ \Pi_{k, \scal_k}(z)$, i.e. the
contraction of the diagonal of the kernel. But
for certain sequences $\scal_k$ of subspaces, the partial density of states $ \Pi_{k, \scal_k}(z)$  has an asymptotic expansion as $k \to \infty$ which roughly gives the probability density
that a quantum state from $\scal_k$ is at the point $z$.  More concretely, in terms of  an orthonormal basis $\{s_i\}_{i=1}^{N_k}$ of $\scal_k$, the partial Bergman densities  defined by  \begin{equation} \label{DOS} \Pi_{k,\scal_k}(z) = \sum_{i=1}^{N_k} \|s_i(z)\|^2_{h^k}. \end{equation} When $\scal_k = H^0(M, L^k)$, $\Pi_{k, \scal_k} = \Pi_k: L^2(M, L^k) \to H^0(M, L^k)$ is the orthogonal (\szego
or Bergman) projection. 
We also call the ratio $\frac{\Pi_{k, \scal_k}(z)}{\Pi_k(z)}$ the partial density of states.

Corresponding to $\scal_k$ there is
an allowed region $\acal$ where the relative partial density of states $\Pi_{k, \scal_k}(z) / \Pi_k(z)$ is one, indicating
that the states in $\scal_k$ ``fill up'' $\acal$, and a forbidden region $\fcal$ where the relative density
of states is $O(k^{-\infty})$, indicating that the states in $\scal_k$ are almost zero in $\fcal$.
On the boundary $\ccal: =\partial \acal$ between the two regions there is a
shell of thickness $O(k^{-\half})$  in which the
density of states decays from $1$ to $0$. The $\sqrt{k}$-scaled  relative partial
density of states is asymptotically Gaussian along this interface, in a way reminiscent of the central limit theorem. This was proved in \cite{RS} for certain Hamiltonian holomorphic  $S^1$ actions, then in greater generality in
\cite{ZZ17}. In fact, it is a universal property of partial Bergman kernels defined by $C^{\infty}$ Hamiltonians. 

To begin with, we define the subspaces
$\scal_k$.
 They are defined as spectral subspaces for the quantization of a smooth function $H: M \to \R$. 
By  the standard (Kostant) method of geometric quantization, one can quantize $H$ as the self-adjoint  zeroth order Toeplitz operator 
\begin{equation} \label{TOEP} {H}_k:= \Pi_k (\frac{i}{k} \nabla_{\xi_H} +   H) \Pi_k: 
H^0(M, L^k) \to H^0(M, L^k) \end{equation} 
 acting on the space
$H^0(M, L^k)$ of holomorphic sections.  Here,  $\xi_H$ is the Hamiltonian vector field of $H$, $\nabla_{\xi_H}$ is the Chern covariant deriative on sections,   and $H$ acts by multiplication. We denote the eigenvalues (repeated with multiplicity)
of $\hat{H}_k$ \eqref{TOEP}  by
\begin{equation} \label{EIGV} \mu_{k,1} \leq \mu_{k,2} \leq \cdots \leq \mu_{k,N_k}, \end{equation}
where $N_k = \dim H^0(M, L^k)$, and the corresponding orthonormal eigensections in $H^0(M, L^k)$ by $s_{k,j}$.

Let $E$ be a regular value of $H$. We denote the  partial Bergman kernels for the corresponding spectral
subspaces by
\begin{equation} \label{PBKE} \Pi_{k, E} : H^0(M, L^k)
\to \hcal_{k, E}, \end{equation}
where 
\begin{equation} \label{HEk} \scal_k: = \hcal_{k, E}: = \bigoplus_{\mu_{k,j} < E}
V_{\mu_{k, j}}, \end{equation}
$\mu_{k,j} $ being the eigenvalues of ${H}_k$ and
\begin{equation} \label{EIGSP} V_{\mu_{k,j}}: = \{s \in H^0(M, L^k) : 
{H}_k s = \mu_{k, j} s\}. \end{equation} 
 We denote by $\Pi_{k,j}:H^0(M, L^k) \to V_{\mu_{k,j}}$ the orthogonal projection to $V_{\mu_{k,j}}$. The associated allowed region $\acal$ is the classical counterpart to \eqref{HEk}, and the forbidden region $\fcal$ and the interface $\ccal$ are 
\begin{equation} \label{acalDef} \acal: = \{z: H(z) < E\},  \quad \fcal = \{z: H(z) >E\}, \quad \ccal = \{z:H(z)=E\}. \end{equation}

More generally, for any  spectral interval $I \subset \R$ we define the 
partial Bergman kernels to be the orthogonal projections,
\begin{equation} \label{PBKI} \Pi_{k, I} : H^0(M, L^k)
\to \hcal_{k, I}, \end{equation}
onto the spectral subspace,
\begin{equation} \label{HEI} \hcal_{k, I}: = \z{span}\{ s_{k,j}: \mu_{k,j} \in I \}
 \end{equation}
Its (Schwartz) kernel is defined by  \begin{equation} \label{PPIKDEF} \Pi_{k, I}(z,w) = \sum_{\mu_{k,j} \in I}   s_{k,j}(z) \overline{ s_{k,j}(w)}. \end{equation} 
and the metric contraction of \eqref{PPIKDEF} on the diagonal
with respect to $h^k$ is the partial density of states,
$$\Pi_{k, I}(z)  = \sum_{\mu_{k,j} \in I}  \|s_{k,j,\alpha}(z)\|^2. $$ 
The  classical allowed region $\acal$ and forbidden region $\fcal$ are the open subsets
\[ \acal := \z{Int}(H^{-1}(I)), \quad \fcal = \z{Int}(M \RM \acal), \]
and the interface as
\[ \ccal = \pa \acal = \pa \fcal. \]
In  \cite{ZZ17} it is proved that
\[ \frac{\Pi_{k,I}(z)}{\Pi_k(z)} = \bcs
 1& \text{if } z \in \acal \\
0 & \text{if } z \in \fcal 
\ecs \mod O(k^{-\infty}),
\]

We denote by  $\Pi_k(z,w)$ and $\Pi_k(z)$  the (full) Bergman kernel and density function.  Here and throughout, we use the notation $K(z)$
for the metric contraction of the diagonal values $K(z,z)$ of a kernel.

For each $z \in \ccal$, let $\nu_z$ be the unit normal vector to $\ccal$ pointing towards $\acal$. And let $\gamma_z(t)$ be the geodesic curve  with respect to the Riemannian metric $g(X,Y) = \omega(X,JY)$ defined by the \Kahler form $\omega$, such that $\gamma_z(0) = z, \dot \gamma_z(0) = \nu_z$. For small enough $\delta>0$, the map 
\be \Phi: \ccal \times (-\delta, \delta) \to M, \quad (z, t) \mapsto  \gamma_z(t) \label{normalexp}\ee
is a diffeomorphism onto its image.

\begin{maintheo}\label{MAINTHEO}
Let $(L, h) \to (M, \omega, J)$ be a polarized \Kahler manifold. Let $H: M\to \R$ be a smooth function and $E$ a regular value of $H$. Let  $\scal_k \subset H^0(X, L^k)$ be defined as in  \eqref{HEk}. Then we have the following asymptotics on partial Bergman densities $\Pi_{k, \scal_k}(z)$:
\[   \left( \frac{\Pi_{k, \scal_k}}{\Pi_{k}} \right)(z)  = \bcs
 1& \text{if } z \in \acal \\
0 & \text{if } z \in \fcal
\ecs \mod O(k^{-\infty}).  \]
For small enough $\delta>0$, let $\Phi:\ccal \times (-\delta, \delta) \to M$ be given by \eqref{normalexp}. Then for any $z \in \ccal$ and $t \in \R$, we have
\begin{equation} \label{REM}  \left( \frac{\Pi_{k, \scal_k}}{\Pi_{k}} \right)(\Phi(z,  t/\sqrt{k}))  = \Erf(2\sqrt{\pi} t) + O(k^{-1/2}),  \end{equation}
where $\Erf(x) = \int_{-\infty}^x e^{-s^2/2} \frac{ds}{\sqrt{2\pi}}$ is the cumulative distribution function of the  Gaussian, i.e., $\P_{X \sim N(0,1)}(X<x)$.\footnote{The usual Gaussian  error function $\rm{erf}(x) = (2\pi)^{-1/2}\int_{-x}^x e^{-s^2/2} ds$
is related to Erf by  $\rm{Erf}(x) 
= \half(1 + \rm{erf}(\frac{x}{\sqrt{2}})).$}
\end{maintheo} 

\brem
The analogous result for critical levels is proved in \cite{ZZ18b}. 
We could also choose an interval $(E_1, E_2)$ with $E_i$ regular values of $H$\footnote{It does not matter whether the endpoints are included in the interval, since contribution from the eigenspaces $V_{k, \mu}$ with $\mu = E_i$ are of lower order than $k^m$.}, and define $\scal_k$ as the span of eigensections with eigenvalue within $(E_1, E_2)$. However the interval case can be deduced from the half-ray case $(-\infty, E)$ by taking difference of the corresponding partial Bergman kernel, hence we only consider allowed region of the type in \eqref{acalDef}. 
\erem
\bex \label{BFEX}
As a quick illustration, holomorphic sections of the trivial line bundle over $\C$ are holomorphic functions on $\C$. We equip the bundle with the Hermitian metric where $1$ has the norm-square $e^{-|z|^2}$. The $k$th power has metric  $e^{-k |z|^2}$  Fix  $\epsilon > 0$ and define the subspaces $\scal_k = \oplus_{j \leq \epsilon k} z^j$ of sections vanishing to order at most $\epsilon k$ at $0$, or sections with eigenvalues $\mu < \epsilon$ for operator $ H_k = \frac{1}{ik} \pa_\theta $ quantizing $H=|z|^2$. The full and partial Bergman densities are
\[ \Pi_k(z) = \frac{k}{2\pi}, \quad \Pi_{k, \epsilon}(z) = \kk \sum_{j \leq \epsilon k} \frac{k^j}{j!} |z^j|^2 e^{-k |z|^2}, \]
As $k \to \infty$, we have
\[ \lim_{k \to \infty} k^{-1} \Pi_{k, \epsilon}(z) = \bcs 1 & |z|^2 < \epsilon \\0 & |z|^2 > \epsilon. \ecs \]
For the boundary behavior, one can consider sequence $z_k$, such that $ |z_k|^2 = \epsilon(1+k^{-1/2} u),$ 
\[\lim_{k \to \infty} k^{-1} \Pi_{k, \epsilon}(z_k) =  \Erf(u). \]
\eex
This example is often used to illustrate the notion of  `filling domains' in the IQH (integer Quantum Hall) effect.  In IQH, one considers a free electron gas confined in plane $\R^2 \simeq \C$, with a uniform magnetic field in the perpendicular direction.  A one-particle electron state is said to be in the  lowest Landau level (LLL) if it has the form  $\Psi(z) = e^{-|z|^2/2} f(z)$, where $f(z)$ is holomorphic as in Example \ref{BFEX}.
The following image of the density profile  is copied from \cite{W}, where the picture on the right illustrates how the states $\frac{(\sqrt{k} \;z)^j}{\sqrt{j!}} e^{- k |z|^2/2}$ with $j \leq \epsilon k$ fill the disc of radius $\sqrt{\epsilon}$, so that the density profile drops from $1$ to $0$. 

%
\begin{figure}[h]
\begin{center} 
\includegraphics[width=0.9\textwidth]{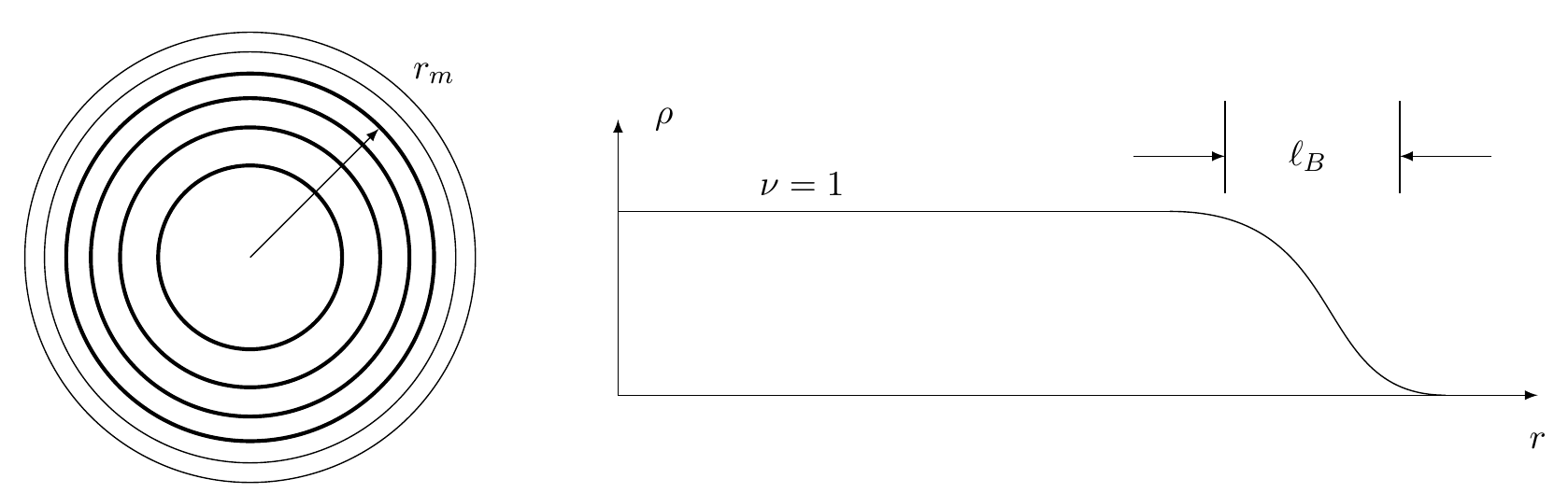} 
\end{center}
\caption{``The density profile of the $\nu=1$ droplet, where the first $m$ levels (represented by the thick lines) are filled." From Fig 7.11 in \cite{W}. }
\end{figure}
  The example is $S^1$ symmetric and therefore the simpler results of \cite{ZZ16} apply. For more general domains $D \subset \C$, it is not obvious how to fill $D$ with LLL states. The Main  Theorem answers the question when $D = \{H \leq E\}$ for some $H$.  For a physics discussion of Erf
  asymptotics and their (as yet unknown) generalization to the fractional
  QH effect, see    \cite{Wieg,CFTW}.

\subsection{\label{3SECT} Three families of measures at different scales}

The rationale for viewing the Erf asymptotics of scaled partial Bergman kernels along the interface $\ccal$ is explained by considering three different scalings of the spectral problem.

\begin{equation} \label{mukzdef} \left\{ \begin{array}{ll} 
(i) &  d\mu_k^z(x)\;\;
=  \sum_{j}
 \Pi_{k,j}(z) \delta_{\mu_{k,j}}(x), \\&\\

(ii) &   d\mu^{z, \half}_{k} (x)
=   \sum_{j}  
  \Pi_{k,j}(z)  \delta_{ \sqrt{k} (\mu_{k,j} - H(z))} (x),  \\  &\\
  
(iii) &   d\mu^{z, 1, \tau}_{k} (x)
=   \sum_{j}  
  \Pi_{k,j}( z)  \delta_{k (\mu_{k,j} - H(z)) + \sqrt{k} \tau} (x), 
  \end{array} \right.  \end{equation}
  where as usual,  $\delta_y$ is the Dirac point mass at $y \in \R$.
We use $\mu(x) = \int_{-\infty}^x d \mu(y)$ to denote the cumulative distribution function. 
  
We view these scalings as analogous to three scalings of the  convolution powers $\mu^{*k}$ of a probability measure $\mu$ supported on $[-1,1]$ (say). The third scaling (iii) corresponds to $\mu^{*k}$, which is supported on $[-k,k]$. The first scaling (i) corresponds to 
 the Law of Large Numbers, which rescales $\mu^{*k}$ back to $[-1,1]$. The second scaling (ii) corresponds to the CLT (central limit theorem) which rescales the measure to $[-\sqrt{k}, \sqrt{k}]$. 
   
Our main results give asymptotic formulae for integrals of test functions
and characteristic functions
against these measures. To obtain the remainder estimate \eqref{REM}, we need to apply semi-classical Tauberian theorems to 
$ \mu^{z, \half}_{k}$ and that forces us to find asymptotics for 
 $\mu^{z, 1,\tau}_{k}$.  

\subsection{Unrescaled bulk results on $d\mu_k^z$}

The first result is that the behavior of the partial density of states in the allowed region $\{z: H(z) < E\}$ is essentially the same as for
the full density of states, while it is rapidly decaying outside this region.

We begin with a simple and general result about partial Bergman kernels for smooth
metrics and Hamiltonians. 

\begin{theointro}\label{AFTHEO}  Let $\omega$ be a $C^{\infty}$ metric on $M$ and let $H \in C^{\infty}(M)$.  Fix a regular value $E$ of $H$ and let $\acal, \fcal, \ccal$ be given by \eqref{acalDef}. 
Then for any $f \in C^\infty(\R)$, we have
\be
\Pi_k(z)^{-1} \int_{-\infty}^E f(\lambda) d\mu_k^z(\lambda) \to \bcs
f(H(z)) &  \rm{if } z \in \acal\\
0 & \rm{if } z \in \fcal.\\
\ecs
\label{TOEPeq1}
\ee
In particular,  the  density of states of the partial Bergman kernel is given by the asymptotic
formula:
\be
\Pi_k(z)^{-1} \Pi_{k, E}(z) \sim 
\bcs
1  \mod O(k^{-\infty})  & \rm{if } z \in \acal\\
0 \mod O(k^{-\infty})  &  \rm{if } z \in \fcal.\\
\ecs 
\label{TOEPeq2}
\ee where the asymptotics are uniform on compact sets of $\acal$ or $\fcal.$
\end{theointro}

In effect, the leading order asymptotics says that the normalized measure
$\Pi_k(z)^{-1} d\mu_k^z \to \delta_{H(z)}$. 
This is a kind of Law of Large Numbers for the sequence $d\mu_k^z $.
The theorem does not specify the behavior of $\mu_k^z(-\infty, E) $
when $H(z) =  E$. The next result pertains to the edge behavior.

\subsection{$\sqrt{k}$-scaling results on $d \mu^{z, 1/2}_{k}$}

 The most interesting  behavior occurs in  $k^{-\half}$-tubes around the interface $\ccal$ between the allowed region $\acal$ and the forbidden region $\fcal$.
 For any $T>0$, the tube
of `radius' $T k^{-\half}$ around $\ccal = \{H = E\}$ is the flowout of $\ccal$ under the gradient flow  of $H$ 
\[ F^t := \exp(t \nabla H) : M \to M, \]
for $|t| < T k^{-1/2}$. 
Thus it suffices to study the partial density of states $\Pi_{k,E}(z_k)$ at points $z_k = F^{\beta/\sqrt{k}}(z_0)$ with $z_0 \in H^{-1}(E). $
The interface result for any smooth Hamiltonian is the same as if the Hamiltonian flow generate a holomorphic $S^1$-actions, and thus our result shows that it is
a universal scaling asymptotics around $\ccal$.

\begin{theointro} \label{RSCOR2} 
 Let $\omega$ be a $C^{\infty}$ metric on $M$ and let $H \in C^{\infty}(M)$.  Fix a regular value $E$ of $H$ and let $\acal, \fcal, \ccal$ be given by \eqref{acalDef}. Let $F^t: M \to M$ denote the gradient flow of $H$ by time $t$. We have the following results: 
 \begin{enumerate}
 \item For any point $z \in \ccal$, any $\beta \in \R$, and any smooth function $f \in C^\infty(\R)$, there exists a complete asymptotic expansion,
\be \label{thm-2-1}  \sum_{j} f(\sqrt{k}(\mu_{k,j} - E))  \Pi_{k,j}( F^{\beta/\sqrt{k}} (z))   \simeq  \kk^m(  I_0  + k^{-\half}
I_{1} + \cdots) , \ee
in descending powers of $k^{\half}$, with the leading coefficient as
\[I_0(f, z, \beta) =   \int_{-\infty}^\infty f(x) e^{-\left(\frac{x  }{|\nabla H|(z)|} - \beta |\nabla H(z)| \right)^2} \frac { dx}{\sqrt{\pi}  |\nabla H (z)|}.\]

\item For any point $z \in \ccal$, and any $\alpha \in \R$, the cumulative distribution function $\mu^{z,1/2}_k(\alpha)=\int_{-\infty}^\alpha d \mu^{z,1/2}_k$ is given by
\be \mu^{z,1/2}_k(\alpha)  = \sum_{\mu_{k,j} < E + \frac{\alpha}{\sqrt{k}}} \Pi_{k,j}(z) =   \kk^m \Erf\left( \frac{ \sqrt{2} \alpha}{|\nabla H(z)|} \right) + O(k^{m-1/2}). \label{eqmuhalf} \ee

\item For any point $z \in \ccal$, and any $\beta \in \R$, the Bergman kernel density near the interface is given by
\begin{equation} \label{REMEST}  \Pi_{k,E}(F^{\beta/\sqrt{k}}(z))=  \sum_{\mu_{j,k} <  E}  \Pi_{k,j}( F^{\beta/\sqrt{k}} (z)) =  \kk^m\Erf \left( -\sqrt{2}\beta |\nabla H(z)| \right)+ O(k^{m-1/2}).\end{equation}

 \end{enumerate}

\end{theointro}

\begin{rem}
The  leading power $\kk^m$ is the same as in Theorem \ref{AFTHEO}, despite the
fact that we sum over a packet of eigenvalues of width  (and cardinality) $k^{-\half}$ times the width (and cardinality) in Theorem
\ref{AFTHEO}. This is because the summands $ \Pi_{k,j}(z)$ already
localize the sum to $\mu_{k,j}$ satisfying $|\mu_{k,j} - H(z)| < C k^{-\half}$. 
\end{rem}

%


\subsection{Energy level localization and $d \mu^{z, 1,\alpha}_{k}$}

To obtain the remainder estimate for the $\sqrt{k}$ rescaled measure $d\mu^{z,1/2}_k$ in \eqref{eqmuhalf} and \eqref{REMEST} , we apply the Tauberian theorem. Roughly speaking, one approximate $d\mu^{z,1/2}_k$ by convoluting the measure with a smooth function $W_h$ of width $h$, and the difference of the two is proportional to $h$. The smoothed measure $d \mu^{z,1/2}_k * W_h$ has a density function, the value of which can be estimated by an integral of the propagator $U_k(t, z, z)$ for $|t| \sim k^{-1} / (h k^{-1/2})$. 
Thus if we choose $h = k^{-1/2}$, and $W_h$ to have Fourier transform supported in $(-\epsilon, +\epsilon)/h$, we only need to evaluted $U_k(t, z,z)$ for $|t| < \epsilon$, where $\epsilon$ can be taken to be arbitrarily small.

\begin{theointro} \label{ELLSMOOTH} 
Let $E$ be a regular value of $H$ and $z \in H^{-1}(E)$. If $\epsilon$ is small enough, such that the Hamiltonian flow trajectory starting at $z$ does not loop back to $z$ for time $|t| < 2\pi \epsilon$, then for any Schwarz function 
 $f \in \scal(\R)$ with $\hat{f}$ supported in $(-\epsilon, \epsilon)$ and $\h{f}(0) = \int f(x) dx = 1$, and for any $\alpha \in \R$ we have 
\[
\int_\R f(x) d \mu^{z, 1, \alpha}_k(x) = \kk^{m-1/2} e^{- \frac{\alpha^2}{\|\xi_H(z)\|^2}} \frac{\sqrt{2}}{2\pi \|\xi_H(z)\|}(1+ O(k^{-1/2})).
\]
\end{theointro}


\subsection{Critical levels}In this section we consider interfaces at critical levels.
 Let $H: M \to \R$ be a smooth function with Morse critical points.Henceforth, to simplify notation,
we use \kahler  local coordinates $u$ centered at $z_0$ to  write points in the $k^{-\epsilon}$ tube around $\ccal$ by 
\[ z = z_0 + k^{-\epsilon} u := \exp_{z_0}(k^{-\epsilon} u ), \quad u \in T_{z_0}. \ccal \]
The abuse of notation in dropping the higher order terms of the normal
exponential map is harmless since we are working so close to $\ccal$. 
At regular points $z_0$ we may use the exponential map along $N_{z_0} \ccal$ but we also want to consider critical points.  More generally we write $z_0+u$ for the point with \kahler normal  coordinate $u$. In these coordinates, 
\[ \omega(z_0+u) = i \sum_{j=1}^m du_j \wedge d \b u_j + O(|u|). \]
We also choose a local frame $e_L$ of $L$ near $z$, such that the corresponding $\varphi = -\log h(e_L, e_L)$ is given by
\[ \varphi(z_0+u) = |u|^2 + O(|u|^3).\] 
See \cite{ZZ17} for more on such adapted frames and Heisenberg coordinates.

Clearly, the formula \eqref{thm-2-1} breaks down at critical points and near such points  on critical levels.   
Our main goal in this paper is to generalize the interface asymptotics to  the case when the Hamiltonian
is a Morse function and the  interface $\ccal = \{H = E\}$ is a critical level, so that $\ccal$ contains a  non-degenerate critical point $z_c$ of $H$.
To allow for non-standard scaling asymptotics, we study  the smoothed partial Bergman density near the critical value $E=H(z_c)$,
\[ \Pi_{k,E, f,\delta}(z) := \sum_{j} \|s_{k,j}(z)\|^2 \cdot f(k^\delta(\mu_{k,j}-E)) \]
where $f \in \scal(\R)$ with Fourier transform $\h f \in C^\infty_c(\R)$, and $0\leq \delta \leq 1$.  This is the smooth analog of summing over eigenvalues within $[E-k^{-\delta}, E+k^{-\delta}]$.

The behavior of the scaled density of states is encoded in the following 
measures,
\begin{equation} \label{2SpM} \left\{ \begin{array}{l}  d \mu_k^z (x) = \sum_{j} \|s_{k,j}(z)\|^2 \,\delta_{\mu_{k,j}}(x), \\ \\
 d \mu_k^{z,\delta} (x) = \sum_{j} \|s_{k,j}(z)\|^2 \,\delta_{k^{\delta}(\mu_{k,j}-H(z))}(x), \\ \\ 
 d \mu_k^{(z, u, \epsilon),  \delta} (x) = \sum_{j} \|s_{k,j}(z+k^{-\epsilon} u)\|^2 \,\delta_{k^{\delta}(\mu_{k,j}-H(z))}(x). \end{array} \right. \end{equation}
For each measure $\mu$ we denote by  $d \h \mu$  the normalized probability measure 
\[ d \h \mu(x) = \mu(\R)^{-1} d \mu(x). \]

For all $z \in M$, we have the following weak limit, reminiscent of the
law of large numbers; 
\[   \h \mu_k^z (x)  \rightharpoonup  \delta_{H(z)}(x). \]
For $z \in M$ with $dH(z) \neq 0$, \eqref{thm-2-1} shows that 
\[   \h \mu_k^{z,1/2} (x)    \rightharpoonup  e^{-  \frac{x^2}{|dH(z)|^2}} \frac{dx}{\sqrt{\pi} |dH(z)|}. \]

\subsection{Interface asymptotics at critical levels}

The next result generalizes the ERF scaling asymptotics to the critical point case. We use the following setup: Let $z_c$ be a non-degenerate Morse critical point of $H$, then for small enough $u \in \C^m$, we denote the Taylor expansion components by
\[ H(z_c+u) = E + H_2(u) + O(|u|^3). \]
where 
\[ E = H(z_c), \quad H_2(u) = \half \Hess_{z_c} H(u,u). \] 

\begin{theo} \label{CRITERF}
For any $f \in \scal(\R)$ with $\h f\in C^\infty_c(\R)$,  we have
\[ \Pi_{k,E,f, 1/2}(z_c + k^{-1/4} u) := \sum_{j} \|s_{k,j}(z_c + k^{-1/4} u)\|^2 \cdot f(k^{1/2}(\mu_{k,j}-E)) = \kk^m f(H_2(u)) + O_f(k^{m-1/4}). \]
More over, the normalized rescaled pointwise spectral measure 
\[ d \h \mu_k^{(z_c, u, 1/4), 1/2}(x):= \frac{\sum_{j} \|s_{k,j}(z_c+k^{-1/4} u)\|^2 \,\delta_{k^{1/2}(\mu_{k,j}-E)}(x)}{\sum_{j} \|s_{k,j}(z_c+k^{-1/4} u)\|^2}\] converges weakly
\[  \h \mu_k^{(z_c, u, 1/4), 1/2}(x)  \rightharpoonup \delta_{H_2(u)}(x). \]

\end{theo}
We notice that the scaling width has changed from $k^{-\half}$ to $k^{-1/4}$
due to the critical point. 
The difference in scalings raises the question of what happens if 
we scale by $k^{-\half}$ around a critical point. The result is stated
in terms of the metaplectic representation on the osculating Bargmann-Fock
space at $z_c$.

\begin{theo} \label{p:crit}
Let $1 \gg T>0$ be small enough, such that there is no non-constant periodic orbit with periods less than $T$. Then 
for any $f \in \scal(\R)$ with $\h f \in C^\infty_c((-T, T))$,  we have
\[ \Pi_{k, E, f, 1}(z_c + k^{-1/2} u) = \kk^{m} \int_\R \h f(t) \ucal(t, u) \frac{dt}{2\pi} + O(k^{m-1/2}) \]
where $\ucal(t,u)$ is the metaplectic quantization of the Hamiltonian flow of
$H_2(u)$ defined as
\[ \ucal(t, u) = (\det P)^{-1/2} \exp( \b u(P^{-1} - 1) u + u \b Q P^{-1} u/2 - \b u P^{-1} Q \b u /2). \]
Here $P=P(t), Q=Q(t)$ be complex $m \times m$ matrices such that if $u(t) = \exp(t \xi_{H_2}) u$, then
\[ \bma u(t) \\ \b u(t) \ema = \bma P(t) & Q(t) \\ \b Q(t) & \b P(t) \ema \bma u \\ \b u \ema. \]

\end{theo}

\brem 
Unlike the universal $\Erf$ decay profile in the $1/\sqrt{k}$-tube around the smooth part of $\ccal$, we cannot give the decay profile of $\Pi_{k,I}(z)$ near the critical point $z_c$. The reason is that there are eigensections that highly peak near $z_c$ and with eigenvalues clustering around $H(z_c)$. Hence it even matters whether we use $[E_1, E_2]$ or $(E_1, E_2)$. See the following case where the Hamiltonian action is holomorphic, where the peak section at $z_c$ is an eigensection, and all other eigensections vanishes at $z_c$. 
\erem

The next result pertains to Hamiltonians generating $\R$ actions, as studied in \cite{RS,ZZ16}. The Hamiltonian flow always extends
to a holomorphic $\C$ action. 

\begin{prop}  Assume 
$H$ generate a holomorphic Hamiltonian $\R$ action. The pointwise spectral measure $d \mu_k^{z_c}(x)$ is always a delta-function 
\[  \mu_k^{z_c} = \delta_{H(z_c)}(x), \quad \forall k =1,2 \cdots \]
Equivalently,  for any spectral interval $I$,
\[ \lim_{k \to \infty}\Pi_{k,I}(z_c) = \bcs 1 & E \in I \\ 0 & E \notin I \ecs. \]
\end{prop}

The above result follows immediately from: 
\bp \label{HOLOPROP}
Let $z_c$ be a Morse critical point of $H$, $E = H(z_c)$. Then 
\bnum
\item The $L^2$-normalized peak section $s_{k, z_c}(z) = C(z_c)  \Pi_k(z, z_c)$ is an eigensection of $\h H_k$ with eigenvalue $H(z_c)$. And all other eigensections orthogonal to $s_{k,z_c}$ vanishes at $z_c$. 
\item If $s_{k,j}  \in H^0(M, L^k)$ is an eigensection of $\h H_k$ with eigenvalue $\mu_{k,j} < E$, then $s_{k,j}$ vanishes on $W^+(z_c)$. 
\item If $s_{k,j}  \in H^0(M, L^k)$ is an eigensection of $\h H_k$ with eigenvalue $\mu_{k,j} > E$, then $s_{k,j}$ vanishes on $W^-(z_c)$. 
\enum
\ep
In particular, this shows the concentration of eigensection near $z_c$. Depending on whether the spectral inteval $I$ includes boundary point $H(z_c)$ or not, the partial Bergman density will differ by a large Gaussian bump of height $\sim k^m$.

\ss{Sketch of Proof}
As in \cite{ZZ17,ZZ18} the proofs involve rescaling  parametrices
for the propagator 
\begin{equation} \label{Ukt} U_k(t) = \exp i t k \hat{H}_k \end{equation}
of the Hamiltonian \eqref{TOEP}. The parametrix construction is reviewed 
in Section \ref{TQD}.
We begin by  observing that for all $z \in M$, the time-scaled propagator has pointwise
scaling asymptotics with the $k^{-\half}$ scaling:
\bp[\cite{ZZ17} Proposition 5.3] \label{ZZ17}
If $z \in M$, then for any $\tau \in \R$,
\[ \h U_k(t/\sqrt{k},\h z , \h z ) = \kk^{m} e^{i t \sqrt{k} H(z )} e^{-t^2 \frac{\| dH(z )\|^2}{4}} (1 + O(|t|^3k^{-1/2})), \]
where the constant in the error term is uniform as $t$ varies over compact subset of $\R$. 
\ep
\noindent The condition $dH(z) \neq 0$ in the original statement in \cite{ZZ17} is never used in the proof,  hence both statement and proof carry over
to the critical point case. We therefore omit the proof of this Proposition.

 We also give asymptotics for the trace of the  scaled propagator   $U_k(t/ \sqrt{k})$. It is based on stationary phase asymptotics and therefore also
 reflects the structure of the critical points.
\begin{theo}\label{TR}
If $t \neq 0$, the trace of the scaled propagator $U_k(t/ \sqrt{k}) = e^{i \sqrt{k} t \h H_k}$ admits the following aymptotic expansion
$$\begin{array}{lll}  \int_{z \in M} U_k(t/\sqrt{k}, z) d \Vol_M(z)  & = & \kk^m  (\frac{t \sqrt{k}}{4 \pi})^{-m} \sum_{z_c \in \z{crit}(H)} \frac{e^{i t \sqrt{k} H(z_c)}e^{(i\pi/4) \z{sgn}(\Hess_{z_c}(H))}}{\sqrt{|\det(\Hess_{z_c}(H))|}} \\&&\\&& \cdot (1 + O(|t|^3k^{-1/2})) \end{array} $$
where  $\z{sgn}(\Hess_{z_c}(H))$ is the signature of the Hessian, i.e. the number of its positive eigenvalues minus the number of its  negative eigenvalues.
\end{theo}

\section{\label{BFINTER} Interfaces for the Bargmann-Fock isotropic Harmonic oscillator}

We continue the discussion of Bargmann-Fock space from Section \ref{BF} by considering partial Bargmann-Fock Bergman
kernels.
In this section, we tie together the results on Wigner distributions of spectral projections for the isotropic Harmonic oscillator,
and on density of states for partial Bergman kernels associated to the natural $S^1$ action on Bargmann-Fock space. This is the most direct analogue of the \Sch 
results.

The classical Bargmann-Fock isotropic Harmonic oscillator corresponds to the degree operator on $H^0(\CP^m, \ocal(N))$. The
total space of the associated line bundle is $\C^{m+1}$. The harmonic operator
generates the standard diagonal  $S^1$ action on $\C^{m+1}$,
$$e^{i \theta} \cdot (z_1, \dots, z_{m+1}) = (e^{i n_1 \theta} z_1, \dots, e^{i n_m \theta} z_{m+1}).$$
Its Hamiltonian is $H_{\vec n}(Z) = \sum_{j = 1}^{m+1} n_j |z_j|^2. $ The critical point set of
$H_{\vec n}$ is its minimum set.

The
eigenspaces $\hcal_{k, m, N}$  consist of monomials $z^{\alpha}$ with $|\alpha| = N$. Given
the Planck constant $k$, the eigenspace projection is given by
\begin{equation} \label{HON} \Pi_{h_{BF}^k, N}(Z,W) = \sum_{|\alpha| = N} \frac{ (k Z)^{\alpha}) (k\bar{W})^{\alpha}}{\alpha!}, \end{equation}
as a kernel relative to the Bargmann-Fock Gaussian volume form.
The partial Bergman kernels arising from spectral projections of the isotropic oscillator thus have the form,
$$\Pi_{h_{BF}^k, E}  = \sum_{N: \frac{N}{k} \geq E} \Pi_{h_{BF}^k, N}(Z,W). $$

We claim that the eigenspace projector \eqref{HON} satisfies,
\begin{equation} \label{HON2} \Pi_{h_{BF}^k, N}(Z,Z) = C_{N, k,m} ||Z||^{2 N}, \end{equation}
where
$$C_{N,k,m} = \frac{p(N, m +1)}{\omega_m} \frac{k^{N }}{\Gamma(N  + m +1)}. $$
Here, $\omega_m = \rm{Vol} (S^{2m+1})$ is the surface are of the unit sphere in 
$\C^{m+1}$. Also,   
 $\dim \hcal_{k, m, N} = p(m +1, N), $
 the partition function which counts the number of ways to express $N$ as a sum
 of $m +1$ positive integers.
To prove this, we first observe that the $U(m +1)$-invariance of the Harmonic oscillator
Hamiltonian $H = ||Z||^2$ implies that 
$U^*  \Pi_{h_{BF}^k, N} U =  \Pi_{h_{BF}^k, N}$ and therefore
$ \Pi_{h_{BF}^k, N}(UZ, UZ) =  \Pi_{h_{BF}^k, N}(Z,Z).$ It follows that
$ \Pi_{h_{BF}^k, N}(Z,Z) =  F(||Z||^2)$ is radial. It is also homogeneous of degree $2 N$,
hence is a constant multiple $C_{N,k,m} ||Z||^{2 N}$ as claimed in \eqref{HON2}. The constant
is calculated from the fact  that
$$\begin{array}{lll} p(m, N) = \dim \hcal_{k,m,N} & = & \frac{k^{m+1}}{(m+1)!}  \int_{\C^{m+1}}  \Pi_{h_{BF}^k, N}(Z,Z) e^{-k ||Z||^2} dL(Z) \\&&\\& = & \omega_m C_{m}  k^{m+1}\int_0^{\infty} e^{- k \rho^2} \rho^{2 N} \rho^{2m+1} d \rho
 \\&&\\& = & \half \omega_m C_{m} k^{m+1} \int_0^{\infty} e^{- k \rho} \rho^{ N} \rho^{m} d \rho
 = \half \frac{k^{m+1}}{(m+1)!}  k^{- (N + m + 1)}\omega_m C_{m, k, N} \Gamma(N + m + 1)  . \end{array}$$
 Solving for $C_{m,k,N}$ establishes the formula.
 It also follows that the density of states is given by,
 
\begin{equation} \label{BFPBKa} \begin{array}{lll} \sum_{N \geq \epsilon k} \Pi_{h_{BF}^k, N}(Z,Z) &= &  \frac{k^{m+1}}{(m-1)! \omega_m} e^{- k ||Z||^2} \sum_{N \geq \epsilon k}
\frac{ (k ||Z||^2)^{ N} p(m, N)}{\Gamma(N + m +1)} \\&&\\& \simeq  & \frac{k^{m+1}}{(m-1)! \omega_m} e^{- k ||Z||^2} \sum_{N \geq \epsilon k}
 \frac{ (k ||Z||^2)^{ N} }{N!}, \end{array} \end{equation}
since   $p(m +1, N) \simeq \frac{1}{(m+1)!} N^{m}(1 + O(N^{-1}))$ \eqref{dimV}; also, $\Gamma(N + m +1) = (N + m)! \simeq (N + m) \cdots (N +1) N! \simeq N^m N!$. 

\section{\label{BFLB} Bargmann-Fock space of a line bundle and interface asymptotics }

In this section, we introduce a new model, the Bargmann-Fock space of an ample  line bundle $\pi: L \to M$ over a \kahler manifold, and generalize the results of the preceding section to 
 density of states for partial Bergman kernels associated to the natural $S^1$ action on the total space $L^*$ of
the dual line bundle.   We let
$X_h = \partial D_h^* \subset L^*$ be the unit $S^1$-bundle given by the boundary 
of the unit codisc bundle,  $D_h^* = \{(z, \lambda) \in L^*: |\lambda|_z < 1\}. $
We sketch the proof that  `interfaces' for the Hamiltonian generating the standard $S^1$ action on the Bargmann-Fock
space of $L$ satisfy the central limit theorem or cumulative Gaussian Erf 
interfaces as in the compact case of \cite{ZZ16}.  The Hamiltonian is simply the norm-square function $N(z, \lambda): = |\lambda|_{h_z}^2$, so the energy balls are simply the co-disc bundles 
$$D_E^* = \{(z, \lambda) \in L^*: |\lambda|_{h_z} \leq E^2\}.$$

As usual, we equivariantly lift sections $s_k \in H^0(M, L^k)$ to $ \hat{s}_k \in \hcal_k(L^*)$, which are homogeneous of degree $k$ in the 
sense that 
$$\hat{s}_k(rx) = r^k \hat{s}_k(x). $$

 \subsection{Volume forms}
  $X_h$ is a contact manifold with contact volume form $dV = \alpha \wedge (\pi^* \omega)^m$. This contact volume form 
  induces a volume form $dVol_{L^*}$ on $L^*$, generalizing the Lebesgue volume form  $dVol_{\C^m}$ in the standard Bargmann-Fock space.
 Namely, the \kahler metric $\omega_h$ of the Hermitian metric $h$ on $L$ lifts to the partial \kahler metric $\pi^* \omega_h$.  Then,
 $$\omega_{L^*} = \pi^* \omega_h + d \lambda \wedge d \bar{\lambda}$$
 is a \kahler metric on $L^*$ with potential $|\lambda|^2 e^{- \phi}$ where $\phi = \log |e_L|^2_{h_z}$ is the local \kahler potential
 on $M$.
 Since $L^* \simeq X_h \times \R_+$ we may use polar coordinates $(x, \rho)$ on $L^*$, which correspond to 
 coordinates $(z, \lambda) \in M \times \C$ in a local trivialization by $\rho = |\lambda|_{h_z}$ and $x = (z, e^{i \theta})$. Since
 $\dim_{\R} X = 2m +1$ when $\dim_{\C} M = m$, the   
 volume form on $L^*$ is given by $$dVol_{L^*}(x, \rho) = \rho^{2m +1} dV(x) d \rho. $$


We then endow $L^*$ with the (normalized) Gaussian measure analogous to \eqref{BFG},  \begin{equation} \label{BFLBG} 
d \Gamma_{m+1, \hbar} :=  \frac{ \hbar^{-(m+1)}}{\rm{Vol}(X_h) \Gamma(m+1)}
e^{- ||Z||^2/\hbar} dVol_{L^*}(Z) \end{equation}
 To check that the measure has mass $1$, we note that
 $$\int_{L^*}e^{- ||Z||^2/\hbar} dVol_{L^*}(Z)  = \rm{Vol}(X_h) \int_0^{\infty} e^{-\rho^2/ \hbar} \rho^{2m+1} d\rho
 = \rm{Vol}(X_h) \hbar^{m+1} \Gamma(m +1) . $$ Here, we denote a general point of $L^*$ by
$Z = \rho x$ with $\rho \in \R_+, x \in X_h$. In the future we put 
$$C_m(h) = \frac{1}{\rm{Vol}(X_h) \Gamma(m+1)}, $$
so that we do not have to keep track of this constant.

\begin{defin} The Bargmann-Fock space of $(L, h)$ is the Hilbert space
$$\hcal^2_{BF, \hbar}(L^*):= \bigoplus_{N=0}^{\infty} \hcal_N(L^*)$$
of  entire square integrable holomorphic functions on $L^*$
with respect to the inner product 
\begin{equation}
 \label{IPF} ||f||^2_{\hbar, BF} =: \frac{ \hbar^{-(m+1)}}{\rm{Vol}(X_h) \Gamma(m+1)}
 \int_{L^*} |f(Z)|^2 e^{- ||Z||^2/\hbar} dVol_{L^*}(Z).  \end{equation}
\end{defin}

\subsection{Orthonormal basis} 
If  $s \in H^0(M, L^k)$ then \begin{equation} \label{IP} ||\hat{s}_k||_{L^2(X_h)}  = \frac{1}{m!}
 \int_{X_h}  |\hat{s} (x )|^2dV(x)  = \int_M ||s(z)||_{h^k}^2 dV_{\omega}, \;\;   \end{equation}
 where the right side is the   inner product on $H^0(M, L^k)$, where $dV_{\omega} = \omega^m/m!$.
Let $N_k = \dim H^0(M, L^k)$ and let $\{\hat{s}_{k, j}\}_{j=1}^{N_k}$
be any orthonormal basis of $\hcal_k(L^*)$, corresponding to an orthonormal basis $\{s_{k, j}\}$ of $H^0(M, L^k)$.  We let $\hbar = k^{-1}$. We also change the notation for powers of a bundle $k \to N$ to agree with the notation for the real
Harmonic oscillator but retain the notation $\hbar = k^{-1}$. Thus, in effect, there are two semi-classical parameters: $N$ and $k$,
parallel to the parameters $N$ and $\hbar^{-1}$ for the \Sch representation of the harmonic oscillator.
 The lifts $\hat{s}_{N, j}$ of an orthonormal basis $s_{N, j}$ of $H^0(M, L^N)$ are
 orthogonal but no longer normalized. 
 
 \begin{lem} There exists a constant $c_m = (\rm{Vol}(X_h) \Gamma(m+1))^{-\half}$ so that
 $\{ c_m \hbar^{-N/2}\frac{\hat{s}_{N,j} (Z)}{\sqrt{(N + m + 1)!}}\}$ is an orthonormal basis
of $\hcal^2_{BF}$. \end{lem}

\begin{proof} We have, $$\begin{array}{lll} ||\hat{s}_N||^2_{BF, \hbar} & = &  \;||\hat{s}||_{L^2(X_h)}^2 C_m \hbar^{-(m +1)} \int_0^{\infty} e^{-\rho^2/\hbar} \rho^{2N + 2m +1}  d \rho,  \\ &&\\&&= C_m  ||\hat{s}||_{L^2(X_h)}^2  \hbar^{N}
 \Gamma(N+ m +1) = C_m \hbar^{N}\; (N + m)! ||\hat{s}||_{L^2(X_h)}^2, \end{array} $$
 since $\hbar^{-(m+1)} \int_0^{\infty} e^{-\rho^2/\hbar} \rho^{2N + 2m +1}  d \rho = \hbar^{N}
 \Gamma(N+ m + 1)$. 
 Putting $c_m = C_m^{-\half}$ completes the proof.

 \end{proof}

 \begin{cor} \label{BFCOR} In the notation above, an  orthonormal basis of $\hcal^2_{BF, \hbar}(L^*)$ is given by
 $\{c_m\hbar^{-\frac{N}{2}} \frac{\hat{s}_{N, j} }{\sqrt{(N+ m)!}} \}$. \end{cor}

\subsection{Bargmann-Fock Bergman kernel of a line bundle}

We now define the Bargmann-Fock Bergman kernel:
\begin{defin} \label{BFDEF} The Bargmann-Fock Bergman kernel is the kernel of the  orthogonal projection, 
$$\hat{\Pi}_{BF, \hbar}: L^2(L^*) \to \hcal_{BF}(L^*), $$
with respect to the Gaussian measure $\Gamma_{m+1, \hbar}$ of the inner product  \eqref{IPF}. The density of states is the
positive measure,
$$\hat{\Pi}_{BF, \hbar}(Z, Z) d \Gamma_{m+1, \hbar}(Z) $$
\end{defin}

Let $\Pi_{h^N}: L^2(M, L^N) \to H^0(M, L^N)$ be the orthogonal projection
with respect to the inner product \eqref{IP}. It lifts to the orthogonal
projection $\hat{\Pi}_N: L^2(X_h) \to \hcal_N(X_h)$ with respect to  the inner product on $L^2(X_h)$ defined by \eqref{IP}. Again by  \eqref{IPF}, $\hat{\Pi}_N$ is equal up to the constant $C_N$  to the orthogonal
 projection $\hcal^2_{BF}(L^*) \to \hcal_N$. The next Lemma is an immediate consequence of Corollary \ref{BFCOR}.
 \begin{lem} \label{BFB1} The Bargmann-Fock Bergman kernel on $\hcal^2_{BF}(L^*)$
 is given for $Z = (z, \lambda), W = (w, \mu) \in L^*$ by
 $$\begin{array}{lll} \hat{\Pi}_{BF, \hbar}(Z, W)  : = c_m
  \sum_{N=0}^{\infty} \frac{\hbar^{-N}}{ (N + m)!} \hat{\Pi}_N(Z, W)=  C_m 
 \sum_{N=0}^{\infty} \hbar^{-N}\frac{ (\lambda \overline{\mu})^N}{ (N + m )!} \hat{\Pi}_N(z, 1, w,1), \;\; \end{array}$$
 where the equivariant kernel $\hat{\Pi}_N$ on $X_h$ is extended by
 homogeneity to $L^*$. The density of states is given by
 $$\begin{array}{lll} \hat{\Pi}_{BF,\hbar}(Z, Z)  e^{- ||Z||^2/\hbar} : & = &  c_m 
  \hbar^{-(m +1)}
  e^{- ||Z||^2/\hbar} \sum_{N=0}^{\infty} \frac{\hbar^{-N}}{ (N + m)!} \hat{\Pi}_N(Z, Z)\\ &&\\ & = &  c_m \hbar^{-(m+1)}  e^{- ||Z||^2/\hbar}\sum_{N=0}^{\infty} \hbar^{- N}\frac{ |\lambda|^{2N}}{ (N + m )!} \Pi_{h^N}(z), \;\; \end{array}$$
 where $\Pi_{h^N}(z)$ is the metric contraction of $\Pi_N(z,z)$ on $M$.
 \end{lem}
 
 The following is the main result of this section:
 \begin{prop}\label{BFB} Let $\hbar = k^{-1}$. For $Z = (z, \lambda)$, the density of states equlas $$ \hat{\Pi}_{BF,k}(Z)  : = c_m k^{m+1} e^{- k ||Z||^2} \sum_{N=0}^{\infty} \frac{ |\lambda|^{2N}}{ (N + m )!}  k^{N} N^m [1 + O(\frac{1}{N})]  dVol_{L^*}(Z).$$

 \end{prop}
 
 \begin{proof}
 
 We recall that the density of states admits an asymptotic expansion,
 $$\Pi_{h^N}(z) \simeq \frac{N^m}{m!}[1 + \frac{a_1(z)}{N} + \cdots], $$ so by Lemma \ref{BFB1}, the density of states equals
 $$\hat{\Pi}_{BF,\hbar}(Z, Z) d \Gamma_{m+1, \hbar} : = c_m \hbar^{-(m+1)}   e^{- ||Z||^2/\hbar} \sum_{N=1}^{\infty} \hbar^{-N} \frac{ |\lambda|^{2N}}{ (N + m)!}N^m[1 + \frac{a_1(z)}{N} + \cdots] dVol_{L^*}(Z), $$
 where $C_m$ is a dimensional constant. Substituting $\hbar = k^{-1}$ completes the proof.
 \end{proof}
 
 We note that
 $\frac{N^m}{(N + m )!} \simeq  \frac{1}{N!}$, so that the asymptotics of Proposition \ref{BFB} agree with 
 the Bargmann-Fock case \eqref{BFPBKa}. 

\subsection{Interface asymptotics}

The Hamiltonian is the norm square of the  Hermitian metric itself, i.e.
$$H(z, \lambda) = |\lambda |^2_{h_z}. $$
The sublevel set $\{H \leq E\}$ is the disc bundle of radius $E^2$. We denote its boundary by $\Sigma_E$.   The normal direction to $\Sigma_E$ is  the gradient 
$\nabla H$ direction,   is given by the radial vector
on $L^*$ generated by the natural $\R_+$ action in the fibers dual to the $S^1$ action generated by $H$. Together, the $\R_+$
and $S^1$ actions define the standard $\C^*$ action on $L^*$ and $\nabla H = J \xi_H$ where $\xi_H = \frac{\partial}{\partial \theta}$ is the Hamilton vector
field of $H$. Thus, the asymptotics of such partial Bergman kernels falls into the $\C^*$ equivariant
setting of  \cite{ZZ16}.

We fix $E$ and consider the partial Bargmann-Fock Bergman kernel of $L^*$ with the energy interval $[0, E]$. Then as in the
standard case, the exterior interface asymptotics pertain to the sums,
\begin{equation}\label{pBKL}  \sum_{N \geq \epsilon k} \Pi_{h_{BF}^k, N}(Z,Z) = \frac{k^{m+1}}{\omega_m m!} e^{- k ||Z||^2} \sum_{N \geq \epsilon k}
\frac{ (k ||Z||^2)^{ N} N^m}{(N + m)!} [1 + \frac{a_1(z)}{N} + \cdots],\end{equation}
or to the complementary sums. Comparison with   the standard Bargmann-Fock case of \eqref{BFPBKa} shows that the agree
to leading order, due to the Bergman kernel asymptotics of the summands $\Pi_N(z, 1, z, 1)$.
 The interface asymptotics are therefore the same as on Bargmann-Fock space for the Toeplitz isotropic Harmonic oscillator, and 
are also essentially the same as in Theorem \ref{MAINTHEO}, with
 $H(z, \lambda) = |\lambda|$ and $|\nabla H(z, \lambda)| =  |\frac{\partial}{\partial \theta}| = \lambda$. We refer to orbits of
 the $\R_+$ action as radial orbits.

\begin{theo} \label{thm:interfaceBF} Let $ \Pi_{h_{BF}^k, (E, \infty]} (Z,Z)=  \sum_{N \geq E k} \Pi_{h_{BF}^k, N}(Z,Z)$. Let $Z = (z, \lambda) \in L^*$
and let $Z_E = (z, \lambda_E)  \in \Sigma_E$ with $|\lambda_E|_{h_z} = E$.
Let $Z_k = e^{ \frac{\beta}{\sqrt{k}} } \cdot Z_E = (z, e^{ \frac{\beta}{\sqrt{k}} }  \lambda_E)$ be sequence of points approaching  $(z, \lambda_E)$ along a radial $\R_+$ orbit, where $\beta \in \R$. Then, as $k \to \infty$,
\begin{equation} \label{PkEzk}
\Pi_{h_{BF}^k, (E, \infty]}(Z_k)
 = k^m \Erf\left(\sqrt{ k} \frac{E- e^{ \frac{\beta}{\sqrt{k}} } E}{E}\right) (1+O(k^{-1/2})) = k^m \Erf\left( -\beta \right)(1+ O(k^{-1/2})).   \end{equation}
\end{theo}

The proof of  Theorem \ref{thm:interfaceBF} is essentially the same as for Theorem \ref{MAINTHEO}, or better the same as in \cite{ZZ16}
for the $\C^*$ equivariant case. The only difference is that $L^*$ is of infinite volume, but this does not affect pointwise asymptotics. However,
there is a more elementary proof in this case. 
 
Let  $x = |Z_k|^2 = |\lambda|^2_{h_z} = e^{2 \frac{\beta}{\sqrt{k}}} Z_E$ with $|Z_E| = E$. It is well known that, as $k \to \infty$,
$$e^{-k x} \sum_{N \leq  k E^2} \frac{(kx)^N N^m}{(N + m)!} \;\;\; \sim  \frac{1}{\sqrt{2 \pi} x} \int_{-\infty}^{\sqrt{k}\frac{E^2 -x}{\sqrt{x}}} e^{- \frac{t^2}{2 x}}
dt. $$
 Indeed, Lemma 1 of \cite{XX2} asserts that

\begin{equation}
e^{-k x} \sum_{N =1}^{ x k +y \sqrt{k }} \frac{(kx)^N}{N!}
 \sim  \frac{1}{\sqrt{2 \pi} } \int_{- \infty}^{\frac{y}{\sqrt{x}}} e^{- \frac{t^2}{2 }} dt
 + O(\frac{A x \sqrt{3x + 1}}{\sqrt{k}((\sqrt{x} + y)^3}). 
\end{equation}
We have,   $$\sqrt{x} =  e^{\frac{\beta}{\sqrt{k}}} E  \simeq E + \frac{\beta}{\sqrt{k}} \implies \frac{E^2 - x}{\sqrt{x}} =
\frac{E^2 -e^{2 \frac{\beta}{\sqrt{k}}} E^2}{e^{\frac{\beta}{\sqrt{k}}} E} =-2 E  \frac{\beta}{\sqrt k} (1 + O(\frac{1}{\sqrt{k}})). $$

Then let $k x + y \sqrt{k} = k E^2,$  i.e. $\; \frac{y}{\sqrt{k}} = E^2 - x \simeq 2 E  \frac{\beta}{\sqrt k}$, thus
$y = 2 \beta E $,  and use $\frac{N^m}{(N + m)!} \simeq \frac{1}{N!}$ to obtain  the desired asymptotic. 

To see this asymptotic implies Theorem \ref{thm:interfaceBF}, we let $\sqrt{k}\frac{E -x}{\sqrt{x}} = \beta$ or
$\frac{E -x}{\sqrt{x}} = \frac{\beta}{\sqrt{k}}.$ Then we get
$$\Pi_{h_{BF}^k, (E, \infty]}(Z_k) \simeq k^m e^{-k  e^{\frac{\beta}{\sqrt{k}}} E} \sum_{N \leq  k E^2} \frac{(k  e^{\frac{\beta}{\sqrt{k}}} E)^N N^m}{(N + m)!} \;\;\; \sim k^m \frac{1}{\sqrt{2 \pi} x} \int_{-\infty}^{\beta} e^{- \frac{t^2}{2 E}}
dt \;(1 + O(\frac{1}{\sqrt{k}})$$

\begin{remark}
In  \cite{Sz50}, Szasz introduces the ``Szasz operator'' 
$$P_f(u, x): = e^{-x u} \sum_{n =1}^{\infty} \frac{(ux)^n}{n!} f(\frac{n}{u}), $$
and shows that, for $f \in C_b(\R)$,  $\lim_{u \to \infty} P_f(u, x) = f(x). $ If we let $f(v) = {\bf 1}_{[E, \infty]}(v) $, then $f(\frac{n}{u}) = {\bf 1}_{u \leq n E}$. Szasz's asymptotic does not apply at the point of discontinuity. Later, Mirjakan introduced the ``Szasz-Mirjakan operator'' 
\cite{Mir} 
$$P_{f, N}(u, x): = e^{-x u} \sum_{n =1}^{N} \frac{(ux)^n}{n!} f(\frac{n}{u}), $$
and Omey \cite{O} proved that if $N = N(n, x)$ with $\lim_{n \to \infty} \frac{N- nx}{\sqrt{n}} = C < \infty$ then 
$\lim_{n \to \infty} P_{f, N}(n, x) = \frac{f(x)}{\sqrt{2 \pi}} \int_{-\infty}^C e^{-\half u^2} du. $  \cite[Lemma 1]{XX2} is
a refinement of this limit formula.

\end{remark}


This asymptotic formula arises in
the analysis of  Bernstein polynomials  of  discontinuous functions with a jump, and we refer to   \cite{Ch,Lev, O,Sz50,XX2} for the analysis.

\section{Further types of interface problems}

\subsection{\label{FURTHER} Further types of interface problems}

Here are some further types of interface asymptotics:
\begin{itemize}
\item Entanglement entropy: 
Sharp spectral cutoffs involve indicator functions ${\bf 1}_{E_1, E_2}(\hat{H}_{\hbar})$ of a quantum Hamiltonian. On the other hand, one
might quantize the indicator function ${\bf 1}_{E_1, E_2}(H)$ of a classical Hamiltonian. This is obviously related but different, since the first is a projection and the second is not. Entanglement entropy is a measure
of how the second fails to be a projection and has been studied by
Charles-Estienne \cite{ChE18} and by the author (unpublished).
\bigskip

\item On a manifold $M$ with boundary $\partial M$ one may study
the spectral projections kernel $E^D_{[0, \lambda]}(x,x)$ of the Laplacian 
with Dirichlet boundary conditions. Away from $\partial M$, 
$\lambda^{-n} E^D_{[0, \lambda]}(x,x) \simeq 1$ where $n = \dim M$. Yet
$E^D_{[0, \lambda]}(x,x) = 0$ on $\partial M$. What is the shape of the drop-off from $1$ to $0$ n a boundary zone of width $\lambda^{-1}$?\bigskip

\item For the hydrogen atom Hamiltonian $\hat{H}_{\hbar}$, there is a
phase space interface $\Sigma_0 \subset T^*\R^d$ separating the bound states from the scattering states. The Hamiltonian flow is periodic on one side of $\Sigma_0$ and unbounded on the other side and parabolic on
$\Sigma_0$. The quantization of the bound state region is the discrete
spectral projection $\Pi_{\rm{disc}, \hbar}(x,y)$. How does its Wigner
distribution behave along $\Sigma_0$?\bigskip

\item Interfaces arise in the quantum Hall effect, a point process defined
by a weight $\phi$ and a Laughlin state which gives probabilities of $N$ electrons to occur in a given configuration. The Laughlin states concentrates
as $N \to \infty$ inside a `droplet'. The interface asymptotics across the droplet in dimension one have been studied in \cite{CFTW,Wieg} and others and from a mathematical point of view by Hedenmalm and Wennmann \cite{HW17,HW18}. In the next section, we discuss
higher dimensional droplets. \bigskip

\item Interfaces are studied for nonlinear equations such as the Allen-Cahn
equation, and are related to phase transition problems; see e.g. \cite{GG18} for references to the literature.
\end{itemize}

\subsection{Droplets in phase space}Let us describe droplets in more detail. 
Droplets in phase space arise as coincidence sets  in envelope problems for plurisubharmonic functions. The boundary
of such coincidence sets is the interface. In special cases, it is the same interface that we have described for spectral
interfaces. But in general, the interface is a free boundary that must be determined from the envelope, and even its
regularity is a problem. We refer to \cite{Ber1} for the origins of the theory of dimensions $> 1$.

The definition involves the inner products ${\rm Hilb}_N(h, \nu)$
induced by the data $(h, \nu)$    on the spaces  $H^0(M,
L^N)$ of holomorphic sections of powers $L^N \to M$  by
\begin{equation} \label{HILB} ||s||^2_{{\rm Hilb}_N(h, \nu)} : = \int_M |s(z)|^2_{h^N} d\nu(z). \end{equation}
We let  $h $ be a general $C^{2}$ Hermitian metric on $L$, and
  denote its positivity set by  \begin{equation} \label{M0} M(0) = \{x \in M:  \omega_{\phi}  |_{T_x M}\; \rm{has \; only\; positive\; eigenvalues} \}, \end{equation}
i.e. the set where $\omega_{\phi}$  is a positive $(1,1)$ form. 
For a compact set $K \subset M$, also define  the  {\it equilibrium
potential }  $\phi_{eq} = V^*_{h, K}$ \footnote{Both notations $\phi_{eq} $ and $ V^*_{h, K}$, and also $P_K(\phi)$, are standard and we use them interchangeably. $V^*_{h,K}$ is called the pluri-complex Green's function.}
\begin{equation}\label{EQPOT} V_{h,
K}^* (z)= \phi_{eq}(z):= \sup \{u(z): u \in PSH(M, \omega_0), 
u \leq  \phi  \; \mbox{on}\;
 K\},
\end{equation}
where $\omega_0$ is a reference \kahler metric on $M$ and $ PSH(M, \omega_0)$ are the psh functions $u$ relative to $\omega_0$,\begin{equation} \label{PSHM} PSH(M, \omega_0) = \{u \in L^1(M, \R \cup \infty): dd^c u + \omega_0  \geq 0,\;\; {\rm and}\; u \; {\rm is} \; \omega_0-u.s.c.\}. \end{equation} Further define the coincidence set,
\begin{equation} \label{D} D : = \{z \in M: \phi(z) = \phi_e(z)\}. \end{equation} The boundary $\partial D$ is the `interface' and
the problem is to determine its regularity and other properties. It carries an  {\it equilibrium measure} defined  by 
\begin{equation}\label{EQMDEF} d\mu_{\phi} = (dd^c \phi_{eq})^m/m! =  {\bf 1}_{D \cap M(0)}   (dd^c \phi)^m/m!. \end{equation}
Here, $d^c = \frac{1}{i} (\partial - \dbar)$.

Some droplets are classically forbidden regions for spectrally defined subpaces. The extent to which one may construct
a spectral problem with this property is unknown. Since the interface is usually only $C^{1,1}$, it cannot be the  level set (even
a critical level) for a smooth (Morse-Bott) Hamiltonian in general.

\section{\label{BACKGROUND} Appendix on \kahler analysis}

In this Appendix, we  give a quick review of the basic notations of \kahler analysis.
 First we introduce co-circle bundle $X \subset L^*$ for a positive Hermitian line bundle $(L,h)$, so that holomorphic sections of $L^k$ for different $k$ can all be represented in the same space of CR-holomorphic functions on $X$, $\hcal(X) = \oplus_k \hcal_k(X)$. The Hamiltonian flow $g^t$ generated by $\xi_H$ on $(M,\omega)$ lifts to a contact flow $\h g^t$ generated by $\h \xi_H$ on $X$. 
\subsection{\label{CRSECT}Holomorphic sections in $L^k$ and CR-holomorphic functions on $X$}
Let $(L,h) \to (M, \omega)$ be a positive Hermitian line bundle, $L^*$ the dual line bundle. Let 
\[ X := \{ p \in L^* \mid \|p\|_h = 1\}, \quad \pi: X \to M \]
be the unit circle bundle over $M$. 

Let $e_L \in \Gamma(U, L)$ be a non-vanishing holomorphic section of $L$ over $U$, $\varphi = -\log \|e_L\|^2$ and $\omega = i \ddbar\varphi$. We also have the following trivialization of $X$:
\be \label{X-triv} U \times S^1 \cong X|_U, (z; \theta) \mapsto e^{i\theta} \frac{e_L^*|_z}{\|e_L^*|_z\|}. \ee

$X$ has a structure of a contact manifold. Let $\rho$ be a smooth function in a neighborhood of $X$ in $L^*$, such that $\rho>0$ in the open unit disk bundle, $\rho|_X=0$ and $d\rho|_X\neq 0$. Then we have a contact one-form on $X$
\begin{equation} \label{alphadef} \alpha =- \Re(i\dbar\rho)|_X, \end{equation}
well defined up to multiplication by a positive smooth function. We fix a choice of $\rho$ by
\[ \rho(x) = - \log \|x\|_h^2, \quad x \in L^*, \]
then in local trivialization of $X$ \eqref{X-triv}, we have
\be \alpha =   d \theta   - \frac{1}{2}  d^c \varphi(z). \label{alpha-def} \ee

$X$ is also a strictly pseudoconvex CR manifold. The {\it  CR structure} on $X$ is defined as
follows:
 The kernel of $\alpha$ defines a horizontal hyperplane bundle \begin{equation} \label{HDEF} HX :=
\ker \alpha \subset TX, \end{equation}  
invariant under $J$ since $\ker \alpha = \ker d \rho \cap \ker d^c\rho$. Thus we have a splitting
\[ TX \ot \C \cong H^{1,0} X \oplus H^{0,1}X \oplus \C R.\]
A function $f: X \to \C$ is CR-holomorphic, if $df|_{H^{0,1}X} = 0$.

A holomorphic section $s_k$ of $L^k$ determines a CR-function $\h s_k$ on $X$ by
\[ \h s_k(x) := \la x^{\ot k}, s_k\ra, \quad x \in X \subset L^*. \]
Furthermore $\h s_k$ is of degree $k$ under the canonical $S^1$ action $r_\theta$ on $X$, $\h s_k(r_\theta x) = e^{i k \theta} \h s_k(x)$. The inner product on $L^2(M,L^k)$ is given by
\[ \la s_1, s_2 \ra := \int_M h^k(s_1(z), s_2(z)) d \Vol_M(z), \quad d \Vol_M = \frac{\omega^m}{m!}, \]
and inner product on $L^2(X)$ is given by
\[ \la f_1, f_2 \ra := \int_X f_1(x) \wb{f_2(x)} d \Vol_X(x), \quad d \Vol_X =  \frac{\alpha}{2\pi}\wedge\frac{(d\alpha)^m}{m!}. \]
Thus, sending $s_k \mapsto \h s_k$ is an isometry.

\subsection{\Szego kernel on $X$}
On the circle bundle $X$ over $M$, we define the orthogonal projection from $L^2(X)$ to the CR-holomorphic subspace $\hcal (X) = \h \oplus_{k \geq 0} \hcal_k(X)$, and degree-$k$ subspace $\hcal_k(X)$: 
\[ \h \Pi: L^2(X) \to \hcal(X), \quad \h \Pi_k: L^2(X) \to \hcal_k(X), \quad \hPi = \sum_{k \geq 0} \hPi_k. \]
The Schwarz kernels $\hPi_k(x,y)$ of $\hPi_k$ is called the degree-$k$ \Szego kernel, i.e. 
\[ (\hPi_k F)(x) = \int_X \hPi_k(x,y) F(y) d \Vol_X(y), \quad \forall F \in L^2(X). \]
If we have an orthonormal basis $\{\h s_{k,j}\}_j$ of $\hcal_k(X)$, then
\[ \hPi_k(x,y) = \sum_j \h s_{k,j}(x) \wb{ \h s_{k,j}(y)}. \] 

The degree-$k$ kernel can be extracted as the Fourier coefficient of $\hPi(x,y)$
\begin{equation} \label{Pikdef}  \hPi_k(x,y) = \frac{1}{2\pi} \int_0^{2\pi} \hPi(r_\theta x, y) e^{-i k \theta} d \theta. \end{equation}
We refer to \eqref{Pikdef} as the {\it semi-classical Bergman kernels}.

\subsection{Boutet de Monvel-Sj\"ostrand parametrix for the \Szego kernel}

Near the diagonal in $X \times X$, there exists a parametrix due to  Boutet de Monvel-Sj\"ostrand 
\cite{BSj} for the \Szego kernel of the form,  
\begin{equation} \label{SZEGOPIintroa}  
\hat{\Pi}(x,y) =  \int_{\R^+} e^{\sigma \h \psi(x,y)} s(x, y ,\sigma) d \sigma  + \hat{R}(x,y). 
\end{equation} 
where $\h \psi(x,y)$ is the almost-CR-analytic extension of $\h \psi(x,x)=-\rho(x) = \log \|x\|^2$, and $s(x,y,\sigma) = \sigma^m s_m(x,y) + \sigma^{m-1} s_{m-1}(x,y) + \cdots$ has a complete asymptotic expansion.  
In local trivialization \eqref{X-triv}, 
\[ \h \psi(x,y) = i (\theta_x - \theta_y) + \psi(z, w) - \half \varphi(z) - \half \varphi(w),\]
where $\psi(z,w)$ is the almost analytic extension of $\varphi(z)$.

\subsection{Lifting the Hamiltonian flow to a contact flow on $X_h$.}\label{LIFT} 
In this seection we review the definition of the  lifting of a Hamiltonian flow to a contact flow, following \cite[Section 3.1]{ZZ17}. 
Let $H: M \to \R$ be a Hamiltonian function on $(M, \omega)$. Let $\xi_H$ be the Hamiltonian vector field associated to $H$, such that $dH = \iota_{\xi_H}\omega$. 
The purpose of this section is to  lift $\xi_H$ to a contact vector field $\hat{\xi}_H$ on $X$. Let $\alpha$ denote the contact 1-form \eqref{alpha-def} on $X$, and $R$ the corresponding Reeb vector field determined by $\la \alpha, R \ra =1$ and $\iota_{R} d\alpha=0$. One can check that $R=\pa_\theta$. 

\begin{defin}
(1) The horizontal lift of $\xi_H$ is a vector field on $X$ denoted by  ${\xi}_H^h$. It is determined by 
\[ \pi_*{\xi}_H^h = \xi_H, \quad \la \alpha, \xi_H^h \ra = 0. \]
(2) The contact lift of $\xi_H$ is a vector field on $X$ denoted by  $\h {\xi}_H$. It is determined by
\[ \pi_*\h {\xi}_H = \xi_H, \quad \lcal_{\h \xi_H} \alpha = 0. \]
\end{defin}

\bl \label{xiHLEM} 
The contact lift $\h \xi_H$ is given by
\[ \hat{\xi}_H = {\xi}_H^h - H R. \]
\el

The Hamiltonian flow on $M$ generated by $\xi_H$ is denoted by $g^t$
\[ g^t: M\to M, \quad g^t = \exp(t \xi_H). \]
The contact flow on $X$ generated by $\h \xi_H$ is denoted by $\h g^t$
\[\h g^t: X \to X, \quad \h g^t = \exp(t \h \xi_H). \]

\begin{lem} \label{gtform} In  local trivialization \eqref{X-triv}, we have a useful formula for the flow, 
	  $\hat{g}^t$ has the form (see \cite[Lemma 3.2]{ZZ17}):
\[ \hat{g}^t(z, \theta) = (g^t(z), \;\; \theta 
	+   \int_0^t \half \la d^c \varphi, \xi_H \ra(g^s(z))ds  - t H(z)).  \]
\end{lem}

Since $\hat{g}^t$ preserves $\alpha$ it preserves the horizontal distribution $H(X_h) = \ker \alpha$, i.e.
\begin{equation} \label{HSPLIT} D \hat{g}^t: H(X)_x \to H(X)_{\hat{g}^t(x)}. \end{equation} It also preserves the vertical
(fiber) direction and therefore preserves the splitting $V \oplus H$ of $T X$. Its action in the vertical direction is determined by 
Lemma \ref{gtform}.  When $g^t$ is non-holomorphic, $\hat{g}^t$ is not CR holomorphic, i.e. does not preserve the horizontal complex structure $J$ or the
splitting of $H(X) \otimes \C$ into its $\pm i $ eigenspaces.

\section{Appendix}

\subsection{\label{AIRYAPP} Appendix on the Airy function}
The Airy function is defined by,

$$Ai(z) = \frac{1}{2 \pi i} \int_L e^{v^3/3 - z v} dv, $$
where $L$ is any contour that beings at a point at infinity in the sector $- \pi/2 \leq \arg (v) \leq - \pi/6$ and ends
at infinity in the sector $\pi/6 \leq \arg(v) \leq \pi/2$.  In the region $|\arg z| \leq (1 - \delta) \pi$
in $\C - \{\R_-\}$ write $v = z^{\half} + i t ^{\half}$ on the upper half of L and $v = z^{\half} - i t^{\half}$
in the lower half. Then
\begin{equation} \label{AIRYASYM} \Ai(z) = \Psi(z) e^{- \frac{2}{3} z^{3/2}}, \;
\mathrm{
with}\;
\Psi(z) \sim z^{-1/4} \sum_{j = 0}^{\infty} a_j z^{- 3j/2}, \;\; a_0 = \frac{1}{4} \pi^{-3/2}. \end{equation}

\subsection{Appendix on Laguerre functions \label{S:Laguerre}}

The Laguerre polynomials $L_k^{\alpha}(x)$  of degree $k$ and of type $\alpha$ on $[0, \infty)$  are defined by
\begin{equation}
e^{-x} x^{\alpha} L_k^{\alpha}(x) = \frac{1}{k!} \frac{d^k}{dx^k} (e^{-x}x^{k +\alpha}). \label{E:LagDef}
\end{equation}
They are solutions of the Laguerre equation(s),
$$x y'' + (\alpha + 1 -x) y(x)' + k y(x) = 0.$$

For fixed $\alpha$ they are orthogonal polyomials of $L^2(\R_+, e^{-x} x^{\alpha} dx) $. 
An othonormal basis is given by
$$\lcal_k^{\alpha}(x) = \left(\frac{\Gamma(k +1)}{\Gamma(k + \alpha+ 1)}\right)^{\half} L_k^{\alpha}(x).$$
We will have occasion to use the following generating function:
\[\sum_{k =0}^{\infty}  L_k^{\alpha}(x)  w^k = (1 -w)^{-\alpha -1} e^{- \frac{w}{1-w} x}\]

\noindent The most useful integral representation for the Laguerre functions is
\begin{equation}\label{INTnew}
e^{-x/2} L_n^{(\alpha)} (x) = \lr{-1}^n\oint \frac{e^{-\frac{x}{2}\cdot \frac{1-z}{1+z}}}{z^n\lr{1+z}^{\alpha+1}}\frac{dz}{2 \pi i z} ,
\end{equation}
where the contour encircles the origin once counterclockwise. Equivalently,
\begin{equation}\label{INT} e^{-x/2} L_n^{(\alpha)} (x) = \frac{(-1)^n}{2^{\alpha}}
\frac{1}{2 \pi i} \int^{1+} e^{- x z/2} \left( \frac{1 + z}{1 - z} \right)^{\nu/4} (1 - z^2)^{
\frac{\alpha-1}{2}} d z \end{equation}
where $\nu= 4n + \alpha + 2$ and the contour encircles $z = 1$ in the positive
direction and closes at $\Re z = \infty, |\Im z| = \mathrm{constant}$. In (5.9) of \cite{FW} the Laguerre functions are represented as the oscillatory integrals,
\begin{equation} \label{nu2}e^{- \nu t/2} L_n^{\alpha}(\nu t) = \frac{(-1)^n}{2^{\alpha}} \frac{1}{2 \pi i}
\int_{\lcal} [1 - z^2(u)]^{\frac{\alpha-1}{2}} \exp \{ \nu \left( \frac{u^3}{3} - B^2(t) u \right) \} du, \end{equation}
where $\nu = 4 n + 2 \alpha + 2$ and $B(t) $ is defined in (5.5) of \cite{FW} and and  $\lcal$ is a branch of the hyperbolic curve in the right half plane.


\begin{thebibliography}{HHHH}



\bibitem[HZZ15]{HZZ15}  Boris Hanin, Steve Zelditch, Peng Zhou  Nodal Sets of Random Eigenfunctions for the Isotropic Harmonic Oscillator, International Mathematics Research Notices, Vol. 2015, No. 13, pp. 4813-��4839,  (2015)  (arXiv:1310.4532)

\bibitem[HZZ16]{HZZ16}  Boris Hanin, Steve Zelditch and Peng Zhou, 
   Scaling of harmonic oscillator eigenfunctions and their nodal sets around the caustic. Comm. Math. Phys. 350 (2017), no. 3, 1147-1183 (arXiv:1602.06848).
   
      \bibitem[HZ19]{HZ19} B. Hanin and S. Zelditch,  
Interface Asymptotics of Eigenspace Wigner distributions for the Harmonic Oscillator, arXiv:1901.06438.

 \bibitem[HZ19b]{HZ19b} B. Hanin and S. Zelditch,  Interface Asymptotics of Wigner-Weyl Distributions for the Harmonic Oscillator,
arXiv:1903.12524.


\bibitem[ZZ16]{ZZ16} S. Zelditch and P. Zhou, Interface asymptotics of partial  Bergman kernels on $S^1$-symmetric Kaehler manifolds, to appear in 
J. Symp. Geom.  (arXiv:1604.06655).

\bibitem[ZZ17]{ZZ17}  S. Zelditch and P. Zhou, 
    Central Limit theorem for spectral Partial Bergman kernels, to appear
    in Geom. Topl.  arXiv:1708.09267.


\bibitem[ZZ18]{ZZ18}   S. Zelditch and P. Zhou,   Interface asymptotics of Partial Bergman kernels around a critical level (arXiv:1805.01804).

    
\bibitem[ZZ18b]{ZZ18b} S. Zelditch and P. Zhou, Pointwise Weyl law for Partial  Bergman kernels,  {\it  Algebraic and Analytic Microlocal Analysis} pp. 589- 634.  M. Hitrik, D. Tamarkin, B. Tsygan, S. Zelditch (eds). Springer Proceedings in Mathematics and Statistics, Springer-Verlag (2018).


\end{thebibliography}

\begin{thebibliography}{HHHH}


\bibitem[AT]{AT} R. J. Adler and J. E.  Taylor, {\it Random fields and
    geometry}.
Springer Monographs in Mathematics. Springer, New York, 2007.






\bibitem[Ag]{Ag} S. Agmon, {\it Lectures on exponential decay of solutions of second-order elliptic equations: bounds on eigenfunctions of N-body Schrödinger operators.} Mathematical Notes, 29. Princeton University Press, Princeton, NJ; University of Tokyo Press, Tokyo, 1982. 

\bibitem[A]{A} N. Aronszajn,
Theory of reproducing kernels.
Trans. Amer. Math. Soc. 68, (1950). 337-404. 

\bibitem[AW]{AW} J. M.~Azais and M.~Wsebor, Level Sets and Extrema of Gaussian Fields. Wiley and Sons, Inc., Hoboken, New Jersey. 2009.

\bibitem[BPD]{BPD} N. L. Balazs, H. C. Pauli and O. B. Dabbousi, Tables of Weyl Fractional Integrals for the Airy Function, 
Mathematics of Computation
Vol. 33, No. 145 (Jan., 1979), pp. 353-358+s1-s9.

\bibitem[Ber1]{Ber1}  R. Berman, Bergman kernels and equilibrium measures for line bundles over projective manifolds. Amer. J. Math. 131 (2009), no. 5, 1485–1524. 

\bibitem[Be]{Be}  M.V. Berry, Semi-classical mechanics in phase space: a study of Wigner's function. Philos. Trans. Roy. Soc. London Ser. A 287 (1977), no. 1343, 237-271.

\bibitem[BH]{BH} W.E.  Bies and E. J. Heller,  Nodal structure of chaotic
    eigenfunctions. J. Phys. A 35 (2002),
    no. 27, 5673-5685.

    
    \bibitem[BSZ]{BSZ}  P. Bleher, B. Shiffman, and S. Zelditch,  Universality and scaling of correlations between zeros on complex manifolds. Invent. Math. 142 (2000), no. 2, 351-395.
    
    

\bibitem[BG81]{BG81} L.  Boutet de Monvel and V.  Guillemin, {\it The SpectralTheory of Toeplitz Operators}, Ann.\ Math.\ Studies 99, Princeton Univ.\
Press, Princeton, 1981.

\bibitem[BSj]{BSj} L. Boutet de Monvel and J. Sj\"ostrand, \textit{Sur la
singularit\'e des noyaux de Bergman et de Szeg\"o}, Asterisque \textbf{34}--\textbf{35} (1976), 123--164.
    
\bibitem[B]{B} Bustamante, Jorge; García-Muñoz, Miguel A.; Quesada, José M. Bernstein polynomial and discontinuous functions. J. Math. Anal. Appl. 411 (2014), no. 2, 829–837.
\bibitem[CG69I]{CG69I} K.E. Cahill and R.J. Glauber, Ordered expansions in boson amplitude operators Phys. Rev. 177, 1857-1881 (1969).

\bibitem[CG69II]{CG69II} K.E. Cahill and R.J. Glauber, Density operators and quasiprobability distributions,  Phys. Rev. 177 (5) (1969):
1882-1902

 \bibitem[CFTW]{CFTW} T. Can, P. J. Forrester, G. Tellez and P. Wiegmann, 
Singular Behavior At The Edge of Laughlin States
Phys. Rev. B 89, 235137 (2014) ( arXiv:1307.3334).

    \bibitem[CH]{CH} Y. Canzani and B. Hanin, Scaling limit for the kernel of the spectral projector and remainder estimates in the pointwise Weyl law. Anal. PDE 8 (2015), no. 7, 1707-1731.

\bibitem[CT]{CT} Y.  Canzani and  J. A. Toth, Nodal sets of Schroedinger eigenfunctions in forbidden regions.  Ann. Henri Poincare 17 (2016), no. 11, 3063-3087 (arXiv:1502.00732).


\bibitem[Ch03]{Ch03} L. Charles,
Berezin-Toeplitz operators, a semi-classical approach.
Comm. Math. Phys. 239 (2003), no. 1-2, 1-28. 






\bibitem[ChE18]{ChE18}  L. Charles and  B. Estienne, 
Entanglement entropy and Berezin-Toeplitz operators 
arXiv:1803.03149.



\bibitem[Ch]{Ch}  J.  Chazarain,  Spectre d'un hamiltonien quantique et
    m\'ecanique classique.  Comm. Partial Differential Equations 5 (1980),
    no. 6, 595-644.



\bibitem[Ch]{Ch} M.I. Chlodovsky, Sur la représentation des fonctions discontinues par les polynômes de M.S. Bernstein, Fund. Math. 13 (1929) 62–72.


\bibitem[Dau80]{D80}  I. Daubechies,  Coherent states and projective representation of the linear canonical transformations. J. Math. Phys. 21 (1980), no. 6, 1377-1389.

\bibitem[DSj]{DSj} M. Dimassi and J. Sj\"ostrand, {\it Spectral asymptotics
    in the semi-classical limit.}
 London Mathematical Society Lecture Note Series, 268. Cambridge University
 Press, Cambridge, 1999.



\bibitem[DF]{DF} H. Donnelly and C. Fefferman, Nodal sets of eigenfunctions
    on
Riemannian manifolds, Invent. Math. 93 (1988), 161-183.

\bibitem[F]{F}  G. B. Folland,
{\it Harmonic analysis in phase space}. Annals of Mathematics
Studies, 122. Princeton University Press, Princeton, NJ, 1989.


\bibitem[F]{F} G. Folland, Harmonic Analysis in Phase Space, Ann. of Math. Stud., vol. 122, Princeton University Press, 1989.

\bibitem[FW]{FW}   C.L. Frenzen and R.  Wong, 
Uniform asymptotic expansions of Laguerre polynomials.
SIAM J. Math. Anal. 19 (1988), no. 5, 1232–1248. 

\bibitem[GG18]{GG18} P. Gaspar and M.  Guaraco, 
The Allen-Cahn equation on closed manifolds.
Calc. Var. Partial Differential Equations 57 (2018), no. 4,

   
\bibitem[GrSj]{GrSj}  A. Grigis and J. Sj\"ostrand, {\it Microlocal analysis
    for differential operators.} An
    introduction. London Mathematical Society Lecture Note Series, 196.
    Cambridge University Press, Cambridge,
    1994.


\bibitem[GU12]{GU12}  V. Guillemin,  A. Uribe,  and Z.  Wang, Band invariants for perturbations of the harmonic oscillator. J. Funct. Anal. 263 (2012), no. 5, 1435-1467.

\bibitem[GUW]{GUW} V. Guillemin, A.  Uribe, and Z. Wang, Canonical forms for perturbations of the harmonic oscillator. New York J. Math. 21 (2015), 163-180.



   \bibitem[HZ20]{HZ20} B. Hanin and S. Zelditch, Universality of Schrodinger  scaling asymptotics around the caustic (in preparation).
   
      \bibitem[HW17]{HW17} H. Hedenmalm and A. Wennman, Planar orthogonal polynomials and boundary universality in the random normal matrix model, arXiv 1710.06493.
   
   \bibitem[HW18]{HW18} H. Hedenmalm and A. Wennman, Off-spectral
   analysis of Bergman kernels, arXiv 1805.00854. 
   
   

\bibitem[HH]{HH} F. Herzog, J.D. Hill, The Bernstein polynomials for discontinuous functions, Amer. J. Math. 68 (1946) 109–124.




\bibitem[HS]{HS}  P.D. Hislop and I. M.  Sigal, {\it Introduction to spectral theory. With applications to Schrödinger operators. } Applied Mathematical Sciences, 113. Springer-Verlag, New York, 1996.

\bibitem[HSj16]{HSj16} M. Hitrik and J. Sjoestrand, Two Minicourses on Analytic Microlocal Analysis, to appear
in {\it Algebraic and Analytic Microlocal Analysis}, M Hitrik, D. Tamarkin, B. Tysgan and S. Zelditch (eds.).


\bibitem[Hor]{Hor}  H\"ormander, Lars {\it The analysis of linear partial differential operators. I. Distribution theory and Fourier analysis. } Classics in Mathematics. Springer-Verlag, Berlin, 2003. 


\bibitem[Hor]{Hor}  H\"ormander, Lars {\it The analysis of linear partial differential operators. III.   } Classics in Mathematics. Springer-Verlag, Berlin, 2003. 














\bibitem[IRT]{IRT} R. Imekraz, D. Robert and L.  Thomann,
On random Hermite series,  Trans. Amer. Math. Soc. 368 (2016), no. 4, 2763-2792.  (arXiv:1403.4913).

\bibitem[JZ]{JZ} A. J. E. M. Janssen and S. Zelditch, 
Szeg\"o limit theorems for the harmonic oscillator. 
Trans. Amer. Math. Soc. 280 (1983), no. 2, 563-��587.


\bibitem[Jin]{Jin} L. Jin, Semiclassical Cauchy estimates and applications,
to appear in Trans. AMS (arXiv:1302.5363).


\bibitem[KT]{KT} H.  Koch and D. Tataru, Daniel Lp eigenfunction bounds for the Hermite operator. Duke Math. J. 128 (2005), no. 2, 369-392.


\bibitem[Lev]{Lev} B. Levikson, 
On the behavior of a certain class of approximation operators for discontinuous functions. 
Acta Math. Acad. Sci. Hungar. 33 (1979), no. 3-4, 299–306. 


\bibitem[Mir]{Mir} Mirakyan, G.
Approximation des fonctions continues au moyen de polynomes de la forme $e^{-nx} \sum_{k =0}^m C_{k,n} x^j$
C. R. (Doklady) Acad. Sci. URSS (N.S.) 31, (1941). 201-205.


\bibitem[L]{L}  G.G. Lorentz, . {\it Bernstein polynomials}. Second edition. Chelsea Publishing Co., New York, 1986.


\bibitem[NS]{NS} F. Nazarov and M. Sodin,
On the number of nodal domains of random spherical harmonics.
Amer. J. Math. 131 (2009), no. 5, 1337-1357.



\bibitem[O]{O} F. W. J. Olver, {\it Asymptotics and special functions. } Academic Press, New York.


\bibitem[O]{O} Omey, E.
Note on operators of  Szasz-Mirakyan type. 
J. Approx. Theory 47 (1986), no. 3, 246-254. 

\bibitem[PT97]{PT97} 
M.A. Pinsky, M. Taylor, Pointwise Fourier inversion: A wave equation approach, J. Fourier Anal. Appl. 3 (6) (1997) 647-703. MR1481629 


\bibitem[PS]{PS} F. Pokorny and M.  Singer, 
Toric partial density functions and stability of toric varieties. 
Math. Ann. 358 (2014), no. 3-4, 879-923. 



\bibitem[Res]{Res}  S. I. Resnick, {\it A probability path}.  Modern Birkhauser Classics. Birkauser/Springer, New York, 2014. 

\bibitem[R87]{R87} D. Robert,{\it
Autour de l'approximation semi-classique. }
Progress in Mathematics, 68. Birkhäuser Boston, Inc., Boston, MA, 1987.
















\bibitem[PS]{PS} F. Pokorny and M.  Singer, 
Toric partial density functions and stability of toric varieties. 
Math. Ann. 358 (2014), no. 3-4, 879-923. 


\bibitem[R87]{R87} D. Robert,{\it
Autour de l'approximation semi-classique. }
Progress in Mathematics, 68. Birkhäuser Boston, Inc., Boston, MA, 1987.





\bibitem[ShZ]{ShZ} B Shiffman and S.  Zelditch, Random polynomials with prescribed Newton polytope. J. Am. Math. Soc. 17(1), 49-108.


\bibitem[ShZ02]{ShZ02} Shiffman, Bernard; Zelditch, Steve Asymptotics of almost holomorphic sections of ample line bundles on symplectic manifolds. J. Reine Angew. Math. 544 (2002), 181-222.

\bibitem[RS]{RS}  J.Ross and M. Singer, 
Asymptotics of Partial Density Functions for Divisors,  J. Geom. Anal. (to appear) (arXiv:1312.1145).

\bibitem[RZ]{RZ}  Y. A. Rubinstein and S. Zelditch,  The Cauchy problem for the homogeneous Monge-Ampère equation, I. Toeplitz quantization. J. Differential Geom. 90 (2012), no. 2, 303-327.


\bibitem[ShZ]{ShZ} B Shiffman and S.  Zelditch, Random polynomials with prescribed Newton polytope. J. Am. Math. Soc. 17(1), 49-108.



\bibitem[Sz50]{Sz50}  O. Szasz, 
Generalization of S. Bernstein's polynomials to the infinite interval. 
J. Research Nat. Bur. Standards 45, (1950). 239-245. 




\bibitem[T]{T}  S. Thangavelu,
Lectures on Hermite and Laguerre expansions. 
With a preface by Robert S. Strichartz. Mathematical Notes, 42. Princeton University Press, Princeton, NJ, 1993.

\bibitem[T12]{T12} S. Thangavelu, Hermite and Laguerre semigroups: some recent developments. Orthogonal families and semigroups in analysis and probability, 251-284, Semin. Congr., 25, Soc. Math. France, Paris, 2012





\bibitem[TW94]{TW94} C.A.  Tracy and H.  Widom, Level-spacing distributions and the Airy kernel. Comm. Math. Phys. 159 (1994), no. 1, 151-174. 




\bibitem[W]{W} X-G. Wen, {\it Quantum Field Theory of Many-Body
Systems}, Oxford Grad Texts (2004).



\bibitem[Wieg]{Wieg} P. Wiegmann, Nonlinear hydrodynamics and fractionally quantized solitons at the fractional quantum Hall edge. Phys. Rev. Lett. 108, 206810 (2012)

\bibitem[W32]{W32} E.P. Wigner, "On the quantum correction for thermodynamic equilibrium", Phys. Rev. 40 (June 1932) 749-759.




\bibitem[XX2]{XX2} Linsen Xie and Tingfan Xie, 
Approximation theorems for localized Szasz-Mirakjan operators. J. Approx. Theory 152 (2008), no. 2, 125-134. 



\bibitem[Z97]{Z97} S. Zelditch, Index and dynamics of quantized contact transformations,
Ann. Inst. Fourier 47 (1997), 305--363, MR1437187, Zbl 0865.47018.



\bibitem[Ze]{Ze} S. Zelditch,  Bernstein polynomials, Bergman kernels and toric Kähler varieties. J. Symplectic Geom. 7, 1-26 (2009).
 

\bibitem[Zw]{Zw} M. Zworski, {\it  Semiclassical analysis}. Graduate Studies
    in Mathematics, 138. American
    Mathematical Society, Providence, RI (2012).   MR2952218


\end{thebibliography}
\end{document}